\documentclass[leqno,twoside]{report}

\usepackage{graphicx,color}
\usepackage{mathrsfs}
\usepackage[latin2]{inputenc}
\usepackage[magyar,english]{babel}
\usepackage{amsmath,amssymb,amscd}
\usepackage{array,tabularx}
\usepackage{pb-diagram}
\usepackage{euscript}
\usepackage{indentfirst}

\makeatletter
\def\fvecfill{$\m@th\hbox{\raisebox{-5.75pt}[1.5pt][0pt]{$\!\mathord\ulcorner$}}
\mkern-5mu
\cleaders\hbox{\raisebox{-1.5pt}[1.5pt][0pt]{$\!\mathord-$}}\hfill
\mkern-5mu
\mathord{\hbox{\raisebox{-5.75pt}[1.5pt][0pt]{$\!\mathord\urcorner\!$}}}$}
\makeatother
\newcommand{\circc}[1]{\overset{\raisebox{-2.5pt}[1pt][0pt]{\tiny{$\circ$}}}{#1}}

\newcommand{\vl}[1]{#1^{\mathsf{v}}}
\newcommand{\hl}[1]{#1^{\mathsf{h}}}
\newcommand{\hlfel}[1]{#1^{\overline{\mathsf{h}}}}
\newcommand{\cl}[1]{#1^{\mathsf{c}}}
\newcommand{\modx}[1]{\mathfrak{X}(#1)}
\newcommand{\modc}[1]{C^\infty(#1)}

\newcommand{\szel}[1]{\mathrm{Sec}(#1)}

\newcommand{\tenz}[3]{\EuScript{T}^{#1}_{#2}(#3)}

\def\tauk { \circc{\tau} }
\def\pik { \circc{\pi} }
\def\Tk { \circc{T} }
\def\ii { \mathbf{i} }
\def\jj { \mathbf{j} }
\def\JJ { \mathbf{J} }
\def\Ffel { \overline{F} }
\def\HH { \mathcal{H} }
\def\HHfel { \overline{\mathcal{H}} }
\def\UU { \mathcal{U} }
\def\VV { \mathcal{V} }
\def\hh { \mathbf{h} }
\def\vv { \mathbf{v} }
\def\FF { \mathbf{F} }
\def\tt { \mathbf{t} }
\def\TT { \mathbf{T} }
\def\Pp { \mathbf{P} }
\def\projp { \mathbf{p} }

\def\RR { \mathbf{R} }
\def\DD { \mathbf{D} }
\def\BB { \mathbf{B} }
\def\AA { \mathbf{A} }
\def\Hh { \mathbf{H} }
\def\KK { \mathbf{K} }
\def\LL { \EuScript{L} }
\def\CC { \EuScript{C} }

\def\nav { \vl{\nabla} }
\def\nah { \hl{\nabla} }
\def\hul { \widetilde }
\def\kal { \widehat }
\newcommand{\pp}[2]{\frac{\partial #1}{\partial #2}} 
\newcommand{\PP}[3]{\frac{\partial^2 #1}{\partial #2 \partial #3}}

\newcommand{\valpp}[2]{\frac{\mathrm{d}#1}{\mathrm{d}#2}}

\def\termN { \mathbb{N} }

\def\valR { \mathbb{R} }

\newtheorem{lemma}{Lemma}[chapter]
\newtheorem{coro}[lemma]{Corollary}
\newtheorem{theorem}[lemma]{Theorem}  
\newtheorem{propo}[lemma]{Proposition}

\newenvironment{proof}{\textit{Proof.}}{\mbox{}\hfill\tiny$\square$\normalsize \bigskip}

\author{Zoltán Szilasi}
\title{On the projective theory of sprays with applications to Finsler geometry}
\date{}

\begin{document}

\maketitle

\tableofcontents

\newpage

\begin{center}

\Large \textsl{ACKNOWLEDGEMENT} \\

\bigskip

\normalsize

\textsl{The author is indebted to Sándor Bácsó, his supervisor,\\
and József Szilasi for their inspiration, encouragement and continuous support.}

\end{center}

%\maketitle
%\pagestyle{headings}
\pagenumbering{roman}
\setcounter{page}{1}

\pagestyle{headings}

\chapter*{Introduction}
%\chaptermark{Bevezetés}
\addcontentsline{toc}{chapter}{Introduction}

\textbf{The origins.} The basic ideas and structures of `modern differential geometry' first appeared in Bernard Riemann's habilitation lecture \textit{``Über die Hypothesen die der Geometrie zu Grunde liegen"} (\textit{``On the hypotheses which lie at the foundation of geometry"}), presented in the summer of 1854. Without possessing a perfected terminology, Riemann introduced what we would call today a topological manifold. He also used requirements of differentiability, so he dealt actually with differentiable manifolds. This structure was called by him ``mehrfach ausgedehnte Mannigfaltigkeit" (``multiple extended manifold"). The first problem which he discussed in detail was to equip with a metric structure such a manifold. He proposed measuring `infinitesimals' (in the simplest case tangent vectors) and integrating over a curve to find its length. Riemann pointed out that there is no reason why the length should be assumed to be independent of the position, or more generally, of the position and the direction.

Adopting the old-fashioned notation and terminology for the moment, let $x(t)$, where $t$ is a real parameter, be a curve on a manifold. A metric determination fixes the length of the piece of the curve between $x(t_0)$ and $x(t_1)$. The most general formula proposed by Riemann is $$\int_{t_0}^{t_1}F(x(t),\dot{x}(t))\textrm{dt}.$$ To obtain a reasonable length-concept, the function $F$ must be positive, if\\ $\dot{x}(t)\neq 0$. One also wants the length to be independent of the parametrization and of the direction of traversal of the curve; these lead to the requirement $$F(x,\lambda v)=\left|\lambda\right|F(x,v)$$ for all tangent vector $v$ at $x$ and real number $\lambda$. Finally, it is also reasonable to make a restrictive convexity assumption about $F$. (At the beginning of Chapter \ref{ch6} of our dissertation we give a precise formulation of these requirements as condition (F$_{\textrm{1}}$)-(F$_{\textrm{4}}$) for a Finsler function $F: TM\rightarrow\valR$.) All the requirements are satisfied by the pointwise positive definite forms $$F^2(x,v):=\sum_{i,j}g_{ij}(x)v^iv^j\textrm{ , }g_{ij}=g_{ji}.$$ Riemann restricted himself to this metric determination, but he did not repudiate the more general (not necessarily quadratic) fundamental function $F$. In this connection he said: ``The investigation of this more general species would not call for essentially different principles, but would be considerably time-consuming...". (``Die Untersuchung dieser allgemeiner Gattung würde zwar keine wesentlich andere Principien erfordern, aber ziemlich zeitranbend sein...". See: \textit{Bernhard Riemann's gesammelte mathematische Werke und wissenschaftlicher Nachlass}. Herausgegeben unter Mitwirkung von R. Dedekind und H. Weber. 2. Auflage: Teubner, Leipzig, 1892. Reprint: Dover, New York, 1953.)

The more general metric determinations were first studied by Paul Finsler on the suggestion of C. Carathéodory in his Göttingen dissertation \textit{Über Kurven und Fl\"{a}chen in allgemeinen R\"{a}umen} (1918; Nachdruck Birkh\"{a}user, Basel 1951). It turned out that in such a general space a metric tensor given by $$g_{ij}(x,v):=\frac{1}{2}\PP{F^2}{v^i}{v^j}(x,v)$$ may also be introduced, and - at least in the first steps - ``essentially different principles" are indeed not required. However, all geometric data will depend not only on the points $x$, but also on the tangent vectors $v$. The geometry of such ``general metric spaces" was called \textit{Finsler geometry} by J. H. Taylor in 1927.

Finsler geometry has become a quite extensive and active research area. The greatest impetus to its development is due to the professor of German University in Prague, Ludwig Berwald (1883-1942). His ideas and methods influenced decisively the Debrecen School of Finsler Geometry, represented in its golden period by O. Varga, A. Rapcsák, A. Moór, L. Tamássy, Gy. Soós and J. Merza.

In spite of the different languare and computational technique, the present work is also debt to Ludwig Berwald.\\

\textbf{Straight lines.} Riemann's habilitation lecture began with the following statements:

\textit{``As is well-known, geometry presupposes the concept of space, as well as assuming the basic principles for constructions in space... . The relationship between these presuppositions is left in the dark... ."} (M. Spivak's translation.)

In Finsler geometry the role of ``space" is played by a (smooth) manifold, it provides points for the geometry. The Finsler function makes ``constructions in space" possible: with its help one can define `straight lines', called \textit{geodesics}, which have properties analogous to those of straight lines in Euclidean space. In the Euclidean $n$-space $\valR^n$, a straight line may be defined as either a curve $\alpha: \valR\rightarrow\valR^n$ such that $\alpha''=0$, or a curve which represents the shortest path between points. Now we briefly sketch how the second approach works in Finsler geometry.

Let $M$ denote our base manifold, and let $\tau: TM\rightarrow M$ be its tangent bundle. Suppose, for simplicity, that $M$ admits a global coordinate system $(u^i)_{i=1}^n$. Then $(x^i,y^i)_{i=1}^n$, where $$x^i:=u^i\circ\tau\textrm{ , }y^i(v):=v(u^i)$$ is a global coordinate system for $TM$. Let $F: TM\rightarrow \valR$ be a Finsler function, i.e., a function with the properties mentioned above. Take two points $p_0$, $p_1$ in $M$, and consider the functional $$\EuScript{F}: \gamma\mapsto\EuScript{F}(\gamma):=\int_0^1F\circ\dot{\gamma}=\int_0^1F(\dot{\gamma}(t))\textrm{dt}\in\valR$$ from the set of all (piecewise) smooth curves $\gamma: [0,1]\rightarrow M$ from $p_0$ to $p_1$. The Euler-Lagrange equations of this functional are traditionally written in the form $$\pp{F}{x^i}-\valpp{}{t}\pp{F}{y^i}=0\textrm{ , }i\in\left\{1,\dots,n\right\},$$ and its solutions, i.e., the curves $\gamma: [0,1]\rightarrow M$ satisfying $$\pp{F}{x^i}\circ\dot{\gamma}-\left(\pp{F}{y^i}\circ\dot{\gamma}\right)'=0\textrm{ , }i\in\left\{1,\dots,n\right\}$$ are called the \textit{extremals} of $\EuScript{F}$, or $F$\textit{-extremals} in the calculus of variations, and \textit{geodesics} in Finslerian context. It is a fundamental fact that one can formulate another variational problem, which leads to the same class of geodesics. Consider the \textit{energy function} $E:=\frac{1}{2}F^2$ associated to $F$, and define a new functional $\EuScript{E}$ on the above set of curves by $$\gamma\mapsto\EuScript{E}(\gamma):=\int_0^1E\circ\dot{\gamma}=\int_0^1E(\dot{\gamma}(t))\textrm{dt}.$$ We show that \textit{every }$E$\textit{-extremal} $\gamma$\textit{, with }$\dot{\gamma}(\tau)\neq 0$\textit{ for some }$\tau\in[0,1]$\textit{, satisfies} $$E(\dot{\gamma}(t))=\frac{1}{2}\lambda^2\textrm{ , }t\in[0,1]$$ \textit{for some positive real number} $\lambda$\textit{, and it is an extremal of }$\EuScript{F}$\textit{. Conversely, if }$\gamma$\textit{ is an extremal of }$\EuScript{F}$\textit{ parametrized in such a way that the above relation holds for some positive }$\lambda\in\valR$\textit{, then }$\gamma$\textit{ is also an }$E$\textit{-extremal.}

Suppose that $\gamma$ is an extremal of $\EuScript{E}$, i.e. satisfies $$\pp{E}{x^i}\circ\dot{\gamma}-\left(\pp{E}{y^i}\circ\dot{\gamma}\right)'=0.$$ First we prove that the energy function is constant along the velocity curves $\dot{\gamma}:[0,1]\rightarrow TM$.

For any $t\in[0,1]$, we have
\begin{center}
$\displaystyle(E\circ\dot{\gamma})'(t)=\ddot{\gamma}(t)E=\left((\gamma^{i'}(t)\left(\pp{}{x^i}\right)_{\dot{\gamma}(t)}+\gamma^{i''}(t)\left(\pp{}{y^i}\right)_{\dot{\gamma}(t)}\right)E=\left(\pp{E}{x^i}\circ\dot{\gamma}\right)(t)\gamma^{i'}(t)+\left(\pp{E}{y^i}\circ\dot{\gamma}\right)(t)\gamma^{i''}(t)$.
\end{center}
Since, taking into account the Euler-Lagrange equations of $\EuScript{E}$,
\begin{center}
$\displaystyle\left(\left(\pp{E}{y^i}\circ\dot{\gamma}\right)\gamma^{i'}\right)'(t)=\left(\pp{E}{y^i}\circ\dot{\gamma}\right)'(t)\gamma^{i'}(t)+\left(\pp{E}{y^i}\circ\dot{\gamma}\right)(t)\gamma^{i''}(t)=\left(\pp{E}{x^i}\circ\dot{\gamma}\right)(t)\gamma^{i'}(t)+\left(\pp{E}{y^i}\circ\dot{\gamma}\right)(t)\gamma^{i''}(t)$,
\end{center}
it follows that $$(E\circ\dot{\gamma})'(t)=\left(\left(\pp{E}{y^i}\circ\dot{\gamma}\right)\gamma^{i'}\right)'(t).$$ However, $E$ is positive-homogeneous of degree 2, which implies by Euler's relation $$\left(\left(\pp{E}{y^i}\circ\dot{\gamma}\right)\gamma^{i'}\right)'(t)=\left(\left(\pp{E}{y^i}y^i\right)\circ\dot{\gamma}\right)'(t)=2(E\circ\dot{\gamma})'(t).$$ Thus we find $(E\circ\dot{\gamma})'=2(E\circ\dot{\gamma})'$, therefore $E\circ\dot{\gamma}$ is indeed constant. So if $\dot{\gamma}$ is not identically zero, then there is a positive real number $\lambda$ such that for all $t\in[0,1]$, $$E(\dot{\gamma}(t))=\frac{1}{2}\lambda^2.$$

Now suppose that a curve $\gamma$ satisfies this relation. Then we also have $$F(\dot{\gamma}(t))=\lambda\textrm{ , }t\in[0,1];$$ and conversely. Since $$\pp{E}{x^i}=F\pp{F}{x^i}\textrm{ , }\pp{E}{y^i}=F\pp{F}{y^i},$$ it follows that along $\dot{\gamma}$ $$\pp{E}{x^i}-\valpp{}{t}\pp{E}{y^i}=\lambda\left(\pp{F}{x^i}-\valpp{}{t}\pp{F}{y^i}\right)\textrm{ , }i\in\left\{1,\dots,n\right\}.$$ This concludes the proof of our assertions.

We exhibit a further, more sophisticated method for introducing ``straight lines" in Finsler geometry. This method, at least implicitly, is of basic importance for our dissertation.

If $F: TM\rightarrow M$ is a Finsler function, then there exists a unique $C^1$ vector field $S$ on $TM$, which is smooth on the slit tangent manifold $\Tk M$, and has coordinate expression of form $$S=y^i\pp{}{x^i}-2G^i\pp{}{y^i},$$ where $$G^i=\frac{1}{4}g^{ij}\left(\PP{F^2}{x^r}{y^j}y^r-\pp{F^2}{x^j}\right),$$ $$(g^{ij}):=(g_{ij})^{-1}\textrm{ , }g_{ij}:=\frac{1}{2}\PP{F^2}{y^i}{y^j}=\PP{E}{y^i}{y^j}.$$ (We continue to suppose that $M$ admits a global coordinate system.) The functions $G^i$ are positive-homogeneous of degree 2, so $S$ is a spray, called the \textit{canonical spray} of the Finsler manifold. (An intrinsic definition of $S$ will be reviewed in Chapter \ref{ch6}; confer the above coordinate expressions with the concise formula (\ref{canonspray85}).)

The velocity curve (or ``canonical lift") $\dot{\gamma}: [0,1]\rightarrow TM$ of a curve\\ $\gamma:[0,1]\rightarrow M$ is an integral curve of $S$, i.e., $S\circ\dot{\gamma}=\ddot{\gamma}$ holds, if and only if, the components $\gamma^i:=u^i\circ\gamma$ of $\gamma$ satisfy the relations $$\gamma^{i''}+2G^i\circ\dot{\gamma}=0\textrm{  }(i\in\left\{1,\dots,n\right\}).$$ It may immediately be seen that \textit{these curves are just the }$E$\textit{-extremals}.\\

\textbf{Principles and method.} \textit{A substantial part of Finsler geometry may be developed purely in terms of the canonical spray determined by the Finsler function.} Briefly,
\begin{center}
\textbf{a large part of Finsler geometry is spray geometry.}
\end{center}
Throughout the Dissertation, our guiding principle will be this observation.

The remarks made on ``straight lines" in the previous section justify, that ``the principles of constructions" based on the spray approach give the same geometry as the classical approach when the straight lines, i.e., geodesics, are defined as F-extremals.

It seems to us, that the above principle has already clearly been recognized by Berwald. In his epoch-making posthumus paper \textit{``Über Finslersche und Cartansche Geometrie IV"} (ref. \cite{Berwald}) his starting point is a system of second-order ordinary differential equations of form $$x^{i''}+2G^i(x,x')=0\textrm{ , }i\in\left\{1,\dots,n\right\},$$ where the functions $G^i$ are of class $C^1$ on their domain, smooth on the set of non-zero tangent vectors, and have positive-homogeneity of degree 2. This means in present-day language, that Berwald takes a \textit{spray} as his starting point. Next, on Riemannian analogy, he derives the equation of affine deviation, which leads him to the \textit{affine deviation tensor} $\KK$ (see also our Remark after \ref{corotracegyh}). In our language and by our apparatus, it may be given by the formula $$\KK(\hul{X}):=\VV[S,\HH\hul{X}],$$ where $S$ is the given spray, $\HH$ is the ``nonlinear connection" or Ehresmann connection determined by $S$, $\VV$ is the vertical map complementary to $\HH$, and $\hul{X}$ is section along $\tauk:=\tau\upharpoonright\Tk M$. In terms of coordinates, $\HH$ is represented by the ``Christoffel symbols" $G^i_j:=\pp{G^i}{y^j}$. Using partial differentiation with respect to the directions, i.e., the operators $\pp{}{y^i}$, from the type $\binom{1}{1}$ tensor $\KK$ Berwald builds a type $\binom{1}{2}$ tensor $\RR$, called the \textit{``Grundtensor der affinen Krümmung"}, and a type $\binom{1}{3}$ tensor $\Hh$ called by him the \textit{``affine Krümmungtensor"}. Our formalism presents these tensors as follows: $$\RR(\hul{X},\hul{Y})=\frac{1}{3}(\vl{\nabla}\KK(\hul{Y},\hul{X})-\vl{\nabla}\KK(\hul{X},\hul{Y})),$$ $$\Hh(\hul{X},\hul{Y})\hul{Z}=\vl{\nabla}\RR(\hul{Z},\hul{X},\hul{Y}).$$
Actually, we follow in their introduction a somewhat different path. We define the tensor $\RR$ as the integrability tensor of the Ehresmann connection $\HH$, and the tensor $\Hh$ as a partial curvature of the curvature of the Berwald derivative arising from $\HH$. Having them, we show that $\RR$ and $\Hh$ can be obtained from $\KK$ by the above formulas.

Following Berwald, we may also construct from the affine deviation tensor (called also \textit{Jacobi endomorphism}) a projectively invariant tensor, the \textit{projective deviation tensor} or \textit{Weyl endomorphism} $$\mathbf{W}^{\circ}=\KK-K\mathbf{1}+\frac{1}{n+1}(\vl{\nabla}K-\textrm{\textup{tr}}\vl{\nabla}\KK)\otimes\delta,$$ where $K:=\frac{1}{n-1}\textrm{tr}\KK$, $\mathbf{1}$ is the unit tensor, and $\delta$ is given by $v\mapsto(v,v)$.

All of the tensors mentioned until now have an analogous tensor in Riemannian geometry. However, from the curvature of the induced Berwald derivative, one can obtain a further partial curvature, the so-called \textit{Berwald curvature} $\BB$, which is non-Riemannian in the sense that it vanishes, if $S$ is the geodesic spray of a Riemannian metric. From $\BB$ one can construct another projectively invariant tensor, the \textit{Douglas curvature} $$\mathbf{D}:=\mathbf{B}-\frac{1}{n+1}(\textrm{tr}\mathbf{B}\odot\mathbf{1}+(\vl{\nabla}\textrm{tr}\mathbf{B})\otimes\delta).$$

In the presence of a Finsler function $F$ we derive the tensors $\KK$, $\RR$, $\Hh$, $\mathbf{W}^{\circ}$, $\BB$, $\mathbf{D}$ from the canonical spray determined by $F$, and we define the conceptually most important special classes of Finsler manifolds (isotropic, Berwald, Douglas, weakly Berwald,...) also by some specific property of their canonical spray. For example, Berwald manifolds may be defined as Finsler manifolds whose canonical spray is of class $C^2$, i.e., is an affine spray, and may be characterized by the vanishing of the Berwald curvature. A great success of modern Finsler geometry is their complete description achieved by Z. I. Szabó \cite{Szabo}, \cite{Szabo2}.

We have two basic tensors in Finsler geometry which depend immediately on the metric structure, and hence cannot be defined in terms of spray geometry: the\\

\textit{Cartan tensor} $\CC_{\flat}:=\frac{1}{2}\vl{\nabla}g=\frac{1}{2}\vl{\nabla}\vl{\nabla}\vl{\nabla}E$\\

and the\\

\textit{Landsberg tensor} $\Pp:=-\frac{1}{2}\hl{\nabla}g=-\frac{1}{2}\hl{\nabla}\vl{\nabla}\vl{\nabla}E$\\

($\hl{\nabla}$ is the ``h-Berwald" or horizontal derivative arising from the canonical spray). The \textit{stretch tensor} $\Sigma$ defined by $$\frac{1}{2}\Sigma(\hul{X},\hul{Y},\hul{Z},\hul{U}):=\hl{\nabla}\Pp(\hul{X},\hul{Y},\hul{Z},\hul{U})-\hl{\nabla}\Pp(\hul{Y},\hul{X},\hul{Z},\hul{U})$$ belongs obviously also to this category. The vanishing of the Cartan tensor characterizes the Riemannian manifolds in the class of Finsler manifolds. The meaning of the vanishing of $\Pp$ is still in the dark. Matsumoto's conjecture \textit{``all positive definite Finsler manifolds with vanishing Landsberg curvature are Berwald manifolds"} has not been proved until now, but we also do not know any regular counterexample.

Giving a definite priority of the spray structure, in this Dissertation we build in an essentially self-contained manner the part of spray-Finsler geometry which we need to treat our specific problems. We do all these in the pull-back bundle framework, applying exclusively coordinate-free methods. Briefly, \textit{we apply the pull-back formalism}. Thus, for example, we re-prove our earlier result published in \cite{BSz2}, where we used classical tensor calculus. The self-containedness of our exposition also means that we prove most of the classically well-known auxiliary results whose formulation and proof is not available (in our best knowledge) in our context. Homogeneity properties, Ricci and Bianchi identites belong typically to this category. Since the translation from classical tensor calculus to an index-free formalism is not always automatical, we were forced to do this work in most cases.

In its contents an methods our exposition is somewhere in a half-way between Z. Shen's monograph \cite{Shen} and J. Szilasi's study \cite{Szilasi}: it follows a more rigorous formalism than the former, but simpler and is more near to the spirit of classical Finsler geometry than the latter.\\

\newpage
\pagestyle{empty} % itt nincs oldalszám

\  % üres oldal

\setcounter{chapter}{0}
\chapter{Conventions and basic definitions} \label{ch1}
\pagenumbering{arabic}
\setcounter{page}{1}
\pagestyle{headings}

\textbf{(A)} By a \textit{manifold} we shall always mean an at least two-dimensional, locally Euclidean, second countable, connected Hausdorff space with a smooth structure. If $M$ and $N$ are manifolds, $C^{\infty}(M,N)$ denotes the set of smooth maps from $M$ to $N$; $C^{\infty}(M):=C^{\infty}(M,\valR)$. The \textit{tangent space} $T_pM$ of $M$ at a point $p\in M$ is the real vector space of linear functions $v:\modc{M}\rightarrow\valR$, which satisfy $$v(fg)=v(f)g(p)+f(p)v(g)\textrm{ ; }f,g\in\modc{M}.$$ Then, for all $p\in M$, $\textrm{dim}T_pM=\textrm{dim}M=:n$. If $$TM:=\bigcup_{p\in M}T_pM\textrm{  }\textrm{(\textit{disjoint union})}$$ and $\tau(v):=p$, if $v\in T_pM$, then $\tau:TM\rightarrow M$ is the \textit{tangent bundle} of $M$. The tangent bundle of the tangent manifold $TM$ is $\tau_{TM}:TTM\rightarrow TM$. $\tau$ and $\tau_{TM}$ are examples of vector bundles, which will be briefly discussed in (C).\\

$\modx{M}:=\left\{X\in C^{\infty}(M,TM)\vert\tau\circ X=1_M\right\}$ is the $C^{\infty}(M)$-module of vector fields on $M$, its dual $\mathfrak{X}^{*}(M)$ is the module of 1-forms on $M$. If $X\in\modx{M}$, $\LL_X$ denotes the \textit{Lie derivative} with respect to $X$, and $i_X$ is the \textit{substitution operator} or \textit{contraction} by $X$. $d$ stands for the \textit{exterior derivative} operator.

If $o\in\modx{M}$ is the zero vector field, $\Tk M:=TM\backslash o(M)$, $\tauk:=\tau\upharpoonright\Tk M$, then $\tauk: \Tk M\rightarrow M$ is said to be the \textit{slit tangent bundle} of $M$. $\phi_{*}\in C^{\infty}(TM,TN)$ is the tangent linear map (or derivative) of $\phi\in C^{\infty}(M,N)$. If $I\subset\valR$ is an open interval and $c: I\rightarrow M$ is a smooth curve, then $\dot{c}:=c_{*}\circ\frac{d}{du}$ is the velocity vector field of $c$. (Here $\frac{d}{du}$ is the canonical vector field on the real line.) The \textit{vertical lift} of a function $f\in\modc{M}$ is $\vl{f}:=f\circ\tau\in\modc{TM}$, the \textit{complete lift} $\cl{f}\in\modc{TM}$ of $f$ is defined by $\cl{f}(v):=v(f)$, $v\in TM$. For any vector field $X$ on $M$ there is a unique vector field $\cl{X}\in\modx{TM}$ such that $\cl{X}\cl{f}=\cl{(Xf)}$ for any function $f\in\modc{M}$. $\cl{X}$ is called the \textit{complete lift of }$X$.\\

\textbf{(B)} We shall use wedge products in various contexts, without any numerical factor. For example, if $\alpha$ and $\beta$ are 1-forms on $M$, then their wedge product is $$\alpha\wedge\beta:=\alpha\otimes\beta-\beta\otimes\alpha,$$ where the symbol $\otimes$ denotes tensor product. If $A$ is a type $\binom{1}{1}$ tensor on $M$, which may be interpreted as an endomorphism of $\modx{M}$, and $\beta\in\mathfrak{X}^{*}(M)$, then the wedge product $A\wedge\beta$ is the skew-symmetric type $\binom{1}{2}$ tensor given by $$A\wedge\beta(X,Y)=\beta(Y)A(X)-\beta(X)A(Y)\textrm{ ; }X,Y\in\modx{M}.$$

If $K$ is a type $\binom{1}{1}$ tensor field on $\Tk M$, i.e., an endomorphism of the $\modc{\Tk M}$-module $\modx{\Tk M}$ and $\eta\in\modx{\Tk M}$, then we define the \textit{Frölicher-Nijenhuis bracket} $[K,\eta]$ by $$[K,\eta]\xi:=[K\xi,\eta]-K[\xi,\eta]\textrm{ ; }\xi\in\modx{\Tk M}.$$ Then $[K,\eta]$ is again a type $\binom{1}{1}$ tensor on $\Tk M$; it is just the negative of the Lie derivative $\LL_{\eta}K$. We also associate to $K$ two graded derivations $i_K$ and $d_K$ of the Grassmann algebra of differential forms on $\Tk M$, prescribing their operation on smooth functions and 1-forms by the following rules:
\begin{gather}\label{graddera}
i_KF:=0\textrm{ , }i_KdF:=dF\circ K\textrm{ ; }F\in\modc{\Tk M};
\end{gather}
\begin{gather}\label{gradderb}
d_K:=i_K\circ d-d\circ i_K.
\end{gather}
Then the degree of $i_K$ is 0, and the degree of $d_K$ is 1. On functions $d_K$ operates by $$d_KF=i_KdF=dF\circ K,$$ so for any vector field $\xi$ on $\Tk M$ we have $$d_KF(\xi)=dF(K(\xi))=K(\xi)F.$$

\textbf{(C)} We recall that a smooth map $\pi: E\rightarrow M$ is said to be a (real) \textit{vector bundle of rank} $k$ ($k\in\termN\backslash\left\{0\right\}$) or $k$\textit{-vector bundle} over $M$, if the following conditions are satisfied:
\begin{itemize}
\item[(VB$_{\textrm{1}}$)] For all $p\in M$, the \textit{fibres} $E_p:=\pi^{-1}(p)$ are $k$-dimensional real vector spaces.
\item[(VB$_{\textrm{2}}$)] To each point $p\in M$ there is a neighbourhood $\UU\subset M$ of $p$ and a diffeomorphism $\varphi: \UU\times\valR^k\rightarrow\pi^{-1}(\UU)$ such that
\begin{itemize}
\item[(i)] $\pi\circ\varphi=\textrm{pr}_1$, where $\textrm{pr}_1$ is the natural projection of $\UU\times\valR^k$ onto its first factor;
\item[(ii)] for each point $q\in\UU$, the map $$\varphi_q:\valR^k\rightarrow E_q\textrm{  ,  }v\mapsto\varphi_q(v):=\varphi(q,v)$$ is a linear isomorphism.
\end{itemize}
\end{itemize}
Then $M$, $E$ and $\pi$ are called the \textit{base manifold}, the \textit{total manifold} and the \textit{projection} of the bundle, respectively. When there is no danger of confusion, we say merely that `$E$ is a vector bundle over $M$' or `$E$ is a vector bundle'. A smooth map $\sigma: M\rightarrow E$ is a \textit{section} of the vector bundle $\pi: E\rightarrow M$, if $\pi\circ\sigma=1_M$, i.e., we have $\sigma(p)\in E_p$ for all $p\in M$. The set $$\textrm{Sec}(\pi):=\left\{\sigma\in\modc{M,E}\textrm{  }\vert\textrm{  }\pi\circ\sigma=1_M\right\}$$ of all sections of $\pi$ is a $\modc{M}$-module with the pointwise operations $$(\sigma_1+\sigma_2)(p):=(\sigma_1)(p)+(\sigma_2)(p)\textrm{  ,  }(f\sigma)(p):=f(p)\sigma(p)$$ ($\sigma_1,\sigma_2,\sigma\in\textrm{Sec}(\pi)\textrm{ , }f\in\modc{M}\textrm{ , }p\in M$).

Let $(\textrm{Sec}(\pi))^{*}$ denote the dual of the $\modc{M}$-module $\textrm{Sec}(\pi)$. By a $\pi$\textit{-tensor} of type $\binom{r}{s}$, where $(r,s)\in\termN\times\termN\backslash\left\{(0,0)\right\}$, we mean a $\modc{M}$-multilinear map $$((\textrm{Sec}(\pi))^{*})^r\times(\textrm{Sec}(\pi))^s\rightarrow\modc{M}.$$ These form a $\modc{M}$-module which we denote by $\tenz{r}{s}{\pi}$. We extend the definition by putting $\tenz{0}{0}{\pi}:=\textrm{Sec}(\pi)$.

The $\pi$-tensors of type $\binom{r}{s}$ over $M$ may naturally be interpreted as the sections of an appropriate vector bundle over $M$. Namely, let $T^r_sE_p$ be the space of the type $\binom{r}{s}$ tensors over the fiber $E_p$, and let $$T^r_sE:=\bigcup_{p\in M}T^r_sE_p\textrm{  }\textrm{(\textit{disjoint union}).}$$ Then there is a unique smooth structure on $T^r_sE$ which makes the natural projection $$\pi^r_s: T^r_sE\rightarrow M$$ into a vector bundle with fibres $T^r_sE_p$, $p\in M$. Now it may be shown that \textit{the }$\modc{M}$\textit{-modules }$\tenz{r}{s}{\pi}$\textit{ and }$\szel{\pi^r_s}$\textit{ are canonically isomorphic.}\\

\textbf{(D)} Let $\pi: E\rightarrow M$ be a $k$-vector bundle. A \textit{covariant derivative operator}, briefly a \textit{covariant derivative} in $\pi$ (or on $E$) is a map $$D:\modx{M}\times\szel{\pi}\rightarrow\szel{\pi}\textrm{ , }(X,\sigma)\mapsto D_X\sigma,$$ which is tensorial in $X$, $\valR$-linear in $\sigma$, and satisfies the following product rule: $$D_Xf\sigma=(Xf)\sigma+fD_X\sigma\textrm{  }\textrm{  for  }f\in\modc{M}.$$ $\nabla_X\sigma$ is called the \textit{covariant derivative of }$\sigma$ \textit{in the direction of} $X$.

Although a covariant derivative operator is defined by its action on global sections, it may be shown by a standard bump function argument that it is actually a local operator: if two sections coincide in a neighbourhood of a point, then their covariant derivatives are the same at the point. Note that, by an abuse of language, a covariant derivative in the tangent bundle $\tau: TM\rightarrow M$ is mentioned as a \textit{covariant derivative on the manifold} $M$.

We define the \textit{covariant differential} of a tensor $A\in\tenz{r}{s}{\pi}$ as the type $\binom{r}{s+1}$ tensor $DA$ given by
$$DA(X,s^1,\dots,s^r,\sigma_1,\dots,\sigma_s):=(D_XA)(s^1,\dots,s^r,\sigma_1,\dots,\sigma_s):=$$
\begin{gather}\label{vbdelta}
X(A(s^1,\dots,s^r,\sigma_1,\dots,\sigma_s))-\sum_{i=1}^rA(s^1,\dots,D_Xs^i,\dots,s^r,\sigma_1,\dots,\sigma_s)-
\end{gather}
$$\sum_{j=1}^sA(s^1,\dots,s^r,\sigma_1,\dots,D_X\sigma_j,\dots,\sigma_s),$$
where $s^i\in(\szel{\pi})^{*}$, $i\in\left\{1,\dots,r\right\}$; $\sigma_j\in\szel{\pi}$, $j\in\left\{1,\dots,s\right\}$, and if $s$ is a `$\pi$-one-form', i.e., $s\in(\szel{\pi})^{*}$, then
\begin{gather}\label{vbs}
(D_Xs)(\sigma):=X(s(\sigma))-s(D_X\sigma).
\end{gather}
The \textit{curvature of} $D$ is the map
$$R^D:\modx{M}\times\modx{M}\times\szel{\pi}\rightarrow\szel{\pi},$$
\begin{gather}\label{vbr}
(X,Y,\sigma)\mapsto R^D(X,Y)\sigma:=D_XD_Y\sigma-D_YD_X\sigma-D_{[X,Y]}\sigma.
\end{gather}
Then $R^D$ is tensorial (i.e., $\modc{M}$-linear) in $X$, $Y$ and $\sigma$, and skew-symmetric in $X$ and $Y$: $$R(X,Y)\sigma=-R(Y,X)\sigma.$$
With fixed vector fields $X$, $Y$ on $M$, the map $$R^D(X,Y): \szel{\pi}\rightarrow\szel{\pi}\textrm{  ,  }\sigma\mapsto R^D(X,Y)\sigma$$is an endomorphism of the $\modc{M}$-module $\szel{\pi}$, so it may be interpreted as a type $\binom{1}{1}$ $\pi$-tensor in a natural manner. We define the covariant derivatives $$D_Z(R^D(X,Y))\in\textrm{End}(\szel{\pi})\cong\tenz{1}{1}\pi\textrm{ , }Z\in\modx{M}$$ by
\begin{gather}\label{biastar}
(D_Z(R^D(X,Y)))(\sigma):=D_Z(R^D(X,Y)\sigma)-R^D(X,Y)D_Z\sigma\textrm{  ,  }\sigma\in\szel{\pi}.
\end{gather}
Now we can formulate the following classical result, quoted as \textit{differential Bianchi identity}:
\begin{gather}\label{vbianchi}
\underset{(X,Y,Z)}{\mathfrak{S}}(D_X(R^D(Y,Z))-R^D([X,Y],Z))=0.
\end{gather}
Here, and in the sequel, $\underset{(X,Y,Z)}{\mathfrak{S}}$ means cyclic sum over $X$, $Y$ and $Z$.

The proof is easy. Since the relation is tensorial, we may assume that $[X,Y]=[Y,Z]=[Z,X]=0$. Then, for any section $\sigma$ in $\szel{\pi}$ we have
\begin{center}
$\displaystyle\underset{(X,Y,Z)}{\mathfrak{S}}(D_X(R^D(Y,Z))-R^D([X,Y],Z))(\sigma)=(D_X(R^D(Y,Z)))(\sigma)+(D_Y(R^D(Z,X)))(\sigma)+(D_Z(R^D(X,Y)))(\sigma)\overset{(\ref{biastar})}{=}D_X(R^D(Y,Z)\sigma)+D_Y(R^D(Z,X)\sigma)+D_Z(R^D(X,Y)\sigma)-R^D(Y,Z)D_X\sigma-R^D(Z,X)D_Y\sigma-R^D(X,Y)D_Z\sigma=D_XD_YD_Z\sigma-D_XD_ZD_Y\sigma+D_YD_ZD_X\sigma-D_YD_XD_Z\sigma+D_ZD_XD_Y\sigma-D_ZD_YD_X\sigma-D_YD_ZD_X\sigma+D_ZD_YD_X\sigma-D_ZD_XD_Y\sigma+D_XD_ZD_Y\sigma-D_XD_YD_Z\sigma+D_YD_XD_Z\sigma=0$.
\end{center}

\textbf{(E)} The main scenes of our considerations will be the pull-back bundles of the tangent bundle $\tau:TM\rightarrow M$ over $\tau$ and $\tauk$, i.e., the vector bundles
\begin{center}
$\displaystyle\pi: TM\times_MTM\rightarrow TM$ and $\displaystyle\pik:\Tk M\times_MTM\rightarrow\Tk M$,
\end{center}
respectively.
Here
\begin{center}
$\displaystyle TM\times_MTM:=\left\{(u,v)\in TM\times TM\textrm{  }\vert\textrm{  }\tau(u)=\tau(v)\right\}$,\\
$\displaystyle \pi: (u,v)\in TM\times_MTM\mapsto u\in TM$,\\
$\displaystyle\pi^{-1}(u)=\left\{u\right\}\times T_{\tau(u)}M$;
\end{center}
and, similarly,
\begin{center}
$\displaystyle\Tk M\times_MTM:=\left\{(u,v)\in\Tk M\times TM\textrm{  }\vert\textrm{  }\tauk(u)=\tau(v)\right\}$,\\
$\displaystyle\pik: (u,v)\in\Tk M\times_MTM\mapsto u\in\Tk M$,\\
$\displaystyle\pik^{-1}(u)=\left\{u\right\}\times T_{\tauk(u)}M$.
\end{center}

The vector space structure of a fibre $\pi^{-1}(u)$ is given by the operations $$(u,v_1)+(u,v_2):=(u,v_1+v_2)\textrm{ ; }v_1,v_2\in T_{\tau(u)}M;$$ $$\lambda(u,v):=(u,\lambda v)\textrm{ , }v\in T_{\tau(u)}M\textrm{ , }\lambda\in\valR.$$

Then, obviously, $\pi^{-1}(u)$ is canonically isomorphic to $T_{\tau(u)}M$. Similarly, the fibres $\pik^{-1}(u)$ are vector spaces, canonically isomorphic to the tangent spaces $T_{\tauk(u)}M$ ($u\in\Tk M$).

\begin{center}
\textit{In what follows, throughout the Dissertation, }$\pi$\textit{ and }$\pik$\textit{ will be the shorthands for the pull-back bundles}\\
$\pi: TM\times_M TM\rightarrow TM$ \textit{and} $\pik: \Tk M\times_MTM\rightarrow\Tk M$,\\
\textit{respectively.}
\end{center}

The sections of $\pi$ are smooth maps $\hul{X}: TM\rightarrow TM\times_MTM$ of form
\begin{center}
$\displaystyle v\in TM\longmapsto\hul{X}(v)=(v,\underline{X}(v))\in TM\times_MTM$,\\
$\displaystyle\underline{X}\in\modc{TM,TM}$ , $\displaystyle\tau\circ\underline{X}=\tau$.
\end{center}
We have a \textit{canonical section}
\begin{center}
$\displaystyle\delta: v\in TM\longmapsto \delta(v):=(v,v)\in TM\times_MTM$,
\end{center}
and any vector field $X$ on $M$ induces a section
\begin{center}
$\displaystyle\kal{X}: v\in TM\longmapsto \kal{X}(v):=(v,X(\tau(v)))\in TM\times_MTM$,
\end{center}
called a \textit{basic section} of $\pi$ or a \textit{basic vector field} along $\tau$. The $\modc{TM}$-module $Sec(\pi)$ of sections of $\pi$ is generated by the basic sections. If
\begin{center}
$\displaystyle\modx{\tau}:=\left\{\underline{X}\in C^{\infty}(TM,TM)\vert \tau\circ\underline{X}=\tau\right\}$,
\end{center}
then $\modx{\tau}$ is naturally isomorphic to $Sec(\pi)$, so the two modules will be identified without any comment, whenever it is convenient. As in the general case of a $k$-vector bundle (see (C)), we may consider the $\modc{TM}$-modules $\tenz{r}{s}{\pi}$ of the $\pi$-tensors over $TM$, and, similarly, the $\modc{\Tk M}$-modules $\tenz{r}{s}{\pik}$ of $\pik$-tensors over $\Tk M$; $(r,s)\in\termN\times\termN$. Note that $\tenz{r}{s}{\pi}$ may naturally be interpreted as a submodule of $\tenz{r}{s}{\pik}$. From obvious reason, the elements of $\tenz{r}{s}{\pi}$ and $\tenz{r}{s}{\pik}$ will also be mentioned as tensors along $\tau$ and $\tauk$, respectively.

If $\AA$ is a type $\binom{1}{s+1}$ tensor along $\tauk$, where $s\in\termN$, then we define its \textit{trace} $\textrm{tr}\AA\in\EuScript{T}^0_s(\pik)$ by $$(\textrm{tr}\AA)(\hul{X_1},\dots,\hul{X_s}):=\textrm{tr}\left(\hul{Z}\longmapsto\AA(\hul{Z},\hul{X_1},\dots,\hul{X_s})\right),$$ for $\hul{X_1},\dots,\hul{X_s},\hul{Z}\in\textrm{Sec}(\pik)$.

More effectively, the trace operator may be introduced inductively as follows:

\textbf{Step 1} There is a unique $\modc{\Tk M}$-linear map $$\textrm{tr}: \tenz{1}{1}{\pik}\rightarrow\modc{\Tk M}\textrm{ , }\AA\mapsto\textrm{tr}\AA$$ such that for all 1-form $\hul{\alpha}\in\tenz{0}{1}{\pik}$ and section $\hul{X}$ in $\szel{\pik}$ we have
\begin{gather}\label{trace01}
\textrm{tr}(\hul{\alpha}\otimes\hul{X}):=\hul{\alpha}(\hul{X}).
\end{gather}

This may be shown by a standard argument, see e.g. \cite{Oneil}, Lemma 2.6.

\textbf{Step 2} Let $\hul{X}\in\szel{\pik}$. First we define a non-conventional substitution operator $$j_{\hul{X}}: \tenz{1}{s+1}{\pi}\rightarrow\tenz{1}{s}{\pi}\textrm{ , }\AA\mapsto j_{\hul{X}}\AA$$ by $$j_{\hul{X}}\AA(\hul{Y}_1,\dots,\hul{Y}_s):=\AA(\hul{Y}_1,\hul{X},\hul{Y}_2,\dots,\hul{Y}_s)$$ for $\hul{Y}_1,\dots,\hul{Y}_s$ in $\szel{\pik}$. Using this operator, let the trace of a tensor\\ $\AA\in\tenz{1}{s+1}{\pik}$ $(s\geq 1)$ be the type $\binom{1}{s}$ tensor $\textrm{tr}\AA$ such that for any section $\hul{X}$ in $\szel{\pik}$,
\begin{gather}\label{trace02}
i_{\hul{X}}\textrm{tr}\AA=\textrm{tr}(j_{\hul{X}}\AA).
\end{gather}
The effect of the trace operator on components is summation over the contravariant index and the \textit{first} covariant index. It may be shown that if $\hul{\alpha}$ is a symmetric type $\binom{0}{s}$ tensor along $\tauk$ $(s\geq 2)$, then
\begin{gather}\label{trace03}
\textrm{tr}(\hul{\alpha}\otimes\delta)=i_{\delta}\hul{\alpha},
\end{gather}
where $i_{\delta}$ is the 'conventional' substitution operator.\\

\textbf{(F)} We have a canonical injective strong bundle map $$\ii: TM\times_MTM\longrightarrow TTM$$ given by
\begin{center}
$\displaystyle\ii(v,w):=\dot{c}(0)$ , if $\displaystyle c(t):=v+tw$,
\end{center}
and a canonical surjective strong bundle map
\begin{center}
$\displaystyle\jj: TTM\longrightarrow TM\times_MTM$, $\displaystyle w\in T_vTM\longmapsto \jj(w):=(v,\tau_{*}(w))$
\end{center}
such that the sequence
\begin{center}
$\displaystyle 0\longrightarrow TM\times_MTM\overset{\ii}{\longrightarrow}TTM\overset{\jj}{\longrightarrow}TM\times_MTM\longrightarrow 0$
\end{center}
is an exact sequence of vector bundle maps. $\ii$ and $\jj$ induce $C^{\infty}(TM)$-\\homomorphisms at the level of sections, which will be denoted by the same letters. So we also have the exact sequence
\begin{center}
$\displaystyle 0\longrightarrow \modx{\tau}\overset{\ii}{\longrightarrow}\modx{TM}\overset{\jj}{\longrightarrow}\modx{\tau}\longrightarrow 0$
\end{center}
of module homomorphisms. $\mathfrak{X}^{\mathsf{v}}(TM):=\ii\modx{\tau}$ is the module of \textit{vertical vector fields} on $TM$, $\vl{X}:=\ii\kal{X}$ is the \textit{vertical lift} of $X\in\modx{M}$. If $\alpha$ is a 1-form on $M$, then there exists a unique 1-form $\vl{\alpha}$ on $TM$ such that $$\vl{\alpha}(\vl{X})=0\textrm{ , }\vl{\alpha}(\cl{X})=\vl{(\alpha(X))}$$ for all $X\in\modx{M}$. $\vl{\alpha}$ is said to be the vertical lift of $\alpha$.

$C:=\ii\delta$ is a canonical vertical vector field on $TM$, the \textit{Liouville vector field}. For any vector field $X$ on $M$ we have
\begin{gather}\label{1}
\left[C,\vl{X}\right]=-\vl{X}\textrm{ , }\left[C,\cl{X}\right]=0.
\end{gather}

$\JJ:=\ii\circ\jj$ is a tensor field on $TM$ of type $\binom{1}{1}$; it is called the \textit{vertical endomorphism}. For all vector fields $X$ on $M$ we have $$\JJ\vl{X}=0\textrm{ , }\JJ\cl{X}=\vl{X};$$ therefore $$\textrm{Im}(\JJ)=\textrm{Ker}(\JJ)=\vl{\mathfrak{X}}(TM)\textrm{ , }\JJ^2=0.$$
%We shall need the \textit{Frölicher-Nijenhuis bracket} $[\JJ,\eta]$ of $\JJ$ and a vector field $\eta\in\modx{TM}$ defined by $$[\JJ,\eta]\xi:=[\JJ\xi,\eta]-\JJ[\xi,\eta]\textrm{ ; }\xi\in\modx{TM}.$$ (Notice that $[\JJ,\eta]$ is just the negatve of the Lie derivative $\LL_{\eta}\JJ$.) Then, in particular
The following useful relations may be verified immediately:
\begin{gather}\label{partic2}
[\JJ,C]=\JJ\textrm{ ; }[\JJ,\vl{X}]=[\JJ,\cl{X}]=0\textrm{ , }X\in\modx{M}.
\end{gather}

\textbf{(G)} We define the \textit{vertical differential} $\vl{\nabla}F\in\tenz{0}{1}{\pi}$ of a function\\ $F\in\modc{TM}$ by
\begin{gather}\label{2}
\vl{\nabla}F(\hul{X}):=(\ii\hul{X})F\textrm{ , }\hul{X}\in\textrm{Sec}(\pi).
\end{gather}
We note that
\begin{gather}\label{gradderc}
\vl{\nabla}F\circ\jj=d_{\JJ}F,
\end{gather}
where $d_{\JJ}$ is the graded derivation associated to the vertical endomorphism by (\ref{graddera}) and (\ref{gradderb}).
The vertical differential of a section $\hul{Y}\in\textrm{Sec}(\pi)$ is the type $\binom{1}{1}$ tensor $\vl{\nabla}\hul{Y}\in\tenz{1}{1}{\pi}$ given by
\begin{gather} \label{3}
\vl{\nabla}\hul{Y}(\hul{X})=:\vl{\nabla}_{\hul{X}}\hul{Y}:=\jj[\ii\hul{X},\eta]\textrm{ , }\hul{X}\in\textrm{Sec}(\pi),
\end{gather}
where $\eta\in\modx{TM}$ is such that $\jj\eta=\hul{Y}$. (It is easy to check that the result does not depend on the choice of $\eta$.) Using the Leibnizian product rule as a guiding principle, the operators $\vl{\nabla}_{\hul{X}}$ may uniquely be extended to a tensor derivation of the tensor algebra of $\textrm{Sec}(\pi)$. Forming the vertical differential of a tensor over $\textrm{Sec}(\pi)$, we use the convention applied in (\ref{vbdelta}): if, e.g., $\mathbf{A}\in\tenz{1}{2}{\pi}$, then $\vl{\nabla}\mathbf{A}\in\tenz{1}{3}{\pi}$ is given by
\begin{center}
$\displaystyle\vl{\nabla}\mathbf{A}(\hul{X},\hul{Y},\hul{Z}):=(\vl{\nabla}_{\hul{X}}\mathbf{A})(\hul{Y},\hul{Z})=\vl{\nabla}_{\hul{X}}\mathbf{A}(\hul{Y},\hul{Z})-\mathbf{A}(\vl{\nabla}_{\hul{X}}\hul{Y},\hul{Z})-\mathbf{A}(\hul{Y},\vl{\nabla}_{\hul{X}}\hul{Z})$.
\end{center}

A type $\binom{0}{s}$ or $\binom{1}{s}$ tensor $\AA$ along $\tauk$ is said to be \textit{homogeneous of degree} $k$, where $k$ is an integer, if $$\vl{\nabla}_{\delta}\AA=k\AA.$$

\chapter{Ehresmann connections and Berwald derivatives} \label{chap2}

By an \emph{Ehresmann connection} over $M$ we mean a map $$\HH\colon TM\times_M TM\rightarrow TTM$$ satisfying the following conditions:
\begin{itemize}
	\item[(C$_{\textrm{1}}$)] $\HH$ is fibre preserving and fibrewise linear, i.e., for every $v\in TM$,\\ $\HH_v:=\HH\upharpoonright \left\{v\right\}\times T_{\tau(v)}M$ is a linear map from $\left\{v\right\}\times T_{\tau(v)}M\cong T_{\tau(v)}M$ into $T_v TM$.
	\item[(C$_{\textrm{2}}$)] $\jj\circ\HH=1_{TM\times_M TM}$, i.e., ``$\HH$ splits".
	\item[(C$_{\textrm{3}}$)] $\HH$ is smooth over $\circc{T}M\times_M TM$.
	\item[(C$_{\textrm{4}}$)] If $o\colon M\rightarrow TM$ is the zero vector field, then $\HH(o(p),v)=(o_{*})_p (v)$, for all $p\in M$ and $v\in T_p M$.
\end{itemize}

We associate to an Ehresmann connection $\HH$
\begin{itemize}
\item[ ] the \textit{horizontal projector} $\hh:=\HH\circ\jj$, the \textit{vertical projector} $\vv:=1_{T\Tk M}-\hh$,
\item[ ] the \textit{vertical map} $\VV:=\ii^{-1}\circ\vv: T\Tk M\rightarrow \Tk M\times_MTM$,
\item[ ] the \textit{almost complex structure} $\FF:=\HH\circ\VV-\ii\circ\jj=\HH\circ\VV-\JJ$.
\end{itemize}

We have the following basic relations:
$$\hh^2=\hh\textrm{ , }\vv^2=\vv\textrm{ ; }\JJ\circ\hh=\JJ\textrm{ , }\hh\circ\JJ=0\textrm{ ; }\JJ\circ\vv=0\textrm{ , }\vv\circ\JJ=\JJ;$$ $$\FF^2=-\mathbf{1}\textrm{ , }\JJ\circ\FF=\vv\textrm{ , }\FF\circ\JJ=\hh;$$ $$\FF\circ\hh=-\JJ\textrm{ , }\hh\circ\FF=\FF\circ\vv=\JJ+\FF\textrm{ , }\vv\circ\FF=-\JJ.$$

The \textit{horizontal lift} of a vector field $X\in\modx{M}$ (with respect to $\HH$) is $$\hl{X}:=\HH\circ\kal{X}=:\HH\kal{X}=\hh\cl{X}.$$

It may be shown (see e.g. \cite{Szilasi}) that for all vector fields $X$, $Y$ on $M$ we have
\begin{gather}\label{jh4}
\JJ[\hl{X},\hl{Y}]=\vl{[X,Y]}\textrm{ , }\hh[\hl{X},\hl{Y}]=\hl{[X,Y]}.
\end{gather}
By the \textit{tension} of $\HH$ we mean the type $\binom{1}{1}$ tensor field $\tt$ along $\tauk$ given by $$\tt(\hul{X}):=\VV[\HH\hul{X},C]\textrm{ , }\hul{X}\in\textrm{Sec}(\pik).$$ Then $$\ii\tt(\kal{X})=[\hl{X},C]\textrm{ , }X\in\modx{M}.$$\\
$\HH$ is said to be \textit{homogeneous} if its tension vanishes. We define the \textit{torsion} and the \textit{curvature} of $\HH$ by $$\TT(\hul{X},\hul{Y}):=\VV[\HH\hul{X},\ii\hul{Y}]-\VV[\HH\hul{Y},\ii\hul{X}]-\jj[\HH\hul{X},\HH\hul{Y}]$$ and $$\RR(\hul{X},\hul{Y}):=-\VV[\HH\hul{X},\HH\hul{Y}]$$ ($\hul{X},\hul{Y}\in\szel{\pik}$), respectively. Evaluating on basic vector fields, we obtain the more expressive relations $$\ii\TT(\kal{X},\kal{Y})=[\hl{X},\vl{Y}]-[\hl{Y},\vl{X}]-\vl{[X,Y]}$$ and $$\ii\RR(\kal{X},\kal{Y})=-\vv[\hl{X},\hl{Y}].$$

Now we recall an elementary, but crucial construction of Ehresmann connections. To this end, at this point we introduce the concept of a semispray and spray, the latter will play the leading role in the Dissertation.

By a \textit{semispray} over a manifold $M$ we mean a map $S:TM\rightarrow TTM$ satisfying the following conditions:
\begin{itemize}
\item[(S$_{\textrm{1}}$)] $\tau_{TM}\circ S=1_{TM}$;
\item[(S$_{\textrm{2}}$)] $S$ is smooth over $\Tk M$;
\item[(S$_{\textrm{3}}$)] $\JJ S=C$ (or, equivalently, $\jj S=\delta$).
\end{itemize}

A semispray $S$ is said to be a \textit{spray}, if it satisfies the additional conditions
\begin{itemize}
\item[(S$_{\textrm{4}}$)] $S$ is of class $C^1$ over $TM$;
\item[(S$_{\textrm{5}}$)] $[C,S]=S$, i.e., $S$ is positive-homogeneous of degree 2.
\end{itemize}

If a spray is of class $C^2$ (and hence smooth) over $TM$, then it is called an \textit{affine spray}. Following S. Lang's terminology \cite{Lang}, we say that a \textit{smooth} map $S: TM\rightarrow TTM$ is a \textit{second-order vector field} over $M$, if it satisfies conditions (S$_{\textrm{1}}$) and (S$_{\textrm{3}}$). Notice, however, that by a `spray' Lang means a second-order vector field satisfying the homogeneity condition (S$_{\textrm{5}}$), i.e., an `affine spray' in our sense.

Given a semispray $S$ over $M$, by a celebrated result of M. Crampin \cite{Cramp} and J. Grifone \cite{Grif}, there exists a unique Ehresmann connection $\HH$ over $M$ such that
\begin{gather}\label{ehre6}
\HH(\kal{X})=\frac{1}{2}\left(\cl{X}+[\vl{X},S]\right)
\end{gather}
for all vector fields $X$ on $M$. $\HH$ is said to be the \textit{Ehresmann connection associated to} (or \textit{generated by}) $\HH$. The torsion of this Ehresmann connection vanishes. Furthermore, we have $$\HH(\delta)=\frac{1}{2}(S+[C,S]).$$ If, in particular, $S$ is a spray, then $\HH(\delta)=S$, and $\HH$ is homogeneous, i.e., its tension also vanishes.

We define the \textit{h-Berwald differentials} $\hl{\nabla}F\in\tenz{0}{1}{\pik}$ $(F\in\modc{\Tk M})$ and $\hl{\nabla}\hul{Y}\in\tenz{1}{1}{\pik}$ $(\hul{Y}\in\textrm{Sec}(\pik))$ by the following rules:
\begin{gather}
\hl{\nabla}F(\hul{X}):=(\HH\hul{X})F\textrm{ , }\hul{X}\in\textrm{Sec}(\pik);
\end{gather}
\begin{gather}\label{13}
\hl{\nabla}\hul{Y}(\hul{X}):=\hl{\nabla}_{\hul{X}}\hul{Y}:=\VV[\HH\hul{X},\ii\hul{Y}]\textrm{ , }\hul{X}\in\textrm{Sec}(\pik).
\end{gather}
The operators $\hl{\nabla}_{\hul{X}}$ $(\hul{X}\in\textrm{Sec}(\pik))$ may also uniquely be extended to the whole tensor algebra of $\textrm{Sec}(\pik)$ as tensor derivations. Forming the h-Berwald differential of an arbitrary tensor, we adopt the same convention as in the vertical case and in general, see (\ref{vbdelta}). We note that the tension of $\HH$ is just the h-covariant differential of the canonical section, i.e., $\tt=\hl{\nabla}\delta$. So the homogeneity of $\HH$ means that
\begin{gather}\label{14}
\hl{\nabla}\delta=0.
\end{gather}

We may also consider the graded derivation $d_{\hh}$ associated to the horizontal projector $\hh=\HH\circ\jj$; then we have
\begin{gather}\label{gradderd}
\hl{\nabla}F\circ\jj=d_{\hh}F\textrm{  }(F\in\modc{\Tk M}).
\end{gather}

From the operators $\vl{\nabla}$ and $\hl{\nabla}$ we build the \textit{Berwald derivative}
\begin{center}
$\displaystyle\nabla: (\xi,\hul{Y})\in\modx{\Tk M}\times\textrm{Sec}(\pik)\longmapsto\nabla_{\xi}\hul{Y}:=\vl{\nabla}_{\VV\xi}\hul{Y}+\hl{\nabla}_{\jj\xi}\hul{Y}\in\textrm{Sec}(\pik)$.
\end{center}
Then, by (\ref{3}) and (\ref{13}), $$\nabla_{\xi}\hul{Y}=\jj[\vv\xi,\HH\hul{Y}]+\VV[\hh\xi,\ii\hul{Y}].$$ In particular,
\begin{center}
$\displaystyle\nabla_{\ii\hul{X}}\hul{Y}=\vl{\nabla}_{\hul{X}}\hul{Y}$ , $\displaystyle\nabla_{\HH\hul{X}}\hul{Y}=\hl{\nabla}_{\hul{X}}\hul{Y}$ ; $\displaystyle\hul{X},\hul{Y}\in\textrm{Sec}(\pik)$;
\end{center}
\begin{gather}\label{15}
\nabla_{\vl{X}}\kal{Y}=0\textrm{ , }\ii\nabla_{\hl{X}}\kal{Y}=\left[\hl{X},\vl{Y}\right]\textrm{ ; }X,Y\in\modx{M}.
\end{gather}

\begin{lemma}\textup{(hh-Ricci identity for functions).}
Let $\HH$ be a torsion-free Ehresmann connection over $M$. If $f:\Tk M\rightarrow\valR$ is a smooth function, then for any sections $\hul{X}$, $\hul{Y}$ in $\szel{\pik}$ we have
\begin{gather}\label{hhricci21}
\hl{\nabla}\hl{\nabla}f(\hul{X},\hul{Y})-\hl{\nabla}\hl{\nabla}f(\hul{Y},\hul{X})=-\ii\RR(\hul{X},\hul{Y})f.
\end{gather}
\end{lemma}

\begin{proof}
It is enough to show that formula (\ref{hhricci21}) is true for basic vector fields $\kal{X},\kal{Y}\in\szel{\pik}$. Then
\begin{center}
$\displaystyle(\hl{\nabla}\hl{\nabla}f)(\kal{X},\kal{Y})=\left(\nabla_{\hl{X}}(\hl{\nabla}f)\right)(\kal{Y})=\hl{X}\hl{Y}f-\hl{\nabla}f(\nabla_{\hl{X}}\kal{Y})=\hl{X}\hl{Y}f-(\HH\nabla_{\hl{X}}\kal{Y})f=\hl{X}\hl{Y}f-\HH\VV[\hl{X},\vl{Y}]f=$\\
$\displaystyle\hl{X}\hl{Y}f-(\FF+\JJ)[\hl{X},\vl{Y}]f=\hl{X}\hl{Y}f-\FF[\hl{X},\vl{Y}]f$,
\end{center}and in the same way
\begin{center}
$\displaystyle(\hl{\nabla}\hl{\nabla}f)(\kal{Y},\kal{X})=\hl{Y}\hl{X}f-\FF[\hl{Y},\vl{X}]f$.
\end{center}
So we obtain
\begin{center}
$\displaystyle\hl{\nabla}\hl{\nabla}f(\kal{X},\kal{Y})-\hl{\nabla}\hl{\nabla}f(\kal{Y},\kal{X})=[\hl{X},\hl{Y}]f-\FF([\hl{X},\vl{Y}]-[\hl{Y},\vl{X}])f\overset{\TT=0}{=}[\hl{X},\hl{Y}]f-(\FF\vl{[X,Y]})f=([\hl{X},\hl{Y}]-\hl{[X,Y]})f=([\hl{X},\hl{Y}]-\hh[\hl{X},\hl{Y}])f=\vv[\hl{X},\hl{Y}]f=-\ii\RR(\kal{X},\kal{Y})f$,
\end{center}
which proves our assertion.
\end{proof}

\begin{propo} \textup{(general Bianchi identity).}
If $\HH$ is a torsion-free Ehresmann connection and $\RR$ is the curvature of $\HH$, then for any vector fields $X$, $Y$, $Z$ on $M$ we have
\begin{gather}\label{genbin}
\underset{(X,Y,Z)}{\mathfrak{S}}(\hl{\nabla}\RR)(\kal{X},\kal{Y},\kal{Z})=0.
\end{gather}
\end{propo}

\begin{proof}
\begin{center}
$\displaystyle(\hl{\nabla}\RR)(\kal{X},\kal{Y},\kal{Z})=(\nabla_{\hl{X}}\RR)(\kal{Y},\kal{Z})=\nabla_{\hl{X}}(\RR(\kal{Y},\kal{Z}))-\RR(\nabla_{\hl{X}}\kal{Y},\kal{Z})-\RR(\kal{Y},\nabla_{\hl{X}}\kal{Z})=\VV[\hl{X},\ii\RR(\kal{Y},\kal{Z})]-\RR(\VV[\hl{X},\vl{Y}],\kal{Z})-\RR(\kal{Y},\VV[\hl{X},\vl{Z}])=\VV([\hl{X},\hl{[Y,Z]}-[\hl{Y},\hl{Z}])+\VV[(\JJ+\FF)[\hl{X},\vl{Y}],\hl{Z}]+\VV[\hl{Y},(\JJ+\FF)[\hl{X},\vl{Z}]]=\VV([\hl{X},\hl{[Y,Z]}]-[\hl{X},[\hl{Y},\hl{Z}]]+[\FF[\hl{X},\vl{Y}],\hl{Z}]+[\hl{Y},\FF[\hl{X},\vl{Z}]])$.
\end{center}
By the Jacobi identity, $\underset{(X,Y,Z)}{\mathfrak{S}}[\hl{X},[\hl{Y},\hl{Z}]]=0$. Hence, applying the vanishing of the torsion of $\HH$ we obtain:
\begin{center}
$\displaystyle\underset{(X,Y,Z)}{\mathfrak{S}}(\hl{\nabla}\RR)(\kal{X},\kal{Y},\kal{Z})=\VV\underset{(X,Y,Z)}{\mathfrak{S}}[\hl{X},\hl{[Y,Z]}]+$\\
$\displaystyle\VV([\FF[\hl{X},\vl{Y}],\hl{Z}]+[\FF[\hl{Y},\vl{Z}],\hl{X}]+[\FF[\hl{Z},\vl{X}],\hl{Y}])
+$\\
$\displaystyle\VV([\hl{Y},\FF[\hl{X},\vl{Z}]]+[\hl{X},\FF[\hl{Z},\vl{Y}]]+[\hl{Z},\FF[\hl{Y},\vl{X}]])=\VV\underset{(X,Y,Z)}{\mathfrak{S}}[\hl{X},\hl{[Y,Z]}]+\VV[\FF([\hl{X},\vl{Y}]-[\hl{Y},\vl{X}]),\hl{Z}]+\VV[\FF([\hl{Y},\vl{Z}]-[\hl{Z},\vl{Y}]),\hl{X}]+\VV[\FF([\hl{Z},\vl{X}]-[\hl{X},\vl{Z}]),\hl{Y}]=\VV([\hl{X},\hl{[Y,Z]}]+[\hl{[Y,Z]},\hl{X}]+[\hl{Y},\hl{[Z,X]}]+[\hl{[Z,X]},\hl{Y}]+[\hl{Z},\hl{[X,Y]}]+[\hl{[X,Y]},\hl{Z}])=0$.
\end{center}
\end{proof}

\chapter{The Berwald curvature of an Ehresmann connection} \label{ch3}

In this section we specify an Ehresmann connection $\HH$ over $M$, and consider the Berwald derivative $\nabla=(\hl{\nabla},\vl{\nabla})$ induced by $\HH$. As in the general theory, we denote by $R^{\nabla}$ the curvature tensor of $\nabla$. By the \textit{Berwald curvature} of $\HH$ we mean the type $\binom{1}{3}$ tensor field $\BB$ along $\tauk$ given by
\begin{gather} \label{bdef}
\BB(\hul{X},\hul{Y})\hul{Z}:=R^{\nabla}(\ii\hul{X},\HH\hul{Y})\hul{Z}=\nabla_{\ii\hul{X}}\nabla_{\HH\hul{Y}}\hul{Z}-\nabla_{\HH\hul{Y}}\nabla_{\ii\hul{X}}\hul{Z}-\nabla_{[\ii\hul{X},\HH\hul{Y}]}\hul{Z},
\end{gather}
where $\hul{X}$, $\hul{Y}$, $\hul{Z}$ are vector fields along $\tauk$.

\begin{lemma}\label{lemma31}
For any vector fields $X$, $Y$,$Z$ on $M$ we have
\begin{gather} \label{blemma1}
\BB(\kal{X},\kal{Y})\kal{Z}=\jj[\vl{X},\FF[\hl{Y},\vl{Z}]]=(\vl{\nabla}\hl{\nabla}\kal{Z})(\kal{X},\kal{Y}),
\end{gather}
or, equivalently,
\begin{gather} \label{blemma2}
\ii\BB(\kal{X},\kal{Y})\kal{Z}=[\vl{X},[\hl{Y},\vl{Z}]]=[[\vl{X},\hl{Y}],\vl{Z}].
\end{gather}
\end{lemma}

\begin{proof}
\begin{center}
$\displaystyle\BB(\kal{X},\kal{Y})\kal{Z}:=R^{\nabla}(\vl{X},\hl{Y})\kal{Z}=\nabla_{\vl{X}}\nabla_{\hl{Y}}\kal{Z}-\nabla_{\hl{Y}}\nabla_{\vl{X}}\kal{Z}-\nabla_{[\vl{X},\hl{Y}]}\kal{Z}=\nabla_{\vl{X}}(\VV[\hl{Y},\vl{Z}])=\jj[\vl{X},\HH\circ\VV[\hl{Y},\vl{Z}]]=$\\
$\displaystyle\jj[\vl{X},(\mathbf{F}+\mathbf{J})[\hl{Y},\vl{Z}]]=\jj[\vl{X},\FF[\hl{Y},\vl{Z}]]$.
\end{center}
On the other hand,
\begin{center}
$\displaystyle(\vl{\nabla}\hl{\nabla}\kal{Z})(\kal{X},\kal{Y})=\nabla_{\vl{X}}(\hl{\nabla}\kal{Z})(\kal{Y})=\nabla_{\vl{X}}\nabla_{\hl{Y}}\kal{Z}=\jj[\vl{X},\HH\nabla_{\hl{Y}}\kal{Z}]=\jj[\vl{X},\HH\circ\VV[\hl{Y},\vl{Z}]]=\jj[\vl{X},\FF[\hl{Y},\vl{Z}]]$,
\end{center}
thus relations (\ref{blemma1}) hold. To prove the remainder, observe that
\begin{center}
$\displaystyle 0=[\mathbf{J},\vl{X}](\FF[\hl{Y},\vl{Z}])=[\vv[\hl{Y},\vl{Z}],\vl{X}]-\mathbf{J}[\FF[\hl{Y},\vl{Z}],\vl{X}]=-[\vl{X},[\hl{Y},\vl{Z}]]+\mathbf{J}[\vl{X},\FF[\hl{Y},\vl{Z}]]$,
\end{center}
and hence
\begin{center}
$\displaystyle\ii\BB(\kal{X},\kal{Y})\kal{Z}\overset{(\ref{blemma1})}{=}\mathbf{J}[\vl{X},\FF[\hl{Y},\vl{Z}]]=[\vl{X},[\hl{Y},\vl{Z}]]$.
\end{center}
Finally, using the Jacobi identity we obtain that $[\vl{X},[\hl{Y},\vl{Z}]]=[[\vl{X},\hl{Y}],\vl{Z}]$.
\end{proof}

\begin{lemma}\label{lem4}
The Berwald curvature of an Ehresmann connection is symmetric in its first and third variable. If the torsion of the Ehresmann connection vanishes, then the Berwald curvature is totally symmetric.
\end{lemma}

\begin{proof}
Keeping the notation of the previous lemma, $\ii\BB(\kal{X},\kal{Y})\kal{Z}=[\vl{X},[\hl{Y},\vl{Z}]]$. Since
\begin{center}
$\displaystyle 0=[\vl{X},[\hl{Y},\vl{Z}]]+[\hl{Y},[\vl{Z},\vl{X}]]+[\vl{Z},[\vl{X},\hl{Y}]]=[\vl{X},[\hl{Y},\vl{Z}]]-[\vl{Z},[\hl{Y},\vl{X}]]=\ii\BB(\kal{X},\kal{Y})\kal{Z}-\ii\BB(\kal{Z},\kal{Y})\kal{X}$,
\end{center}
$\BB$ is indeed symmetric in its first and third variable. If the Ehresmann connection has vanishing torsion, then
\begin{center}
$\displaystyle [\hl{X},\vl{Z}]-[\hl{Z},\vl{X}]-[X,Z]^{\mathsf{v}}=0\textrm{  }\textrm{  }(X,Z\in\mathfrak{X}(M))$,
\end{center}
and hence
\begin{center}
$\displaystyle\ii\BB(\kal{Y},\kal{X})\kal{Z}=[\vl{Y},[\hl{X},\vl{Z}]]=[\vl{Y},[\hl{Z},\vl{X}]]+[\vl{Y},[X,Z]^{\mathsf{v}}]=$\\
$\displaystyle[\vl{Y},[\hl{Z},\vl{X}]]=\ii\BB(\kal{Y},\kal{Z})\kal{X}$.
\end{center}
Thus, if the torsion vanishes,
\begin{center}
$\displaystyle\ii\BB(\kal{X},\kal{Y})\kal{Z}=\ii\BB(\kal{Z},\kal{Y})\kal{X}=\ii\BB(\kal{Z},\kal{X})\kal{Y}=$\\
$\displaystyle\ii\BB(\kal{Y},\kal{X})\kal{Z}=\ii\BB(\kal{Y},\kal{Z})\kal{X}=\ii\BB(\kal{X},\kal{Z})\kal{Y}$.
\end{center}
\end{proof}

\begin{lemma}
The tension and the Berwald curvature of an Ehresmann connection are related by
\begin{gather} \label{tenzberw}
\BB(\kal{X},\kal{Y})\delta=\vl{\nabla}\tt(\kal{X},\kal{Y})\textrm{  ;  }X,Y\in\modx{M}.
\end{gather}
\end{lemma}

\begin{proof}
By an important identity, due to J. Grifone, for any vector field $\xi$ on $TM$ we have 
\begin{gather}\label{grif13}
\JJ[\JJ\xi,S]=\JJ\xi,
\end{gather}
where $S$ is an arbitrary semispray over $M$ (\cite{Grif},\cite{SzGy}). Since $\HH\circ\delta$ is a semispray over $M$, this implies that
\begin{center}
$\displaystyle\nabla_{\ii\hul{X}}\delta=\jj[\ii\hul{X},\HH\circ\delta]=\hul{X}$ ; $\hul{X}\in\modx{\tau}$.
\end{center}
So we obtain
\begin{center}
$\displaystyle\BB(\kal{X},\kal{Y})\delta=R^{\nabla}(\vl{X},\hl{Y})\delta=\nabla_{\vl{X}}\nabla_{\hl{Y}}\delta-\nabla_{\hl{Y}}\nabla_{\vl{X}}\delta-\nabla_{[\vl{X},\hl{Y}]}\delta=\nabla_{\vl{X}}(\tt(\kal{Y}))-\nabla_{\hl{Y}}\kal{X}-\VV[\vl{X},\hl{Y}]=\nabla_{\vl{X}}(\tt(\kal{Y}))-\VV[\hl{Y},\vl{X}]+\VV[\hl{Y},\vl{X}]=(\nabla_{\vl{X}}\tt)(\kal{Y})=(\vl{\nabla}\tt)(\kal{X},\kal{Y})$,
\end{center}
as was to be proved.
\end{proof}

\begin{coro}\label{lem5}
If the torsion and the vertical differential of the tension of an Ehresmann connection vanishes, then its Berwald curvature has the property
\begin{gather}\label{lem5kepl}
\delta\in\left\{\hul{X},\hul{Y},\hul{Z}\right\}\Rightarrow \BB(\hul{X},\hul{Y})\hul{Z}=0.
\end{gather}
\end{coro}
\mbox{}\hfill\tiny$\square$\normalsize \bigskip

\begin{lemma}\label{lemma35}
The Berwald curvature of a homogeneous Ehresmann connection is homogeneous of degree $-1$, i.e., $$\vl{\nabla}_{\delta}\BB=\nabla_C\BB=-\BB.$$
\end{lemma}

\begin{proof}
Using the first relation in (\ref{partic2}), the Jacobi identity (repeatedly) and the homogeneity of $\HH$, for any vector fields $X$, $Y$, $Z$ on $M$ we get
\begin{center}
$\displaystyle\ii\left(\nabla_C\BB\right)(\kal{X},\kal{Y},\kal{Z})=\ii\nabla_C(\BB(\kal{X},\kal{Y})\kal{Z})=\JJ[C,\HH\BB(\kal{X},\kal{Y})\kal{Z}]=[\JJ,C]\HH\BB(\kal{X},\kal{Y})\kal{Z}-[\ii\BB(\kal{X},\kal{Y})\kal{Z},C]=\ii\BB(\kal{X},\kal{Y})\kal{Z}-[[\vl{X},[\hl{Y},\vl{Z}]],C]=\ii\BB(\kal{X},\kal{Y})\kal{Z}+[[[\hl{Y},\vl{Z}],C],\vl{X}]+[[C,\vl{X}],[\hl{Y},\vl{Z}]]=$\\
$[[[\hl{Y},\vl{Z}],C],\vl{X}]=-[[[\vl{Z},C],\hl{Y}],\vl{X}]+[[C,\hl{Y}],\vl{Z}],\vl{X}]=$\\
$-[[\vl{Z},\hl{Y}],\vl{X}]=-[\vl{X},[\hl{Y},\vl{Z}]]=-\ii\BB(\kal{X},\kal{Y})\kal{Z}$.
\end{center}
This proves the lemma.
\end{proof}

\begin{lemma}\textup{(vh-Ricci formulae for functions and sections).}
If $F$ is a smooth function on $\Tk M$ and $\hul{Z}$ is a section along $\tauk$, then for any sections $\hul{X}$, $\hul{Y}$ in $\szel{\pik}$ we have
\begin{gather}
\vl{\nabla}\hl{\nabla}F(\hul{X},\hul{Y})=\hl{\nabla}\vl{\nabla}F(\hul{Y},\hul{X});\label{ric1}\\
\vl{\nabla}\hl{\nabla}\hul{Z}(\hul{X},\hul{Y})-\hl{\nabla}\vl{\nabla}\hul{Z}(\hul{Y},\hul{X})=\BB(\hul{X},\hul{Y})\hul{Z}.\label{ric2}
\end{gather}
\end{lemma}

\begin{proof}
The expression on the left-hand side of (\ref{ric1}) is
\begin{center}
$\displaystyle\vl{\nabla}\hl{\nabla}F(\hul{X},\hul{Y})=(\nabla_{\ii\hul{X}}\hl{\nabla}F)(\hul{Y})=(\ii\hul{X})(\HH\hul{Y})F-\hl{\nabla}F(\nabla_{\ii\hul{X}}\hul{Y})=(\ii\hul{X})(\HH\hul{Y})F-(\HH\nabla_{\ii\hul{X}}\hul{Y})F=(\ii\hul{X})(\HH\hul{Y})F-\hh[\ii\hul{X},\HH\hul{Y}]F$.
\end{center}
The right-hand side of (\ref{ric1}) can be written in the form
\begin{center}
$\displaystyle\hl{\nabla}\vl{\nabla}F(\hul{Y},\hul{X})=(\nabla_{\HH\hul{Y}}\vl{\nabla}F)(\hul{X})=(\HH\hul{Y})(\ii\hul{X})F-\vl{\nabla}F(\nabla_{\HH\hul{Y}}\hul{X})=(\HH\hul{Y})(\ii\hul{X})F-(\ii\nabla_{\HH\hul{Y}}\hul{X})F=(\HH\hul{Y})(\ii\hul{X})F-\vv[\HH\hul{Y},\ii\hul{X}]F$,
\end{center}
so their difference is
\begin{center}
$\displaystyle[\ii\hul{X},\HH\hul{Y}]F+\vv[\HH\hul{Y},\ii\hul{X}]F-\hh[\ii\hul{X},\HH\hul{Y}]F=0$.
\end{center}
This proves relation (\ref{ric1}). Relation (\ref{ric2}) may be checked by a similar calculation: We have, on the one hand,
\begin{center}
$\displaystyle\vl{\nabla}\hl{\nabla}\hul{Z}(\hul{X},\hul{Y})=\nabla_{\ii\hul{X}}\nabla_{\HH\hul{Y}}\hul{Z}-\nabla_{\HH\nabla_{\ii\hul{X}}\hul{Y}}\hul{Z}$.
\end{center}
On the other hand,
\begin{center}
$\displaystyle\hl{\nabla}\vl{\nabla}\hul{Z}(\hul{Y},\hul{X})=\nabla_{\HH\hul{Y}}\nabla_{\ii\hul{X}}\hul{Z}-\nabla_{\ii\nabla_{\HH\hul{Y}}\hul{X}}\hul{Z}$.
\end{center}
Since $\HH\nabla_{\ii\hul{X}}\hul{Y}-\ii\nabla_{\HH\hul{Y}}\hul{X}=\hh[\ii\hul{X},\HH\hul{Y}]-\vv[\HH\hul{Y},\ii\hul{X}]=[\ii\hul{X},\HH\hul{Y}]$, it follows that the difference of the left-hand sides is indeed $\BB(\hul{X},\hul{Y})\hul{Z}$.
\end{proof}

\begin{lemma}\textup{(vh-Ricci formula for covariant tensors).}
Let $\AA\in\tenz{0}{s}{\pik}$, $s\geq 1$. For any sections $\hul{X},\hul{Y},\hul{Z_1},\dots\hul{Z_s}$ along $\tauk$ we have
\begin{gather}\label{ric34}
\vl{\nabla}\hl{\nabla}\AA(\hul{X},\hul{Y},\hul{Z_1},\dots\hul{Z_s})-\hl{\nabla}\vl{\nabla}\AA(\hul{Y},\hul{X},\hul{Z_1},\dots\hul{Z_s})=-\sum_{i=1}^s\AA(\hul{Z_1},\dots,\BB(\hul{X},\hul{Y})\hul{Z_i},\dots,\hul{Z_s}).
\end{gather}
\end{lemma}

\begin{proof}
For brevity, we sketch the argument only for a type $\binom{0}{2}$ tensor $\AA$. It may easily be shown that the left-hand side of (\ref{ric34}) is tensorial in its first two variables (actually, in all variables), so we may chose in the role of $\hul{X}$ and $\hul{Y}$ basic vector fields $\kal{X}$, $\kal{Y}$. Then
\begin{center}
$\displaystyle\vl{\nabla}\hl{\nabla}\AA(\kal{X},\kal{Y},\hul{Z_1},\hul{Z_2})=\vl{X}(\hl{Y}\AA(\hul{Z_1},\hul{Z_2}))-\vl{X}\AA(\nabla_{\hl{Y}}\hul{Z_1},\hul{Z_2})-\vl{X}\AA(\hul{Z_1},\nabla_{\hl{Y}}\hul{Z_2})-\hl{Y}\AA(\nabla_{\vl{X}}\hul{Z_1},\hul{Z_2})+\AA(\nabla_{\hl{Y}}\nabla_{\vl{X}}\hul{Z_1},\hul{Z_2})+\AA(\nabla_{\vl{X}}\hul{Z_1},\nabla_{\hl{Y}}\hul{Z_2})-\hl{Y}\AA(\hul{Z_1},\nabla_{\vl{X}}\hul{Z_2})+\AA(\nabla_{\hl{Y}}\hul{Z_1},\nabla_{\vl{X}}\hul{Z_2})+\AA(\hul{Z_1},\nabla_{\hl{Y}}\nabla_{\vl{X}}\hul{Z_2})$;
\end{center}

\begin{center}
$\displaystyle\hl{\nabla}\vl{\nabla}\AA(\kal{Y},\kal{X},\hul{Z_1},\hul{Z_2})=\hl{Y}(\vl{X}\AA(\hul{Z_1},\hul{Z_2}))-\hl{Y}\AA(\nabla_{\vl{X}}\hul{Z_1},\hul{Z_2})-\hl{Y}\AA(\hul{Z_1},\nabla_{\vl{X}}\hul{Z_2})-[\hl{Y},\vl{X}]\AA(\hul{Z_1},\hul{Z_2})+\AA(\nabla_{[\hl{Y},\vl{X}]}\hul{Z_1},\hul{Z_2})+\AA(\hul{Z_1},\nabla_{[\hl{Y},\vl{X}]}\hul{Z_2})-\vl{X}\AA(\nabla_{\hl{Y}}\hul{Z_1},\hul{Z_2})+\AA(\nabla_{\vl{X}}\nabla_{\hl{Y}}\hul{Z_1},\hul{Z_2})+\AA(\nabla_{\hl{Y}}\hul{Z_1},\nabla_{\vl{X}}\hul{Z_2})-\vl{X}\AA(\hul{Z_1},\nabla_{\hl{Y}}\hul{Z_2})+\AA(\nabla_{\vl{X}}\hul{Z_1},\nabla_{\hl{Y}}\hul{Z_2})+\AA(\hul{Z_1},\nabla_{\vl{X}}\nabla_{\hl{Y}}\hul{Z_2})$,
\end{center}
and after substraction we get
\begin{center}
$\displaystyle\vl{\nabla}\hl{\nabla}\AA(\kal{X},\kal{Y},\hul{Z_1},\hul{Z_2})-\hl{\nabla}\vl{\nabla}\AA(\kal{Y},\kal{X},\hul{Z_1},\hul{Z_2})=\AA(\nabla_{\hl{Y}}\nabla_{\vl{X}}\hul{Z_1}-\nabla_{\vl{X}}\nabla_{\hl{Y}}\hul{Z_1}-\nabla_{[\hl{Y},\vl{X}]}\hul{Z_1},\hul{Z_2})+$\\
$\AA(\hul{Z_1},\nabla_{\hl{Y}}\nabla_{\vl{X}}\hul{Z_2}-\nabla_{\vl{X}}\nabla_{\hl{Y}}\hul{Z_2}-\nabla_{[\hl{Y},\vl{X}]}\hul{Z_2})=$\\
$-\AA(\BB(\kal{X},\kal{Y})\hul{Z_1},\hul{Z_2})-\AA(\hul{Z_1},\BB(\kal{X},\kal{Y})\hul{Z_2})$.
\end{center}
\end{proof}

\begin{propo}
An Ehresmann connection $\HH$ over $M$ has vanishing Berwald curvature, if and only if, there exists a (necessarily unique) covariant derivative operator $D$ on the base manifold $M$ such that for any vector fields $X$, $Y$ on $M$ we have
\begin{gather} \label{propo3418}
[\hl{X},\vl{Y}]=\vl{\left(D_XY\right)}.
\end{gather}
\end{propo}

\begin{proof}
The \textit{sufficiency} of the condition is immediate: if there exists a covariant derivatve operator $D$ on $M$ satisfying (\ref{propo3418}), then for all vector fields $X$, $Y$, $Z$ on $M$ we have $$\ii\BB(\kal{X},\kal{Y},\kal{Z})\overset{(\textrm{\ref{blemma2}})}{=}[\vl{X},[\hl{Y},\vl{Z}]]\overset{(\textrm{\ref{propo3418}})}{=}[\vl{X},\vl{\left(D_YZ\right)}]=0,$$ since the Lie bracket of vertically lifted vector fields vanishes.

\textit{Conversely}, if $\HH$ has vanishing Berwald curvature, then for all vector fields $X,Y,Z$ in $\modx{M}$, $$[\vl{X},[\hl{Y},\vl{Z}]]=0.$$ This implies that $[\hl{Y},\vl{Z}]$ is a vertical lift, so we may define a map $$D:\modx{M}\times\modx{M}\rightarrow\modx{M}\textrm{ , }(Y,Z)\mapsto D_YZ$$ by $$\vl{\left(D_YZ\right)}:=[\hl{Y},\vl{Z}].$$ It is easy to check, that $D$ is a covariant derivative operator on $M$. For example, if $f\in\modc{M}$, then
\begin{center}
$\displaystyle\vl{\left(D_YfZ\right)}:=[\hl{Y},\vl{(fZ)}]=[\hl{Y},\vl{f}\vl{Z}]=(\hl{Y}\vl{f})\vl{Z}+\vl{f}[\hl{Y},\vl{Z}]=\vl{(Yf)}\vl{Z}+\vl{\left(fD_YZ\right)}=\vl{\left((Yf)Z+fD_YZ\right)},$
\end{center}
hence $$D_YfZ=(Yf)Z+fD_YZ.$$ Similarly,
\begin{center}
$\displaystyle\vl{\left(D_{fY}Z\right)}:=\left[\hl{(fY)},\vl{Z}\right]=[\vl{f}\hl{Y},\vl{Z}]=$\\
$\displaystyle-(\vl{Z}\vl{f})\hl{Y}+\vl{f}[\hl{Y},\vl{Z}]=\vl{f}[\hl{Y},\vl{Z}]=\vl{\left(fD_YZ\right)},$
\end{center}
which implies that $$D_{fY}Z=fD_YZ.$$

The other rules are immediate consequences of the definition of $D$.
\end{proof}

Relation (\ref{propo3418}) can also be written in the form $$\hl{\nabla}_{\kal{X}}\kal{Y}=\kal{D_XY},$$ so it is reasonable to call an Ehresmann connection \textit{h-basic} or briefly \textit{basic}, if it has vanishing Berwald curvature, since in this case the Christoffel symbols of the h-covariant derivative do not depend on the direction. More generally, we say that an Ehresmann connection is \textit{weakly Berwald} if the trace of its Berwald curvature vanishes.

We shall use similar terminology for sprays. A spray will be called \textit{Berwald}, if its associated Ehresmann connection has vanishing Berwald curvature, and will be called \textit{weakly Berwald} if the Berwald curvature of its associated Ehresmann connection is traceless.

\chapter{The affine curvature of an Ehresmann connection} \label{fej4}

We continue to assume that an Ehresmann connection $\HH$ is specified over $M$, and consider the Berwald derivative $\nabla=(\hl{\nabla},\vl{\nabla})$ determined by $\HH$. By the \textit{affine curvature} of $\HH$ we mean the type $\binom{1}{3}$ tensor $\Hh$ along $\tauk$ given by $$\Hh(\hul{X},\hul{Y})\hul{Z}:=R^{\nabla}(\HH\hul{X},\HH\hul{Y})\hul{Z}\textrm{ ; }\hul{X},\hul{Y},\hul{Z}\in\textrm{Sec}(\pik).$$ This tensor was essentially introduced by L. Berwald (\cite{Berwald}) in terms of the classical tensor calculus and in the more specific context of an Ehresmann connection associated to a spray. So we think that it is appropriate to preserve his terminology. To indicate the meaning of the affine curvature, we remark, that if an Ehresmann connection is basic with base covariant derivative $D$ on $M$, then its affine curvature may be indentified with the curvature of $D$. More precisely, we have $$\ii\Hh(\kal{X},\kal{Y})\kal{Z}=\vl{(R^D(X,Y)Z)}\textrm{ ; }X,Y,Z\in\modx{M}.$$ According to Z. Shen's usage, we say that an Ehresmann connection is $\textit{R-quadratic}$ if $\vl{\nabla}\mathbf{H}=0$, i.e., the affine curvature ``depends only on the position".

Now we formulate and prove in our setting some basic relations found by Berwald. The first observation, roughly speaking, is that the affine curvature is just the vertical differential of the curvature of $\HH$. The exact relation between $\Hh$ and $\RR$ is formulated in

\begin{lemma}\label{lemma41}
For all $\hul{X},\hul{Y},\hul{Z}\in\szel{\pik}$,
\begin{gather}\label{affcurv19}
\Hh(\hul{X},\hul{Y})\hul{Z}=\vl{\nabla}\RR(\hul{Z},\hul{X},\hul{Y}).
\end{gather}
\end{lemma}

\begin{proof}
It is enough to check that (\ref{affcurv19}) is true for basic vector fields $\kal{X},\kal{Y},\kal{Z}$ along $\tauk$. Then, on the one hand,
\begin{center}
$\displaystyle\Hh(\kal{X},\kal{Y})\kal{Z}:=\nabla_{\hl{X}}\nabla_{\hl{Y}}\kal{Z}-\nabla_{\hl{Y}}\nabla_{\hl{X}}\kal{Z}-\nabla_{[\hl{X},\hl{Y}]}\kal{Z}\overset{\textrm{(\ref{15})}}{=}\nabla_{\hl{X}}\VV[\hl{Y},\vl{Z}]-\nabla_{\hl{Y}}\VV[\hl{X},\vl{Z}]-\nabla_{\hh[\hl{X},\hl{Y}]}\kal{Z}\overset{\textrm{(\ref{13}),(\ref{jh4})}}{=}\VV[\hl{X},[\hl{Y},\vl{Z}]]-\VV[\hl{Y},[\hl{X},\vl{Z}]]-\VV[\hl{[X,Y]},\vl{Z}]=\VV([\hl{X},[\hl{Y},\vl{Z}]]+[\hl{Y},[\vl{Z},\hl{X}]]+[\vl{Z},\hl{[X,Y]}])=\VV([-\vl{Z},[\hl{X},\hl{Y}]]+[\vl{Z},\hl{[X,Y]}])=\VV[\vl{Z},\hl{[X,Y]}-[\hl{X},\hl{Y}]]=\VV[\vl{Z},\ii\RR(\kal{X},\kal{Y})].$
\end{center}
On the other hand,
\begin{center}
$\displaystyle\vl{\nabla}\RR(\kal{Z},\kal{X},\kal{Y})=\left(\nabla_{\vl{Z}}\RR\right)(\kal{X},\kal{Y})\overset{\textrm{(\ref{15})}}{=}\nabla_{\vl{Z}}(\RR(\kal{X},\kal{Y}))\overset{\textrm{(\ref{3})}}{=}\jj[\vl{Z},\HH\RR(\kal{X},\kal{Y})]=-\jj[\vl{Z},\HH\VV[\hl{X},\hl{Y}]]=-\jj[\vl{Z},\FF[\hl{X},\hl{Y}]+\JJ[\hl{X},\hl{Y}]]=-\jj[\vl{Z},\FF[\hl{X},\hl{Y}]].$
\end{center}
Now, taking into account the second relation in (\ref{partic2}),
\begin{center}
$\displaystyle 0=[\JJ,\vl{Z}]\FF[\hl{X},\hl{Y}]=[\JJ\FF[\hl{X},\hl{Y}],\vl{Z}]-\JJ[\FF[\hl{X},\hl{Y}],\vl{Z}]=[\vv[\hl{X},\hl{Y}],\vl{Z}]-\JJ[\FF[\hl{X},\hl{Y}],\vl{Z}]=[\vl{Z},\ii\RR(\kal{X},\kal{Y})]+\JJ[\vl{Z},\FF[\hl{X},\hl{Y}]]=\ii\VV[\vl{Z},\ii\RR(\kal{X},\kal{Y})]+\JJ[\vl{Z},\FF[\hl{X},\hl{Y}]]$,
\end{center}
hence
\begin{center}
$\displaystyle\vl{\nabla}\RR(\kal{Z},\kal{X},\kal{Y})=-\jj[\vl{Z},\FF[\hl{X},\hl{Y}]]=\VV[\vl{Z},\ii\RR(\kal{X},\kal{Y})]=\Hh(\kal{X},\kal{Y})\kal{Z}$.
\end{center}
\end{proof}

\begin{lemma}\label{homehrlemm421}
If $\HH$ is a homogeneous Ehresmann connection, then the curvature of $\HH$ may be reproduced from the affine curvature, namely, we have
\begin{gather}\label{homehrlemm42}
\RR(\hul{X},\hul{Y})=\Hh(\hul{X},\hul{Y})\delta\textrm{ ; }\hul{X},\hul{Y}\in\szel{\pik}.
\end{gather}
\end{lemma}

\begin{proof}
\begin{center}
$\displaystyle\Hh(\hul{X},\hul{Y})\delta:=\nabla_{\HH\hul{X}}\nabla_{\HH\hul{Y}}\delta-\nabla_{\HH\hul{Y}}\nabla_{\HH\hul{X}}\delta-\nabla_{[\HH\hul{X},\HH\hul{Y}]}\delta\overset{\textrm{(\ref{14})}}{=}$\\
$\displaystyle -\nabla_{\vv[\HH\hul{X},\HH\hul{Y}]}\delta=-\jj[\vv[\HH\hul{X},\HH\hul{Y}],\HH\circ\delta]$.
\end{center}
Since $\HH\circ\delta=S$ is a spray, we obtain that
\begin{center}
$\displaystyle\Hh(\hul{X},\hul{Y})\delta=-\ii^{-1}\JJ[\JJ\FF[\HH\hul{X},\HH\hul{Y}],S]\overset{\textrm{(\ref{grif13})}}{=}-\ii^{-1}\JJ\FF[\HH\hul{X},\HH\hul{Y}]=$\\
$\displaystyle-\VV[\HH\hul{X},\HH\hul{Y}]=\RR(\hul{X},\hul{Y})$,
\end{center}
as we claimed.
\end{proof}

\begin{coro}\label{coro43}
If an Ehresmann connection is homogeneous, then its curvature $\RR$ is homogeneous of degree 1, i.e., $\nabla_C\RR=\RR$.
\end{coro}

\begin{proof}
For any vector fields $X$, $Y$ on $M$,
\begin{center}
$\displaystyle(\nabla_C\RR)(\kal{X},\kal{Y})=\vl{\nabla}\RR(\delta,\kal{X},\kal{Y})\overset{\textrm{(\ref{affcurv19})}}{=}\Hh(\kal{X},\kal{Y})\delta\overset{\textrm{(\ref{homehrlemm42})}}{=}\RR(\kal{X},\kal{Y})$.
\end{center}
\end{proof}

\begin{lemma}
The affine curvature of a homogeneous Ehresmann connection is homogeneous of degree zero, i.e., $\nabla_C\Hh=0$.
\end{lemma}

\begin{proof}
By a similar technique as above, we have for any vector fields $X,Y,Z$ on $M$:
\begin{center}
$\displaystyle\ii(\nabla_C\Hh)(\kal{X},\kal{Y},\kal{Z})=\ii\nabla_C(\Hh(\kal{X},\kal{Y})\kal{Z})\overset{\textrm{(\ref{affcurv19})}}{=}\ii\nabla_C\nabla_{\vl{Z}}(\RR(\kal{X},\kal{Y}))=\ii\nabla_C\jj[\vl{Z},\HH\RR(\kal{X},\kal{Y})]=\JJ[C,\hh[\vl{Z},\HH\RR(\kal{X},\kal{Y})]]=\JJ[C,[\vl{Z},\HH\RR(\kal{X},\kal{Y})]]=-\JJ([\vl{Z},[\HH\RR(\kal{X},\kal{Y}),C]]+[\HH\RR(\kal{X},\kal{Y}),[C,\vl{Z}]])=-\JJ([\hh[C,\HH\RR(\kal{X},\kal{Y})],\vl{Z}]-[\HH\RR(\kal{X},\kal{Y}),\vl{Z}])=0$,
\end{center}
since
\begin{center}
$\displaystyle\HH\RR(\kal{X},\kal{Y})\overset{\textrm{(\ref{coro43})}}{=}\HH\nabla_C(\RR(\kal{X},\kal{Y}))=\HH\jj[C,\HH\RR(\kal{X},\kal{Y})]=\hh[C,\HH\RR(\kal{X},\kal{Y})]$.
\end{center}
\end{proof}

\begin{lemma}\label{bian45} \textup{(Bianchi identities).}
Let $\HH$ be a torsion-free Ehresmann connection, and let $(\vl{\nabla},\hl{\nabla})$ be the Berwald derivative determinded by $\HH$. For any sections $\hul{X}$, $\hul{Y}$, $\hul{Z}$, $\hul{U}$ in $\szel{\pik}$ we have
\begin{gather}\label{Bi}
\underset{(\hul{X},\hul{Y},\hul{Z})}{\mathfrak{S}}\Hh(\hul{X},\hul{Y})\hul{Z}=0
\end{gather}
and
\begin{gather}\label{bian2}
\vl{\nabla}\mathbf{H}(\hul{X},\hul{Y},\hul{Z},\hul{U})-\hl{\nabla}\mathbf{B}(\hul{Y},\hul{X},\hul{Z},\hul{U})+\hl{\nabla}\mathbf{B}(\hul{Z},\hul{X},\hul{Y},\hul{U})=0.
\end{gather}
\end{lemma}

\begin{proof}
Since the expressions on the left-hand sides are tensorial in each variables, we may use basic vector fields $\kal{X}$, $\kal{Y}$, $\kal{Z}$, $\kal{U}$ in our calculations.

First we show the cyclicity property (\ref{Bi}) of $\Hh$. Taking into account the first partial result in the proof of \ref{lemma41}, we get
\begin{center}
$\displaystyle\ii\Hh(\kal{X},\kal{Y})\kal{Z}=[\vl{Z},\ii\RR(\kal{X},\kal{Y})]=[\vl{Z},-\vv[\hl{X},\hl{Y}]]=[\vl{Z},\hl{[X,Y]}]-[\vl{Z},[\hl{X},\hl{Y}]]=[\vl{Z},\hl{[X,Y]}]+[\hl{X},[\hl{Y},\vl{Z}]]+[\hl{Y},[\vl{Z},\hl{X}]]$.
\end{center}
In the same way,
\begin{center}
$\displaystyle\ii\Hh(\kal{Y},\kal{Z})\kal{X}=[\vl{X},\hl{[Y,Z]}]+[\hl{Y},[\hl{Z},\vl{X}]]+[\hl{Z},[\vl{X},\hl{Y}]]$,
$\displaystyle\ii\Hh(\kal{Z},\kal{X})\kal{Y}=[\vl{Y},\hl{[Z,X]}]+[\hl{Z},[\hl{X},\vl{Y}]]+[\hl{X},[\vl{Y},\hl{Z}]]$.
\end{center}
Now adding these three relations and applying the vanishing of the torsion of $\HH$ repeatedly, we obtain
\begin{center}
$\displaystyle\ii(\Hh(\kal{X},\kal{Y})\kal{Z}+\Hh(\kal{Y},\kal{Z})\kal{X}+\Hh(\kal{Z},\kal{X})\kal{Y})=[\hl{X},[\hl{Y},\vl{Z}]-[\hl{Z},\vl{Y}]]+[\hl{Y},[\vl{Z},\hl{X}]-[\vl{X},\hl{Z}]]+[\hl{Z},[\hl{X},\vl{Y}]-[\hl{Y},\vl{X}]]+[\vl{X},\hl{[Y,Z]}]+[\vl{Y},\hl{[Z,X]}]+[\vl{Z},\hl{[X,Y]}]=[\hl{X},\vl{[Y,Z]}]-[\hl{[Y,Z]},\vl{X}]+[\hl{Y},\vl{[Z,X]}]-[\hl{[Z,X]},\vl{Y}]+[\hl{Z},\vl{[X,Y]}]-[\hl{[X,Y]},\vl{Z}]=\vl{\left([X,[Y,Z]]+[Y,[Z,X]]+[Z,[X,Y]]\right)}=0$,
\end{center}
which proves that the cyclic symmetrization of $\Hh$ is 0.

For (\ref{bian2}), we use the differential Bianchi identity proved in the context of general vector bundles in Chapter 1, (D). Applying (\ref{vbianchi}) to the curvature tensor $R^{\nabla}$ and the triplet $(\vl{X},\hl{Y},\hl{Z})$, we obtain:
\begin{center}
$\displaystyle 0=(\underset{(\vl{X},\hl{Y},\hl{Z})}{\mathfrak{S}}(\nabla_{\vl{X}}(R^{\nabla}(\hl{Y},\hl{Z}))-R^{\nabla}([\vl{X},\hl{Y}],\hl{Z})))(\kal{U})\overset{\textrm{(\ref{biastar})}}{=}\nabla_{\vl{X}}(R^{\nabla}(\hl{Y},\hl{Z})\kal{U})+\nabla_{\hl{Y}}(R^{\nabla}(\hl{Z},\vl{X})\kal{U})-R^{\nabla}(\hl{Z},\vl{X})\nabla_{\hl{Y}}\kal{U}+\nabla_{\hl{Z}}(R^{\nabla}(\vl{X},\hl{Y})\kal{U})-R^{\nabla}(\vl{X},\hl{Y})\nabla_{\hl{Z}}\kal{U}-R^{\nabla}([\vl{X},\hl{Y}],\hl{Z})\kal{U}-R^{\nabla}([\hl{Y},\hl{Z}],\vl{X})\kal{U}-R^{\nabla}([\hl{Z},\vl{X}],\hl{Y})\kal{U}=\nabla_{\vl{X}}(\Hh(\kal{Y},\kal{Z})\kal{U})-\nabla_{\hl{Y}}(\BB(\kal{X},\kal{Z})\kal{U})+\BB(\kal{X},\kal{Z})\nabla_{\hl{Y}}\kal{U}+\nabla_{\hl{Z}}(\BB(\kal{X},\kal{Y})\kal{U})-\BB(\kal{X},\kal{Y})\nabla_{\hl{Z}}\kal{U}+R^{\nabla}(\ii\nabla_{\hl{Y}}\kal{X},\hl{Z})\kal{U}+R^{\nabla}(\vl{X},[\hl{Y},\hl{Z}])\kal{U}-R^{\nabla}(\ii\nabla_{\hl{Z}}\kal{X},\hl{Y})\kal{U}=\vl{\nabla}\Hh(\kal{X},\kal{Y},\kal{Z},\kal{U})-\nabla_{\hl{Y}}(\BB(\kal{X},\kal{Z})\kal{U})+\BB(\nabla_{\hl{Y}}\kal{X},\kal{Z})\kal{U}+\BB(\kal{X},\kal{Z})\nabla_{\hl{Y}}\kal{U}+\nabla_{\hl{Z}}(\BB(\kal{X},\kal{Y})\kal{U})-\BB(\nabla_{\hl{Z}}\kal{X},\kal{Y})\kal{U}-\BB(\kal{X},\kal{Y})\nabla_{\hl{Z}}\kal{U}+\BB(\kal{X},\kal{[Y,Z]})\kal{U}$.
\end{center}
Since the torsion of $\HH$ vanishes,
\begin{center}
$\displaystyle\kal{[Y,Z]}=\VV\vl{[Y,Z]}=\VV[\hl{Y},\vl{Z}]-\VV[\hl{Z},\vl{Y}]=\nabla_{\hl{Y}}\kal{Z}-\nabla_{\hl{Z}}\kal{Y}$,
\end{center}
and hence
\begin{center}
$\displaystyle\BB(\kal{X},\kal{[Y,Z]})\kal{U}=\BB(\kal{X},\nabla_{\hl{Y}}\kal{Z})\kal{U}-\BB(\kal{X},\nabla_{\hl{Z}}\kal{Y})\kal{U}$.
\end{center}
Thus finally we obtain
\begin{center}
$\displaystyle 0=\vl{\nabla}\Hh(\kal{X},\kal{Y},\kal{Z},\kal{U})-\nabla_{\hl{Y}}(\BB(\kal{X},\kal{Z})\kal{U})+\BB(\nabla_{\hl{Y}}\kal{X},\kal{Z})\kal{U}+\BB(\kal{X},\nabla_{\hl{Y}}\kal{Z})\kal{U}+\BB(\kal{X},\kal{Z})\nabla_{\hl{Y}}\kal{U}+\nabla_{\hl{Z}}(\BB(\kal{X},\kal{Y})\kal{U})-\BB(\nabla_{\hl{Z}}\kal{X},\kal{Y})\kal{U}-\BB(\kal{X},\nabla_{\hl{Z}}\kal{Y})\kal{U}-\BB(\kal{X},\kal{Y})\nabla_{\hl{Z}}\kal{U}=\vl{\nabla}\Hh(\kal{X},\kal{Y},\kal{Z},\kal{U})-\hl{\nabla}\BB(\kal{Y},\kal{X},\kal{Z},\kal{U})+\hl{\nabla}\BB(\kal{Z},\kal{X},\kal{Y},\kal{U})$.
\end{center}
This concludes the proof.
\end{proof}

\textbf{Remark.} If the Berwald curvature is totally symmetric, then $\hl{\nabla}\BB$ is also totally symmetric in its last three variables. (This may be seen immediately.) So it follows that \textit{if an Ehresmann connection has vanishing torsion and hence totally symmetric Berwald tensor, then we also have}
\begin{gather}\label{bian3}
\vl{\nabla}\mathbf{H}(\hul{X},\hul{Y},\hul{Z},\hul{U})-\hl{\nabla}\mathbf{B}(\hul{Y},\hul{Z},\hul{X},\hul{U})+\hl{\nabla}\mathbf{B}(\hul{Z},\hul{Y},\hul{X},\hul{U})=0.
\end{gather}

\begin{lemma}\textup{(hh-Ricci formulae for sections and 1-forms).}
If $\HH$ is a torsion-free Ehresmann connection, then for any section $\hul{Z}\in\szel{\pik}$ and 1-form $\hul{\alpha}\in\tenz{0}{1}{\pik}$ we have
\begin{gather}\label{hhricci24}
\hl{\nabla}\hl{\nabla}\hul{Z}(\kal{X},\kal{Y})-\hl{\nabla}\hl{\nabla}\hul{Z}(\kal{Y},\kal{X})=\Hh(\kal{X},\kal{Y})\hul{Z}-\vl{\nabla}\hul{Z}(\RR(\kal{X},\kal{Y})),
\end{gather}
\begin{gather}\label{hhricci25}
\hl{\nabla}\hl{\nabla}\hul{\alpha}(\kal{X},\kal{Y},\kal{Z})-\hl{\nabla}\hl{\nabla}\hul{\alpha}(\kal{Y},\kal{X},\kal{Z})=-\hul{\alpha}(\Hh(\kal{X},\kal{Y})\kal{Z})-\vl{\nabla}\hul{\alpha}(\RR(\kal{X},\kal{Y}),\kal{Z}),
\end{gather}
($X,Y,Z\in\modx{M}$).
\end{lemma}

\begin{proof}
\begin{center}
$\displaystyle\hl{\nabla}\hl{\nabla}\hul{Z}(\kal{X},\kal{Y})=\nabla_{\hl{X}}(\hl{\nabla}\hul{Z})(\kal{Y})=\nabla_{\hl{X}}\nabla_{\hl{Y}}\hul{Z}-\hl{\nabla}\hul{Z}\left(\nabla_{\hl{X}}\kal{Y}\right)=\nabla_{\hl{X}}\nabla_{\hl{Y}}\hul{Z}-\nabla_{\HH\nabla_{\hl{X}}\kal{Y}}\hul{Z}=\nabla_{\hl{X}}\nabla_{\hl{Y}}\hul{Z}-\nabla_{\HH\VV[\hl{X},\vl{Y}]}\hul{Z}$,
\end{center}
and, similarly,
\begin{center}
$\displaystyle\hl{\nabla}\hl{\nabla}\hul{Z}(\kal{Y},\kal{X})=\nabla_{\hl{Y}}\nabla_{\hl{X}}\hul{Z}-\nabla_{\HH\VV[\hl{Y},\vl{X}]}\hul{Z}$.
\end{center}
Hence
\begin{center}
$\displaystyle\hl{\nabla}\hl{\nabla}\hul{Z}(\kal{X},\kal{Y})-\hl{\nabla}\hl{\nabla}\hul{Z}(\kal{Y},\kal{X})=\nabla_{\hl{X}}\nabla_{\hl{Y}}\hul{Z}-\nabla_{\hl{Y}}\nabla_{\hl{X}}\hul{Z}+\nabla_{\HH\VV\left([\hl{Y},\vl{X}]-[\hl{X},\vl{Y}]\right)}\hul{Z}=\Hh(\kal{X},\kal{Y})\hul{Z}+\nabla_{[\hl{X},\hl{Y}]-(\FF+\JJ)\vl{[X,Y]}}\hul{Z}=\Hh(\kal{X},\kal{Y})\hul{Z}+\nabla_{[\hl{X},\hl{Y}]-\hl{[X,Y]}}\hul{Z}=\Hh(\kal{X},\kal{Y})\hul{Z}-\nabla_{\ii\RR(\kal{X},\kal{Y})}\hul{Z}=\Hh(\kal{X},\kal{Y})\hul{Z}-\vl{\nabla}\hul{Z}(\RR(\kal{X},\kal{Y}))$,
\end{center}
which proves relation (\ref{hhricci24}). Relation (\ref{hhricci25}) can be checked in the same way.
\end{proof}

Now we suppose that the Ehresmann connection $\HH$ is associated to a spray $S$. Then, as we have already mentioned, $\HH$ is homogeneous and torsion-free. The type $\binom{1}{1}$ tensor field $\KK$ along $\tauk$ defined by
\begin{gather}\label{affinedev}
\KK(\hul{X}):=\VV[S,\HH\hul{X}]\textrm{ , }\hul{X}\in\szel{\pik}
\end{gather}
is said to be the \textit{affine deviation tensor} (L. Berwald \cite{Berwald}) or the \textit{Jacobi endomorphism} (W. Sarlet et al. \cite{Sarlet}) of the spray $S$, or of the Ehresmann connection associated to $S$. The homogeneity of $S$ implies that $S=\HH\delta$, so it follows that
\begin{gather}\label{jacobi22}
\KK(\hul{X})=\VV[\HH\delta,\HH\hul{X}]=-\RR(\delta,\hul{X})=\RR(\hul{X},\delta).
\end{gather}

\begin{coro}\label{affdevhom}
The affine deviation tensor of a spray is homogeneous of degree 2.
\end{coro}

\begin{proof}
For any vector field $X$ on $M$
\begin{center}
$\displaystyle(\nabla_C\KK)(\kal{X})=\nabla_C(\KK(\kal{X}))\overset{\textrm{(\ref{jacobi22})}}{=}\nabla_C(\RR(\kal{X},\delta))=$\\
$\displaystyle(\nabla_C\RR)(\kal{X},\delta)+\RR(\kal{X},\delta)\overset{\textrm{\ref{coro43}}}{=}2\RR(\kal{X},\delta)=2\KK(\kal{X})$.
\end{center}
This proves our claim.
\end{proof}

In view of (\ref{jacobi22}), the affine deviation can immediately be obtained from the curvature of the Ehresmann connection. The converse is also true:

\begin{propo}\label{propo46}
Let $S$ be a spray over $M$, and let $\HH$ be the Ehresmann connection associated to $S$. Then the curvature and the affine deviation of $\HH$ are related by
\begin{gather}\label{curvaff}
\RR(\hul{X},\hul{Y})=\frac{1}{3}(\vl{\nabla}\KK(\hul{Y},\hul{X})-\vl{\nabla}\KK(\hul{X},\hul{Y}))\textrm{ ; }\hul{X},\hul{Y}\in\szel{\pik}.
\end{gather}
\end{propo}

\begin{proof}
Let $X$, $Y$ be vector fields on $M$. Then
\begin{center}
$\displaystyle\vl{\nabla}\KK(\kal{Y},\kal{X})=\left(\nabla_{\vl{Y}}\KK\right)(\kal{X})=\nabla_{\vl{Y}}(\KK(\kal{X}))\overset{\textrm{(\ref{jacobi22})}}{=}\nabla_{\vl{Y}}(\RR(\kal{X},\delta))=\left(\nabla_{\vl{Y}}\RR\right)(\kal{X},\delta)+\RR\left(\nabla_{\vl{Y}}\kal{X},\delta\right)+\RR\left(\kal{X},\nabla_{\vl{Y}}\delta\right)=\vl{\nabla}\RR(\kal{Y},\kal{X},\delta)+\RR(\kal{X},\kal{Y})$.
\end{center}
Similarly,
\begin{center}
$\displaystyle\vl{\nabla}\KK(\kal{X},\kal{Y})=\vl{\nabla}\RR(\kal{X},\kal{Y},\delta)+\RR(\kal{Y},\kal{X})$,
\end{center}
therefore
\begin{center}
$\displaystyle\vl{\nabla}\KK(\kal{Y},\kal{X})-\vl{\nabla}\KK(\kal{X},\kal{Y})=\vl{\nabla}\RR(\kal{Y},\kal{X},\delta)-\vl{\nabla}\RR(\kal{X},\kal{Y},\delta)+2\RR(\kal{X},\kal{Y})\overset{\textrm{(\ref{affcurv19})}}{=}\Hh(\kal{X},\delta)\kal{Y}-\Hh(\kal{Y},\delta)\kal{X}+2\RR(\kal{X},\kal{Y})=\Hh(\kal{X},\delta)\kal{Y}+\Hh(\delta,\kal{Y})\kal{X}+2\RR(\kal{X},\kal{Y})\overset{\textrm{\ref{bian45}}}{=}-\Hh(\kal{Y},\kal{X})\delta+2\RR(\kal{X},\kal{Y})=\Hh(\kal{X},\kal{Y})\delta+2\RR(\kal{X},\kal{Y})\overset{\textrm{\ref{homehrlemm421}}}{=}3\RR(\kal{X},\kal{Y})$.
\end{center}
\end{proof}

\begin{coro}\label{corotracegyh}
For the trace of the curvature of a spray we have
\begin{gather}\label{traceegyharm}
\textrm{\textup{tr}}\RR=\frac{1}{3}(\vl{\nabla}\textrm{\textup{tr}}\KK-\textrm{\textup{tr}}\vl{\nabla}\KK).
\end{gather}
\end{coro}

\begin{proof}
For any section $\hul{X}$ in $\szel{\pik}$,
\begin{center}
$\displaystyle(\textrm{tr}\RR)(\hul{X}):=\textrm{tr}(\hul{Z}\mapsto\RR(\hul{Z},\hul{X}))\overset{\textrm{(\ref{curvaff})}}{=}\frac{1}{3}\textrm{tr}(\hul{Z}\mapsto(\vl{\nabla}\KK(\hul{X},\hul{Z})-\vl{\nabla}\KK(\hul{Z},\hul{X})))=-\frac{1}{3}(\textrm{tr}\vl{\nabla}\KK)(\hul{X})+\frac{1}{3}\textrm{tr}(\hul{Z}\mapsto(\nabla_{\ii\hul{X}}\KK)(\hul{Z}))\overset{\textrm{(*)}}{=}\frac{1}{3}(\nabla_{\ii\hul{X}}(\textrm{tr}\KK)-(\textrm{tr}\vl{\nabla}\KK)(\hul{X}))=\frac{1}{3}(\vl{\nabla}\textrm{tr}\KK-\textrm{tr}\vl{\nabla}\KK)(\hul{X})$,
\end{center}
using at step (*) that trace operators and covariant derivatives commute.
\end{proof}

\textbf{Remark.} As we have already mentioned in the Introduction, Berwald's starting point in his famous posthumus paper \cite{Berwald} is a SODE of form 
\begin{gather}\label{BerwaldSODE}
(x^i)''+2G^i(x,x')=0\textrm{ , }i\in\left\{1,\dots,n\right\},
\end{gather}
where $$G^i\in C^1(\tau^{-1}(\UU))\cap C^{\infty}(\tauk^{-1}(\UU))\textrm{ , }CG^i=2G^i$$ ($\UU\subset M$ is a coordinate neighbourhood).

As a first step, Berwald deduces the following `equation of affine deviation': $$\frac{D^2\xi^i}{ds}+K^i_j\left(x,\frac{dx}{ds}\right)\xi^j=0\textrm{ , }i\in\left\{1,\dots,n\right\},$$ where $$\frac{D\xi^i}{ds}:=\frac{d\xi^i}{ds}+G^i_r\xi^r\textrm{ , }G^i_r:=\pp{G^i}{y^r},$$ and the functions $$K^i_j:=2\pp{G^i}{x^j}-\pp{G^i_j}{x^r}y^r+2G^i_{jr}G^r-G^i_rG^r_j\textrm{ , }(i,j\in\left\{1,\dots,n\right\})$$ are the components of a type $\binom{1}{1}$ tensor, called \textit{affinen Abweichungstensor} by Berwald.
It may easily be checked that our tensor $\KK$ is indeed an intrinsic form of the tensor obtained by him in this way.

In the second step, Berwald introduces the `Grundtensor der affinen Krümmung', giving its components by $$K^i_{jk}:=\frac{1}{3}\left(\pp{K^i_k}{y^j}-\pp{K^i_j}{y^k}\right).$$
Proposition \ref{propo46} shows that this is just the curvature of the Ehresmann connection which may be associated to the SODE (\ref{BerwaldSODE}).

Finally, in the third step, Berwald defines the `affine Krümmungstensor' by its components $$K^i_{hjk}:=\pp{K^i_{jk}}{y^h}.$$

In view of Lemma \ref{lemma41}, this is just our tensor $\Hh$.\\

Some very specific forms of the Jacobi endomorphism lead to important special classes of sprays. Namely (cf. \cite{GM}, \cite{VD}), we say that a spray is

\textit{flat}, if
\begin{gather}\label{flat}
\KK=\lambda\mathbf{1}_{\szel{\pik}}\textrm{ , }\lambda\in\modc{\Tk M};
\end{gather}

\textit{isotropic}, if
\begin{gather}\label{isotropic}
\KK=\lambda\mathbf{1}_{\szel{\pik}}+\hul{\alpha}\otimes\delta\textrm{ ; }\lambda\in\modc{\Tk M}\textrm{ , }\hul{\alpha}\in\tenz{0}{1}{\pik}.
\end{gather}

Condition of \textit{flatness} is very strong: it implies that $\KK$, and hence $\RR$ and $\Hh$ vanish. Indeed, we have
\begin{center}
$\displaystyle\lambda\delta=(\lambda\mathbf{1})(\delta)\overset{\textrm{(\ref{flat})}}{=}\KK(\delta)\overset{\textrm{(\ref{jacobi22})}}{=}\RR(\delta,\delta)=0$,
\end{center}
whence $\lambda=0$ and $\KK=0$. (If $S$ is only a semispray, then such a radical conclusion is not possible.)

We show that in the \textit{isotropic case} we have $$\lambda=\frac{1}{n-1}\textrm{tr}\KK,$$ $$\textrm{tr}(\hul{\alpha}\otimes\delta)=\hul{\alpha}(\delta)=-\lambda=\frac{1}{1-n}\textrm{tr}\KK,$$ $$\nabla_C\hul{\alpha}=\hul{\alpha}\textrm{ , i.e., }\hul{\alpha}\textrm{\textit{ is homogeneous of degree 1}}.$$

To this end, observe first that
\begin{center}
$\displaystyle0=\RR(\delta,\delta)\overset{\textrm{(\ref{jacobi22})}}{=}\KK(\delta)\overset{\textrm{(\ref{isotropic})}}{=}\lambda(\delta)+\hul{\alpha}(\delta)\delta=(\lambda+\hul{\alpha}(\delta))\delta$,
\end{center}
hence $\hul{\alpha}(\delta)=-\lambda$. Now, taking the trace of both sides of (\ref{isotropic}), we obtain
\begin{center}
$\displaystyle\textrm{tr}\KK=n\lambda+\hul{\alpha}(\delta)=n\lambda-\lambda=(n-1)\lambda$,
\end{center}
whence $\lambda=\frac{1}{n-1}\textrm{tr}\KK$. Since $\KK$ is homogeneous of degree 2 by Corollary \ref{affdevhom}, $\textrm{tr}\KK$ has the same homogeneity property, so $$\nabla_C\textrm{tr}\KK=\textrm{tr}\nabla_C\KK=2\textrm{tr}\KK.$$ Then $\hul{\alpha}\otimes\delta=\KK-\left(\frac{1}{n-1}\textrm{tr}\KK\right)\mathbf{1}$ is also homogeneous of degree 2, so we have $$2(\hul{\alpha}\otimes\delta)=\nabla_C(\hul{\alpha}\otimes\delta)=(\nabla_C\hul{\alpha})\otimes\delta+\hul{\alpha}\otimes\delta,$$ hence $$\nabla_C\hul{\alpha}\otimes\delta=\hul{\alpha}\otimes\delta,$$ and, therefore, for all $X\in\modx{M}$ we have $$(\nabla_C\hul{\alpha})(\kal{X})\delta=\hul{\alpha}(\hul{X})\delta.$$ This implies the desired relation $\nabla_C\hul{\alpha}=\hul{\alpha}$.

We note finally that the 1-form $\hul{\alpha}$ is unique. Indeed, if a 1-form $\hul{\beta}$ also satisfies (\ref{isotropic}), then for all $X\in\modx{M}$,
\begin{center}
$\displaystyle(\KK-\lambda\mathbf{1})(\kal{X})=\hul{\alpha}(\kal{X})\delta\textrm{  }$  and  $\displaystyle\textrm{  }(\KK-\lambda\mathbf{1})(\kal{X})=\hul{\beta}(\kal{X})\delta$,
\end{center}
which implies the equality $\hul{\alpha}=\hul{\beta}$.

\chapter{Projectively related sprays} \label{ch5}

We recall (for details, see \cite{Shen}, \cite{SzV}, \cite{SzV2}, \cite{SzGy}) that two sprays $S$ and $\overline{S}$ over $M$ are said to be (pointwise) \textit{projectively related}, if there exists a function $P: TM\rightarrow\valR$, $C^1$ on $TM$, smooth on $\Tk M$, such that
\begin{gather}\label{projrel30}
\overline{S}=S-2PC.
\end{gather}

The \textit{projective factor} $P$ in (\ref{projrel30}) is necessarily positive-homogeneous of degree 1, i.e., $CP=P$. If $\AA$ is a geometric object associated to $S$, then we denote by $\overline{\AA}$ the corresponding geometric object determined by $\overline{S}$. The following well-known relations may easily be checked:

If $\HH$ is the Ehresmann connection associated to $S$, then
\begin{gather}\label{corr31}
\overline{\HH}=\HH-P\ii-\vl{\nabla}P\otimes C,
\end{gather}
\begin{gather}\label{corr32}
\overline{\hh}=\hh-P\JJ-(\vl{\nabla}P\circ\jj)\otimes C=\hh-P\JJ-d_{\JJ}P\otimes C,
\end{gather}
\begin{gather}\label{corr33}
\hlfel{X}=\hl{X}-P\vl{X}-(\vl{X}P)C\textrm{  }\textrm{  }(X\in\modx{M}),
\end{gather}
\begin{gather}\label{corr34}
\overline{\VV}=\VV+P\jj+(\vl{\nabla}P\circ\jj)\otimes\delta=\VV+P\jj-d_{\JJ}P\otimes\delta.
\end{gather}

We also have the less immediate

\begin{lemma}
\begin{gather}\label{corr35}
\hlfel{\overline{\nabla}}=\hl{\nabla}-P\vl{\nabla}-\vl{\nabla}P\otimes\nabla_C+\vl{\nabla}P\odot\mathbf{1}+\vl{\nabla}\vl{\nabla}P\otimes\delta,
\end{gather}
where the symbol $\odot$ denotes symmetric product (without any numerical factor), and $\mathbf{1}\in\tenz{1}{1}{\pi}$ is the unit tensor.
\end{lemma}

\begin{proof}
Let $\hul{X}$ and $\hul{Y}$ be vector fields along $\tauk$. Then
\begin{center}
$\displaystyle\overline{\nabla}_{\HHfel\hul{X}}\hul{Y}=\overline{\VV}[\HHfel\hul{X},\ii\hul{Y}]\overset{\textrm{(\ref{corr31})}}{=}\overline{\VV}[\HH\hul{X}-P\ii\hul{X}-\vl{\nabla}P(\hul{X})C,\ii\hul{Y}]=\overline{\VV}[\HH\hul{X},\ii\hul{Y}]+\overline{\VV}\left(\ii\hul{Y}(P)\ii\hul{X}-P[\ii\hul{X},\ii\hul{Y}]\right)+\overline{\VV}\left(\ii\hul{Y}(\ii\hul{X}P)C-\ii\hul{X}(P)[C,\ii\hul{Y}]\right)\overset{\textrm{(\ref{corr34})}}{=}\VV[\HH\hul{X},\ii\hul{Y}]-P\jj[\ii\hul{Y},\HH\hul{X}]-(\JJ[\ii\hul{Y},\HH\hul{X}]P)\delta+\ii\hul{Y}(P)\hul{X}-P\ii^{-1}[\ii\hul{X},\ii\hul{Y}]+\ii\hul{Y}(\ii\hul{X}P)\delta-\ii\hul{X}(P)\ii^{-1}[\ii\delta,\ii\hul{Y}]$.
\end{center}
An easy calculation shows that
\begin{center}
$\displaystyle\ii^{-1}[\ii\hul{X},\ii\hul{Y}]=\nabla_{\ii\hul{X}}\hul{Y}-\nabla_{\ii\hul{Y}}\hul{X}$,\\
$\displaystyle\ii^{-1}[\ii\delta,\ii\hul{Y}]=\nabla_C\hul{Y}-\nabla_{\ii\hul{Y}}\delta=\nabla_C\hul{Y}-\hul{Y}$,
\end{center}
so we obtain
\begin{center}
$\displaystyle\overline{\nabla}^{\overline{\hh}}(\hul{X},\hul{Y})=\overline{\nabla}_{\HHfel\hul{X}}\hul{Y}=\nabla_{\HH\hul{X}}\hul{Y}-P\nabla_{\ii\hul{Y}}\hul{X}-(\ii\nabla_{\ii\hul{Y}}\hul{X})P\delta+\ii\hul{Y}(P)\hul{X}-P\nabla_{\ii\hul{X}}\hul{Y}+P\nabla_{\ii\hul{Y}}\hul{X}+(\ii\hul{Y}(\ii\hul{X})P)\delta-\ii\hul{X}(P)\nabla_C\hul{Y}+\ii\hul{X}(P)\hul{Y}=\nabla_{\HH\hul{X}}\hul{Y}-P\vl{\nabla}_{\hul{X}}\hul{Y}-\vl{\nabla}P(\hul{X})\nabla_C\hul{Y}+\vl{\nabla}P(\hul{X})\hul{Y}+\vl{\nabla}P(\hul{Y})\hul{X}-\left(\ii\nabla_{\ii\hul{Y}}\hul{X}\right)P\delta+\left(\ii\nabla_{\ii\hul{Y}}\hul{X}-\ii\nabla_{\ii\hul{X}}\hul{Y}\right)P\delta+(\ii\hul{X}(\ii\hul{Y})P)\delta=(\hl{\nabla}-P\vl{\nabla}-\vl{\nabla}P\otimes\nabla_C+\vl{\nabla}P\odot\mathbf{1}+\vl{\nabla}\vl{\nabla}P\otimes\delta)(\hul{X},\hul{Y})$.
\end{center}
\end{proof}

\begin{coro}
For all vector fields $X$, $Y$ on $M$,
\begin{gather}\label{corr36}
\overline{\nabla}_{\hlfel{X}}\kal{Y}=\nabla_{\hl{X}}\kal{Y}+(\vl{X}P)\kal{Y}+(\vl{Y}P)\kal{X}+\vl{X}(\vl{Y}P)\delta,
\end{gather}
\begin{gather}\label{corr37}
\overline{\nabla}_{\overline{S}}\kal{Y}=\nabla_S\kal{Y}+P\kal{Y}+(\vl{Y}P)\delta.
\end{gather}
\end{coro}

\begin{proof}
Since $$\vl{\nabla}(\kal{X},\kal{Y})=\nabla_{\vl{X}}\kal{Y}=0,$$ $$\nabla_C\kal{Y}=\jj[C,\vl{Y}]=-\jj\vl{Y}=0,$$ for basic vector fields $\kal{X},\kal{Y}$ relation (\ref{corr35}) leads to (\ref{corr36}). As to the second relation, observe that
\begin{center}
$\displaystyle\overline{\hh}\textrm{ }\overline{S}=(\hh-P\JJ-(\vl{\nabla}P\circ\jj)C)(S-2PC)=S-PC-(CP)C=S-2PC=\overline{S}$.
\end{center}
Hence
\begin{center}
$\displaystyle\overline{\nabla}_{\overline{S}}\kal{Y}=\overline{\nabla}_{\overline{\hh}\textrm{ }\overline{S}}\kal{Y}=\hlfel{\overline{\nabla}}_{\delta}\kal{Y}\overset{\textrm{(\ref{corr35})}}{=}\hl{\nabla}_{\delta}\kal{Y}-P\vl{\nabla}_{\delta}\kal{Y}-\vl{\nabla}P(\delta)\nabla_C\kal{Y}+\vl{\nabla}P(\delta)\kal{Y}+\vl{\nabla}P(\kal{Y})\delta+\vl{\nabla}\vl{\nabla}P(\delta,\kal{Y})\delta=\nabla_S\kal{Y}+P\kal{Y}+(\vl{Y}P)\delta$,
\end{center}
since
\begin{center}
$\displaystyle\vl{\nabla}\vl{\nabla}P(\delta,\kal{Y})=\nabla_C(\vl{\nabla}P)(\kal{Y})=C(\vl{Y}P)=$\\
$\displaystyle [C,\vl{Y}]P+\vl{Y}(CP)=-\vl{Y}P+\vl{Y}P=0$.
\end{center}
\end{proof}

It was shown in \cite{SzGy}, that the Berwald curvatures of $S$ and $\overline{S}$ and their traces are related by
\begin{gather}\label{corr38}
\overline{\BB}=\BB-\vl{\nabla}\vl{\nabla}P\odot\mathbf{1}-\vl{\nabla}\vl{\nabla}\vl{\nabla}P\otimes\delta
\end{gather}
and
\begin{gather}\label{corr39}
\textrm{tr}\overline{\BB}=\textrm{tr}\BB-(n+1)\vl{\nabla}\vl{\nabla}P,
\end{gather}
respectively.

From relation (\ref{corr39}) it follows at once that the trace of the Berwald curvature is a projective invariant of the spray, if and only if, the projective factor satisfies the PDE
%Now it is very natural to ask:
%
%\textit{Under what conditions on the projective factor is preserved the Berwald or the weakly Berwald character of a spray?}
%
%The first half of the question was raised and studied, applying the tools of classical tensor calculus, by A. Moór \cite{Moor} and very recently by J. Szilasi and Á. Győry \cite{Gyory2}. As for the second half of the question, it is immediate from (\ref{corr39}) that a weakly Berwald spray remains weakly Berwald, if and only if, the projective factor $P$ satisfies the PDE
\begin{gather}\label{pde40}
\vl{\nabla}\vl{\nabla}P=0.
\end{gather}
However, \textit{this relation gives also the criterion of the projective invariance of the Berwald curvature}.

Indeed, if $P$ satisfies (\ref{pde40}), then (\ref{corr38}) yields $\overline{\BB}=\BB$. Conversely, if $\overline{\BB}=\BB$, then $\textrm{tr}\overline{\BB}=\textrm{tr}\BB$, and (\ref{corr39}) implies that $\vl{\nabla}\vl{\nabla}P=0$.

Now we determine the solutions of this PDE in an index-free manner.

\begin{propo}
The Berwald curvature of a spray $S$ is invariant under a projective change $\overline{S}=S-2PC$, if and only if, the projective factor is of form
\begin{gather}\label{weaklyb41}
P=i_{\xi}\vl{\alpha}=:\overline{\alpha}\textrm{ , }\alpha\in\mathfrak{X}^{*}(M),
\end{gather}
where $\xi$ is an arbitrary second-order vector field over $M$.
\end{propo}

\begin{proof}
First we check that the functions given by (\ref{weaklyb41}) solve (\ref{pde40}). To do this, observe that for any vector field $Y$ on $M$,
\begin{gather}\label{weak2}
\vl{Y}\overline{\alpha}=\vl{(\alpha(Y))}.
\end{gather}
Indeed, $$\vl{Y}\overline{\alpha}=\vl{Y}i_{\xi}\vl{\alpha}=\LL_{\vl{Y}}i_{\xi}\vl{\alpha}=i_{\xi}\LL_{\vl{Y}}\vl{\alpha}+i_{\left[\vl{Y},\xi\right]}\vl{\alpha}.$$ The first term on the right-hand side vanishes, since for any vector field $X$ on $M$ we have $$(\LL_{\vl{Y}}\vl{\alpha})(\vl{X})=\vl{Y}(\vl{\alpha}(\vl{X}))-\vl{\alpha}(\left[\vl{Y},\vl{X}\right])=0,$$ $$(\LL_{\vl{Y}}\vl{\alpha})(\cl{X})=\vl{Y}(\vl{\alpha}(\cl{X}))-\vl{\alpha}(\left[\vl{Y},\cl{X}\right])=\vl{Y}\vl{(\alpha(X))}-\vl{\alpha}(\vl{\left[Y,X\right]})=0.$$
As for the second term, it is known (see e.g. \cite{Szilasi}, 3.2, Corollary) that $$\left[\vl{Y},\xi\right]=\cl{Y}+\eta\textrm{ , }\eta\in\vl{\mathfrak{X}}(TM),$$ therefore $$\vl{Y}\overline{\alpha}=\vl{\alpha}\left(\left[\vl{Y},\xi\right]\right)=\vl{\alpha}(\cl{Y})+\vl{\alpha}(\eta)=\vl{(\alpha(Y))},$$ as we claimed.

Now, for any vector fields $X$, $Y$ on $M$, $$\vl{\nabla}\vl{\nabla}\overline{\alpha}(\kal{X},\kal{Y})=\left(\nabla_{\vl{X}}\vl{\nabla}\overline{\alpha}\right)(\kal{Y})=\vl{X}(\vl{\nabla}\overline{\alpha}(\kal{Y}))=\vl{X}(\vl{Y}\overline{\alpha})\overset{(\ref{weak2})}{=}\vl{X}\vl{(\alpha(Y))}=0,$$
therefore the functions $\overline{\alpha}$, where $\alpha\in\mathfrak{X}^{*}(M)$, solve our PDE (\ref{pde40}).

We show that these solutions satisfy the homogeneity condition $C\overline{\alpha}=\overline{\alpha}$. We may suppose that $\xi$ is homogeneous of degree two, i.e., $\left[C,\xi\right]=\xi$, since the definition of $\overline{\alpha}$ does not depend on the choice of $\xi$. Then we obtain $$C\overline{\alpha}=\LL_Ci_{\xi}\vl{\alpha}=i_{\xi}\LL_C\vl{\alpha}-i_{\left[\xi,C\right]}\vl{\alpha}=i_{\left[C,\xi\right]}\vl{\alpha}=i_{\xi}\vl{\alpha}=\overline{\alpha},$$
since $\LL_C\vl{\alpha}=0$.

Conversely, if $\vl{\nabla}\vl{\nabla}P=0$, then for all $X,Y\in\modx{M}$, $$0=\vl{\nabla}\vl{\nabla}P(\kal{X},\kal{Y})=(\nabla_{\vl{X}}(\vl{\nabla}P))(\kal{Y})=\vl{X}(\vl{Y}P).$$ Then $\vl{Y}P$ is the vertical lift of a smooth function on $M$, therefore it is of form $$\vl{Y}P=\vl{(\alpha(Y))}\overset{(\ref{weak2})}{=}\vl{Y}\overline{\alpha}\textrm{ , }\alpha\in\mathfrak{X}^{*}(M).$$ Since $Y$ is arbitrary, this implies that $$P=\overline{\alpha}+\vl{f}\textrm{ , }f\in\modc{M}.$$ However, the homogeneity condition $CP=P$ forces that $f=0$, since $C\overline{\alpha}=\overline{\alpha}$ and $C\vl{f}=0$.
\end{proof}

\begin{coro}
A Berwald or a weakly Berwald spray remains of that type under a projective change, if and only if, the projective factor is given by (\ref{weaklyb41}).
\end{coro}
\mbox{}\hfill\tiny$\square$\normalsize \bigskip

It was discovered by J. Douglas \cite{Douglas} that from the Berwald curvature it is possible to construct a projectively invariant tensor. After him, this tensor is said to be the \textit{Douglas curvature} of the given spray; we denote it by $\mathbf{D}$ in the following. An index-free description of the Douglas curvature is due to J. Szilasi and Sz. Vattamány \cite{SzV}. They worked on the bundle $\tau_{TM}: TTM\rightarrow TM$ and applied the Frölicher-Nijenhuis formalism of vector-valued forms. In our setting their definition reads as follows:
\begin{gather}\label{d26}
\mathbf{D}:=\mathbf{B}-\frac{1}{n+1}(\textrm{tr}\mathbf{B}\odot\mathbf{1}+(\vl{\nabla}\textrm{tr}\mathbf{B})\otimes\delta).
\end{gather}
Directly from (\ref{corr38}) and (\ref{corr39}), we see that $\mathbf{D}$ is indeed invariant under any projective change of the given spray. Observe that for weakly Berwald sprays the Berwald and the Douglas curvature coincide.

We say that a spray is a \textit{Douglas spray} if its Douglas curvature vanishes. It is immediate from (\ref{d26}) that \textit{a spray is a Berwald spray, if and only if, it is a weakly Berwald Douglas spray}.

As the Berwald curvature, the affine deviation tensor (i.e., the Jacobi endomorphism), and hence the curvature and the affine curvature of a spray $S$ are not invariant under a projective change $\overline{S}:=S-2PC$ of $S$. It may be shown by a straightforward but quite lengthy calculation (\cite{Shen},\cite{SzGy3}), that the change of the Jacobi endomorphism is given by $$\overline{\KK}=\KK+\lambda\mathbf{1}_{\szel{\pik}}+\hul{\alpha}\otimes\delta,$$ where $\lambda\in\modc{\Tk M}$ and $\hul{\alpha}$ is a 1-form along $\tauk$. The second-degree homogeneity of the Jacobi endomorphism implies that $\lambda$ is homogeneous of degree 2, $\hul{\alpha}$ is homogeneous of degree 1. Since $\overline{\KK}(\delta)=\KK(\delta)=0$, it follows that $\hul{\alpha}(\delta)=-\lambda$, hence $$\textrm{tr}\overline{\KK}=\textrm{tr}\KK+n\lambda+\hul{\alpha}(\delta)=\textrm{tr}\KK+(n-1)\lambda.$$ Explicitly (see \cite{SzGy3}), $$\lambda=P^2-SP\textrm{ , }\hul{\alpha}=3(\hl{\nabla}P-P\vl{\nabla}P)+\vl{\nabla}\hul{\sigma}.$$ However, one can construct also from the Jacobi endomorphism projectively invariant tensors. Following del Castillo \cite{delc}, \textit{mutatis mutandis}, we define the \textit{Weyl endomorphism} (or, by Berwald's terminology, the \textit{projective deviation tensor}) $\mathbf{W}^{\circ}$ of a spray $S$ by
\begin{gather}\label{Weylk}
\mathbf{W}^{\circ}:=\KK-\frac{1}{n-1}(\textrm{tr}\KK)\mathbf{1}+\frac{3}{n+1}(\textrm{tr}\RR)\otimes\delta+\frac{2-n}{n^2-1}(\vl{\nabla}\textrm{tr}\KK)\otimes\delta.
\end{gather}
Now, on the analogy of (\ref{curvaff}) and (\ref{affcurv19}), let
\begin{gather}
\mathbf{W}(\hul{X},\hul{Y}):=\frac{1}{3}(\vl{\nabla}\mathbf{W}^{\circ}(\hul{Y},\hul{X})-\vl{\nabla}\mathbf{W}^{\circ}(\hul{X},\hul{Y}));
\end{gather}
and
\begin{gather}
\mathbf{W}^{*}(\hul{X},\hul{Y})\hul{Z}:=\vl{\nabla}\mathbf{W}(\hul{Z},\hul{X},\hul{Y})
\end{gather}
$(\hul{X},\hul{Y},\hul{Z}\in\szel{\pik})$. Then $\mathbf{W}^{\circ}$, $\mathbf{W}$ and $\mathbf{W}^{*}$ are projectively invariant. Using Berwald's terminology \cite{Berwald}, we call $\mathbf{W}$ and $\mathbf{W}^{*}$ the \textit{fundamental projective curvature tensor} and the \textit{projective curvature tensor} of the spray, respectively. Notice that in his book \cite{Shen} Z. Shen also has constructed a projectively invariant tensor of type $\binom{1}{1}$, called Weyl curvature, but it differs from the Weyl endomorphism $\mathbf{W}^{\circ}$. $\mathbf{W}^{\circ}$, $\mathbf{W}$ and $\mathbf{W}^{*}$ have the same homogeneity properties as the corresponding affine curvature tensors $\mathbf{K}$, $\RR$ and $\Hh$. It may also be shown, as Berwald has already pointed out, that the statements $$\mathbf{W}^{\circ}=0\textrm{ , }\mathbf{W}=0\textrm{ }\textrm{ and }\textrm{ }\mathbf{W}^{*}=0$$ are equivalent. When $\textrm{dim}M=2$, the vanishing of these tensors holds automatically; we shall check this in the Finslerian case in Chapter \ref{ch9}. \textit{In dimension greater than 2 a spray is isotropic, if and only if, one and hence all of the tensors} $\mathbf{W}^{\circ}$, $\mathbf{W}$ \textit{and} $\mathbf{W}^{*}$ \textit{vanishes}. For two recent proofs of this fundamental fact we refer to \cite{Cramp2} and \cite{VD}; in the next Chapter we shall present a new index-free argument in the context of Finslerian sprays. To complete the story, in Chapter \ref{ch9} we shall show that the canonical spray of any 2-dimensional Finsler manifold is isotropic.

For these purposes, it will be useful to represent the Weyl endomorphism in a more convenient form.
% We define the \textit{scalar curvature} $K$ of a spray by $$K:=\frac{1}{n-1}\textrm{tr}\KK.$$

\begin{lemma}
The Weyl endomorphism of a spray can be written in the form
\begin{gather}\label{WWeyl}
\mathbf{W}^{\circ}=\KK-K\mathbf{1}+\frac{1}{n+1}(\vl{\nabla}K-\textrm{\textup{tr}}\vl{\nabla}\KK)\otimes\delta,
\end{gather}
where $K:=\frac{1}{n-1}\textrm{tr}\KK$.
\end{lemma}

\begin{proof}
In (\ref{Weylk}) we replace $\textrm{tr}\KK$ by $(n-1)K$ and $\textrm{tr}\RR$ by the right-hand side of (\ref{traceegyharm}). Then we obtain
\begin{center}
$\displaystyle\mathbf{W}^{\circ}=\KK-K\mathbf{1}+\frac{n-1}{n+1}\vl{\nabla}K\otimes\delta-\frac{1}{n+1}(\textrm{tr}\vl{\nabla}\KK)\otimes\delta+\frac{2-n}{n+1}\vl{\nabla}K\otimes\delta=\KK-K\mathbf{1}+\frac{1}{n+1}(\vl{\nabla}K-\textrm{\textup{tr}}\vl{\nabla}\KK)\otimes\delta$.
\end{center}
\end{proof}

\chapter{Basic facts on Finsler manifolds}\label{ch6}

By a \textit{Finsler function} over $M$ we mean a function $F:TM\rightarrow\valR$ satisfying the following conditions:
\begin{itemize}
\item[(F$_{\textrm{1}}$)] $F$ is smooth on $\Tk M$.
\item[(F$_{\textrm{2}}$)] $F$\textit{ is positive-homogeneous of degree 1} in the sense that for each non-negative real number $\lambda$ and each vector $v\in TM$, we have $$F(\lambda v)=\lambda F(v).$$
\item[(F$_{\textrm{3}}$)] The \textit{metric tensor} $$g:=\frac{1}{2}\vl{\nabla}\vl{\nabla}F^2\in\tenz{0}{2}{\pik}$$ is (fibrewise) non-degenerate.
\end{itemize}
A \textit{Finsler manifold} is a pair $(M,F)$ consisting of a manifold $M$ and a Finsler function on its tangent manifold.

By (F$_{\textrm{1}}$) and (F$_{\textrm{2}}$), $F$ is continuous on $TM$ and identically zero on $\sigma(M)$. The function $E:=\frac{1}{2}F^2$ is called the \textit{energy function} of the Finsler manifold $(M,F)$. It is continuous on $TM$, smooth on $\Tk M$, and also indentically zero on $\sigma(M)$. By (F$_{\textrm{2}}$), $E$ satisfies $$E(\lambda v)=\lambda^2E(v)$$ for all $v\in TM$ and non-negative $\lambda\in\valR$, i.e., $E$ is positive-homogeneous of degree 2. Over $\Tk M$ this holds, if and only if, $CE=2E$. It may be shown (see e.g. \cite{Warner}) that, actually, $E$ \textit{is of class} $C^1$\textit{ on} $TM$.

For any vector fields $X$, $Y$ on $M$ we have
\begin{gather}\label{fins43}
g(\kal{X},\kal{Y})=\vl{X}(\vl{Y}E),
\end{gather}
from which it follows immediately that $g$ is \textit{symmetric}. More generally, if $\hul{X}$ and $\hul{Y}$ are in $\szel{\pik}$, then
\begin{gather}\label{alpha}
g(\hul{X},\hul{Y})=(\ii\hul{X})(\ii\hul{Y})E-(\ii\nabla_{\ii\hul{X}}\hul{Y})E=(\ii\hul{X})(\ii\hul{Y})E-\JJ[\ii\hul{X},\HH\hul{Y}]E,
\end{gather}
where $\HH$ is an arbitrary Ehresmann connection over $M$. In particular, we get
\begin{gather}\label{beta}
g(\delta,\delta)=2E.
\end{gather}
A further elementary property of the metric tensor is that it is homogeneous of degree 0, i.e.,
\begin{gather}
\vl{\nabla}_{\delta}g=\nabla_Cg=0.
\end{gather}
Indeed, for any vector fields $X$, $Y$ on $M$,
\begin{center}
$\displaystyle\left(\nabla_Cg\right)(\kal{X},\kal{Y})=Cg(\kal{X},\kal{Y})\overset{(\ref{fins43})}{=}C(\vl{X}(\vl{Y}E))=[C,\vl{X}](\vl{Y}E)+\vl{X}(C(\vl{Y}E))=-\vl{X}\vl{Y}E+\vl{X}([C,\vl{Y}]E+\vl{Y}(CE))=-2\vl{X}(\vl{Y}E)+2\vl{X}(\vl{Y}E)=0$.
\end{center}
A Finsler manifold $(M,F)$ is said to be \textit{positive definite} if the condition
\begin{itemize}
\item[(F$_{\textrm{4}}$)] $F(v)>0$ whenever $v\in\Tk M$
\end{itemize}
is also satisfied. It may be shown (see \cite{Lovas}) that in this case the metric tensor is (fibrewise) positive definite.

The type $\binom{0}{3}$ tensor
\begin{gather}\label{fins45}
\CC_{\flat}:=\frac{1}{2}\vl{\nabla}g=\frac{1}{2}\vl{\nabla}\vl{\nabla}\vl{\nabla}E
\end{gather}
is said to be the \textit{Cartan tensor} of the Finsler manifold $(M,F)$. The \textit{vector-valued Cartan tensor} is the type $\binom{1}{2}$ tensor $\CC$ determined by the requirement $$g(\CC(\hul{X},\hul{Y}),\hul{Z})=\CC_{\flat}(\hul{X},\hul{Y},\hul{Z})\textrm{ ; }\hul{X},\hul{Y},\hul{Z}\in\szel{\pik}.$$ In classical tensor calculus, this change of type is accomplished by `raising a covariant index of $\CC_{\flat}$'.

Due to the non-degeneracy of $g$, there exists a unique vector field $\overset{\ast}{\CC}$ along $\tauk$ such that for each section $\hul{X}$ in $\szel{\pik}$,
\begin{gather}\label{cartan}
g(\overset{\ast}{\CC},\hul{X})=\textrm{tr}\CC(\hul{X}).
\end{gather}
$\overset{\ast}{\CC}$ is said to be the \textit{Cartan vector field} of $(M,F)$. This elegant construction is taken from \cite{MT}.

It may easily be seen that $\CC_{\flat}$ is \textit{totally symmetric}: for any sections $\hul{X}_1$, $\hul{X}_2$, $\hul{X}_3$ in $\szel{\pik}$ and any permutation $\sigma:\left\{1,2,3\right\}\rightarrow\left\{1,2,3\right\}$ we have $$\CC_{\flat}(\hul{X}_{\sigma(1)},\hul{X}_{\sigma(2)},\hul{X}_{\sigma(3)})=\CC_{\flat}(\hul{X}_1,\hul{X}_2,\hul{X}_3).$$
Since $g$ is homogeneous of degree 0, it follows that $\CC_{\flat}$ is homogeneous of degree $-1$, i.e.,
\begin{gather}
\vl{\nabla}_{\delta}\CC_{\flat}=\nabla_C\CC_{\flat}=-\CC_{\flat}.
\end{gather}
As a consequence of the total symmetry of $\CC_{\flat}$ and the 0-homogeneity of $g$, we get
\begin{gather}\label{49ny}
\delta\in\left\{\hul{X},\hul{Y},\hul{Z}\right\}\textrm{  }\Rightarrow\textrm{  }\CC_{\flat}(\hul{X},\hul{Y},\hul{Z})=0.
\end{gather}
Indeed,
\begin{center}
$\displaystyle2\CC_{\flat}(\delta,\hul{Y},\hul{Z})=\vl{\nabla}g(\delta,\hul{Y},\hul{Z})=\left(\nabla_Cg\right)(\hul{Y},\hul{Z})=0$.
\end{center}
Obviously, the vector-valued Cartan tensor also has an analogous property:
\begin{gather}\label{analog}
\CC(\hul{X},\delta)=\CC(\delta,\hul{X})=0\textrm{ , }\textrm{for all }\hul{X}\in\szel{\pik}.
\end{gather}
Using this, we obtain the following result, which can also be found in \cite{MT}.

\begin{lemma}\label{astdelta}
The Cartan vector field of a Finsler manifold is $g$-orthogonal to the canonical section, i.e., $g(\overset{\ast}{\CC},\delta)=0$.
\end{lemma}

\begin{proof}
Let, as usual, $i_{\delta}\CC(\hul{X}):=\CC(\delta,\hul{X})\textrm{ , }\hul{X}\in\szel{\pik}$. Then $i_{\delta}\CC=0$ by (\ref{analog}), and we have
\begin{center}
$\displaystyle g(\overset{\ast}{\CC},\delta):=\textrm{tr}\CC(\delta)=i_{\delta}\textrm{tr}\CC=\textrm{tr}(i_{\delta}\CC)=0$.
\end{center}
\end{proof}

It is also known (see e.g. \cite{Warner} again) that the following assertions are equi\-valent for a positive definite Finsler manifold $(M,F)$:
\begin{itemize}
\item[(i)] \textit{The energy function} $E$ \textit{of} $(M,F)$ \textit{is of class} $C^2$ \textit{(and hence smooth) on} $TM$.
\item[(ii)] $E$ \textit{is the norm associated to a Riemannian metric on} $M$.
\item[(iii)] \textit{There exists a Riemannian metric} $\gamma$ \textit{on} $M$\textit{, such that} $$g(\kal{X},\kal{Y})=\gamma(X,Y)\circ\tau,$$ \textit{for all vector fields} $X,Y\in\modx{M}$.
\item[(iv)] \textit{The Cartan tensor of} $(M,F)$ \textit{vanishes.}
\end{itemize}

Surprisingly, a much stronger result, due to A. Deicke \cite{Deicke} is also true. Namely, \textit{a positive definite Finsler manifold reduces to a Riemannian manifold} (in the sense of (ii) and (iii)), \textit{if and only if, its vector-valued Cartan tensor is traceless}. By (\ref{cartan}), relations $\textrm{tr}\CC=0$ and $\overset{\ast}{\CC}=0$ are equivalent.

The 1-form $$\theta: \hul{X}\in\szel{\pik}\mapsto\theta(\hul{X}):=g(\hul{X},\delta)\in\modc{\Tk M}$$ is said to be the \textit{canonical 1-form} or \textit{Hilbert 1-form} of $(M,F)$. It may be seen immediately that
\begin{gather}\label{can55}
\theta=F\vl{\nabla}F=\vl{\nabla}E.
\end{gather}
The section
\begin{gather}\label{nsef}
\ell:=\frac{1}{F}\delta
\end{gather}
is traditionally called the \textit{normalized support element field} of $(M,F)$. To justify the attribute 'normalized' we note that
\begin{gather}\label{elelegy}
g(\ell,\ell)=1.
\end{gather}
Indeed,
\begin{center}
$\displaystyle g(\ell,\ell)=\frac{1}{F^2}g(\delta,\delta)\overset{(\ref{beta})}{=}\frac{1}{F^2}\cdot2E=\frac{F^2}{F^2}=1$.
\end{center}
The dual form of $\ell$ is
\begin{gather}\label{fins47}
\ell_{\flat}:=\frac{1}{F}\theta=\vl{\nabla}F
\end{gather}
since
\begin{gather}\label{can56}
\ell_{\flat}(\ell)=\frac{1}{F^2}\theta(\delta)=\frac{1}{F^2}g(\delta,\delta)=1.
\end{gather}
By the \textit{angular metric tensor} of $(M,F)$ we mean the type $\binom{0}{2}$ tensor
\begin{gather}\label{fins472}
\eta:=g-\ell_{\flat}\otimes\ell_{\flat}=g-\vl{\nabla}F\otimes\vl{\nabla}F
\end{gather}
along $\tauk$. We obtain:
\begin{gather}\label{fins48}
\frac{1}{F}\eta=\vl{\nabla}\vl{\nabla}F.
\end{gather}
Indeed, for any vector fields $X$, $Y$, $Z$ on $M$ we have
\begin{center}
$\displaystyle\frac{1}{F}\eta(\kal{X},\kal{Y})=\frac{1}{F}(g(\kal{X},\kal{Y})-\vl{\nabla}F(\kal{X})\vl{\nabla}F(\kal{Y}))\overset{(\ref{fins43})}{=}\frac{1}{F}(\frac{1}{2}\vl{X}(\vl{Y}F^2)-(\vl{X}F)(\vl{Y}F))=\frac{1}{F}(\vl{X}((\vl{Y}F)F)-(\vl{X}F)(\vl{Y}F))=\vl{X}(\vl{Y}F)=(\vl{\nabla}\vl{\nabla}F)(\kal{X},\kal{Y})$.
\end{center}
We note that (\ref{fins47}) and (\ref{fins48}) imply
\begin{gather}\label{fins49}
\vl{\nabla}\ell_{\flat}=\frac{1}{F}\eta.
\end{gather}

\begin{lemma}
If $(M,F)$ is a Finsler manifold, then for any vector fields $X$, $Y$, $Z$ on $M$ we have
\begin{gather}\label{sym59}
\vl{\nabla}\vl{\nabla}\vl{\nabla}F(\kal{X},\kal{Y},\kal{Z})=\frac{2}{F}\CC_{\flat}(\kal{X},\kal{Y},\kal{Z})-\frac{1}{F^2}\underset{(\kal{X},\kal{Y},\kal{Z})}{\mathfrak{S}}\ell_{\flat}\otimes\eta(\kal{X},\kal{Y},\kal{Z}).
\end{gather}
\end{lemma}

\begin{proof}
\begin{center}
$\displaystyle\vl{\nabla}\vl{\nabla}\vl{\nabla}F(\kal{X},\kal{Y},\kal{Z})\overset{(\ref{fins48})}{=}\nabla_{\vl{X}}\left(\frac{1}{F}\eta\right)(\kal{Y},\kal{Z})=$\\
$\displaystyle\vl{X}\left(\frac{1}{F}\right)\eta(\kal{Y},\kal{Z})+\frac{1}{F}\nabla_{\vl{X}}(g-\ell_{\flat}\otimes\ell_{\flat})(\kal{Y},\kal{Z})=-\frac{1}{F^2}(\vl{X}F)\eta(\kal{Y},\kal{Z})+\frac{1}{F}\left(\nabla_{\vl{X}}g\right)(\kal{Y},\kal{Z})-\frac{1}{F}\left(\nabla_{\vl{X}}\ell_{\flat}\right)(\kal{Y})\ell_{\flat}(\kal{Z})-\frac{1}{F}\ell_{\flat}(\kal{Y})\left(\nabla_{\vl{X}}\ell_{\flat}\right)(\kal{Z})=\frac{1}{F}\vl{\nabla}g(\kal{X},\kal{Y},\kal{Z})-\frac{1}{F^2}\vl{\nabla}F(\kal{X})\eta(\kal{Y},\kal{Z})-$\\
$\displaystyle\frac{1}{F}\left(\vl{\nabla}\ell_{\flat}\right)(\kal{X},\kal{Y})\ell_{\flat}(\kal{Z})-\frac{1}{F}\ell_{\flat}(\kal{Y})\vl{\nabla}\ell_{\flat}(\kal{X},\kal{Z})\overset{(\ref{fins45}),(\ref{fins47}),(\ref{fins49})}{=}\frac{2}{F}\CC_{\flat}(\kal{X},\kal{Y},\kal{Z})-\frac{1}{F^2}\left(\ell_{\flat}\otimes\eta(\kal{X},\kal{Y},\kal{Z})+\ell_{\flat}\otimes\eta(\kal{Z},\kal{X},\kal{Y})+\ell_{\flat}\otimes\eta(\kal{Y},\kal{X},\kal{Z})\right)=\frac{2}{F}\CC_{\flat}(\kal{X},\kal{Y},\kal{Z})-\frac{1}{F^2}\underset{(\kal{X},\kal{Y},\kal{Z})}{\mathfrak{S}}\ell_{\flat}\otimes\eta(\kal{X},\kal{Y},\kal{Z})$,
\end{center}
taking into account in the last step that
\begin{center}
$\displaystyle\vl{\nabla}\ell_{\flat}(\kal{X},\kal{Z})=\frac{1}{F}\eta(\kal{X},\kal{Z})=\frac{1}{F}(g(\kal{X},\kal{Z})-(\vl{X}F)(\vl{Z}F))=\frac{1}{F}(g(\kal{Z},\kal{X})-(\vl{Z}F)(\vl{X}F))=\vl{\nabla}\ell_{\flat}(\kal{Z},\kal{X})$.
\end{center}
\end{proof}

\textbf{Remark.} Let Sym denote the symmetrizer defined by $$(\textrm{Sym}\hul{\AA})(\hul{X},\hul{Y},\hul{Z}):=\underset{(\hul{X},\hul{Y},\hul{Z})}{\mathfrak{S}}\AA(\hul{X},\hul{Y},\hul{Z}),$$ if $\AA\in\tenz{0}{3}{\pik}$. Let $\lambda:=\frac{\vl{\nabla}F}{F}$, $\mu:=\vl{\nabla}\vl{\nabla}F$. Then $$\frac{1}{F^2}\ell_{\flat}\otimes\eta=\frac{\vl{\nabla}F}{F}\otimes\mu,$$ and (\ref{sym59}) may be written in the more concise form
\begin{gather}\label{sym60}
\vl{\nabla}\mu=\frac{2}{F}\CC_{\flat}-\textrm{Sym}(\lambda\otimes\mu).
\end{gather}

If $(M,F)$ is a Finsler manifold, then the 2-form $$\frac{1}{2}d(\nabla F^2\circ\jj)=dd_{\JJ}E$$ is (fibrewise) non-degenerate on $\Tk M$ by (F$_{\textrm{3}}$). So there exists a unique map $$S:TM\rightarrow TTM$$ defined to be zero on $o(M)$, and defined on $\Tk M$ to be the unique vector field such that
\begin{gather}\label{canonspray85}
i_Sdd_{\JJ}E=-dE.
\end{gather}
Then, actually, $S$ \textit{is a spray over} $M$, i.e., has the properties (C$_{\textrm{1}}$)-(C$_{\textrm{6}}$). A very instructive, but quite forgotten proof of this fundamental fact may be found in F. Warner's quoted paper \cite{Warner}. The spray $S$ will be called the \textit{canonical spray} of the Finsler manifold $(M,F)$. As we have indicated in the Introduction, the structure of a Finsler manifold is considerably determined by the properties of its canonical spray. This is reflected by the fact that the conceptually most important special classes of Finsler manifolds may be defined entirely in terms of their canonical sprays. Namely, a Finsler manifold is said to be\\
a \textit{Berwald manifold}, if its canonical spray is a Berwald spray, i.e., its Berwald cuvature vanishes;\\
a \textit{weakly Berwald manifold}, if its canonical spray is weakly Berwald, i.e., its Berwald curvature is traceless;\\
a \textit{Douglas manifold}, if its canonical spray has vanishing Douglas curvature;\\
\textit{isotropic}, if its canonical spray is isotropic.

Positive definite Berwald manifolds have been completely classified by Z. I. Szabó \cite{Szabo}, see also \cite{Szabo2}. The sytematic investigation of Douglas manifolds was initiated by S. Bácsó and M. Matsumoto, and in their papers \cite{BM1}-\cite{BM4} the theory has been considerably developed. A good account on weakly Berwald manifolds is the paper \cite{BY} by S. Bácsó and R. Yoshikawa.

As we have already remarked and we shall prove soon, isotropic Finsler manifolds may be characterized by the vanishing of their Weyl endomorphism. Naturally, by the Weyl endomorphism, fundamental projective curvature tensor and curvature tensor of a Finsler manifold we mean the corresponding data of its canonical spray.\\
%there exists a unique spray $S$ over $M$ such that
%\begin{gather}\label{ref56}
%i_Sd(\vl{\nabla}F^2\circ\jj)=-dF^2,
%\end{gather}
%it is called the \textit{canonical spray} of $(M,F)$. In view of (\ref{gradderc}), relation (\ref{ref56}) may also be written in the form
%\begin{gather}\label{graddere}
%i_Sdd_{\JJ}F^2=-dF^2.
%\end{gather}
%(Of course, both in (\ref{ref56}) and (\ref{graddere}), $F^2$ may be replaced by the energy function $E$.)

If $\HH$ is the Ehresmann connection associated to $S$ according to (\ref{ehre6}), then
\begin{itemize}
\item[(i)] $\HH$ \textit{is homogeneous and torsion-free};
\item[(ii)] $\HH$ \textit{is conservative} in the sense that
\begin{gather}\label{cons}
dF\circ\HH=0\textrm{ }\Leftrightarrow\hl{X}F=\hl{X}E=0\textrm{ , }X\in\modx{M}.
\end{gather}
\end{itemize}

Property (i), as we have already learnt, holds for any Ehresmann connection associated to a spray. Here the new and surprising phenomenon is property (ii), which expresses that \textit{the Finsler function (and hence the energy function) is a first integral of the horizontally lifted vector fields}. We call this Ehresmann connection the \textit{canonical connection} of the Finsler manifold. We note that other terms - \textit{Barthel connection}, \textit{Cartan's nonlinear connection}, \textit{Berwald connection} - are also frequently used in the literature. A recent index-free proof of (ii) can be found in \cite{SzL}. In the last section we shall show that \textit{the canonical connection of a Finsler manifold} $(M,F)$ \textit{is unique} in the sense that there is only one Ehresmann connection over $M$ which satisfies conditions (i), (ii).

\textit{Warning.} The \textit{Berwald connection} $\HH: TM\times_M TM\rightarrow TTM$ associated to the canonical spray of a Finsler manifold and the \textit{Berwald derivative} $\nabla=(\hl{\nabla},\vl{\nabla})$ induced by $\HH$ are essentially different objects: the latter is a covariant derivative operator (`linear connection') in a (special) vector bundle. It will sometimes be mentioned as the \textit{Finslerian Berwald derivative}.

%For a very instructive, but quite forgotten proof of the existence and uniqueness of the canonical spray of a Finsler manifold we refer to F. Warner's quoted paper \cite{Warner}. A recent proof of property (ii) can be found in \cite{SzL}. With the help of an intrinsic reformulation of a theorem of D. Laugwitz (\cite{Laug}, Theorem 15.8.1), in the next section we shall present a new proof of the unicity of the canonical spray.
%

Now we are in a position to give a new index-free proof of the following classical result.

\begin{theorem} \textup{(Berwald - del Castillo - Szabó).} \label{berwszab}
A Finsler manifold is isotropic, if and only if, its Weyl endomorphism is the zero transformation.
\end{theorem}

\begin{proof}
Consider an $n$-dimensional Finsler manifold $(M,F)$. Let $S$ be the canonical spray of $(M,F)$, and $\KK$ the Jacobi endomorphism of $S$. First we show that $$\vl{\nabla}F\circ\KK=0.$$ Indeed, for any vector field $X$ on $M$,
\begin{center}
$\displaystyle\vl{\nabla}F(\KK(\kal{X}))=\ii\KK(\kal{X})F=\vv[S,\hl{X}]F=[S,\hl{X}]F-\hh[S,\hl{X}]F=0$,
\end{center}
since the horizontal vector fields kill the Finsler function.

Now suppose that $(M,F)$ is isotropic. Then the Jacobi endomorphism of $S$ can be written in the form $$\KK=K\mathbf{1}+\hul{\alpha}\otimes\delta\textrm{ , }K:=\frac{1}{n-1}\textrm{tr}\KK\textrm{ , }\hul{\alpha}\in\tenz{0}{1}{\pik}$$ (see the last paragraphs of Chapter \ref{fej4}). By our previous observation, for any vector field $X$ on $M$ we have
\begin{center}
$\displaystyle 0=\vl{\nabla}F(\KK(\kal{X}))=\vl{\nabla}F(K\kal{X})+\vl{\nabla}F(\hul{\alpha}(\kal{X})\delta)=K(\vl{X}F)+\hul{\alpha}(\kal{X})CF=K\vl{\nabla}F(\kal{X})+\hul{\alpha}(\kal{X})F$,
\end{center}
hence $\hul{\alpha}=-\frac{K}{F}\vl{\nabla}F$. So it follows that the Jacobi endomorphism of an isotropic Finsler manifold may be written in the form
\begin{gather}\label{berwszabp}
\KK=K(\mathbf{1}-\frac{1}{F}\vl{\nabla}F\otimes\delta)=K(\mathbf{1}-\ell_{\flat}\otimes\ell),
\end{gather}
therefore its Weyl endomorphism takes the form $$\mathbf{W}^{\circ}=\frac{1}{n+1}(\vl{\nabla}K-\textrm{tr}\vl{\nabla}\KK-(n+1)\frac{K}{F}\vl{\nabla}F)\otimes\delta.$$
Now we evaluate the vertical differential $\vl{\nabla}\KK$ at a pair $(\kal{X},\kal{Y})$ of basic vector fields:
\begin{center}
$\displaystyle\vl{\nabla}\KK(\kal{X},\kal{Y})=\nabla_{\vl{X}}(\KK(\kal{Y}))=\nabla_{\vl{X}}(K\kal{Y}-\frac{K}{F}(\vl{Y}F)\delta)=(\vl{X}K)\kal{Y}-\vl{X}(\frac{K}{F}(\vl{Y}F))\delta-\frac{K}{F}(\vl{Y}F)\kal{X}=(\vl{\nabla}K\otimes\mathbf{1})(\kal{X},\kal{Y})+\frac{K}{F^2}(\vl{X}F)(\vl{Y}F)\delta-\frac{1}{F}(\vl{X}K)(\vl{Y}F)\delta-\frac{K}{F}\vl{X}(\vl{Y}F)\delta-\frac{K}{F}(\mathbf{1}\otimes\vl{\nabla}F)(\kal{X},\kal{Y})$.
\end{center}
Thus we find
\begin{center}
$\displaystyle\vl{\nabla}\KK=\vl{\nabla}K\otimes\mathbf{1}+\frac{K}{F^2}\vl{\nabla}F\otimes\vl{\nabla}F\otimes\delta$\\
$\displaystyle-\frac{1}{F}\vl{\nabla}K\otimes\vl{\nabla}F\otimes\delta-\frac{K}{F}\vl{\nabla}\vl{\nabla}F\otimes\delta-\frac{K}{F}\mathbf{1}\otimes\vl{\nabla}F$.
\end{center}
In the next step we calculate the trace of the right-hand side of this relation, term by term.

In view of (\ref{trace02}), for any basic vector field $X$ on $M$ we have
\begin{center}
$\displaystyle i_{\kal{X}}\textrm{tr}(\vl{\nabla}K\otimes\mathbf{1})=\textrm{tr}(\vl{\nabla}K\otimes\kal{X})\overset{(\ref{trace01})}{=}\vl{\nabla}K(\kal{X})$,
\end{center}
hence $$\textrm{tr}(\vl{\nabla}K\otimes\mathbf{1})=\vl{\nabla}K.$$ Applying (\ref{trace03}),
\begin{center}
$\displaystyle\textrm{tr}(\vl{\nabla}F\otimes\vl{\nabla}F\otimes\delta)=i_{\delta}(\vl{\nabla}F\otimes\vl{\nabla}F)=\vl{\nabla}F(\delta)\vl{\nabla}F=(CF)\vl{\nabla}F=F\vl{\nabla}F$.
\end{center}
By our inductive definition (\ref{trace02}) again,
\begin{center}
$\displaystyle i_{\kal{X}}\textrm{tr}(\vl{\nabla}K\otimes\vl{\nabla}F\otimes\delta)=\textrm{tr}(\vl{\nabla}K\otimes\vl{X}F\otimes\delta)=$\\
$\displaystyle \vl{X}F\textrm{tr}(\vl{\nabla}K\otimes\delta)\overset{(\ref{trace01})}{=}(\vl{X}F)CK=2K\vl{\nabla}F(\kal{X})$,
\end{center}
since the second-degree homogeneity of $\KK$ implies that the function $K$ is also (positive-)homogeneous of degree 2.

As to the fourth term,
\begin{center}
$\displaystyle\textrm{tr}(\vl{\nabla}\vl{\nabla}F\otimes\delta)\overset{(\ref{trace03})}{=}i_{\delta}\vl{\nabla}\vl{\nabla}F=\nabla_C\vl{\nabla}F=0$,
\end{center}
since for all $X\in\modx{M}$,
\begin{center}
$\displaystyle(\nabla_C\vl{\nabla}F)(\kal{X})=C(\vl{X}F)=[C,\vl{X}]F+\vl{X}(CF)=-\vl{X}F+\vl{X}F=0$.
\end{center}
Finally,
\begin{center}
$\displaystyle i_{\kal{X}}\textrm{tr}(\mathbf{1}\otimes\vl{\nabla}F)=\textrm{tr}(\mathbf{1}\otimes\vl{X}F)=\textrm{tr}((\vl{X}F)\mathbf{1})=(\vl{X}F)\textrm{tr}\mathbf{1}=n\vl{\nabla}F(\kal{X})$,
\end{center}
hence
\begin{center}
$\displaystyle\textrm{tr}(\mathbf{1}\otimes\vl{\nabla}F)=n\vl{\nabla}F$.
\end{center}
To sum up,
\begin{center}
$\displaystyle\textrm{tr}\vl{\nabla}\KK=\vl{\nabla}K+\frac{K}{F}\vl{\nabla}F-2\frac{K}{F}\vl{\nabla}F-n\frac{K}{F}\vl{\nabla}F=\vl{\nabla}K-(n+1)\frac{K}{F}\vl{\nabla}F$,
\end{center}
whence
\begin{center}
$\displaystyle\mathbf{W}^{\circ}=\frac{1}{n+1}(\vl{\nabla}K-\vl{\nabla}K+(n+1)\frac{K}{F}\vl{\nabla}F-(n+1)\frac{K}{F}\vl{\nabla}F)\otimes\delta=0$.
\end{center}

Thus we proved that if $(M,F)$ is isotropic, then its Weyl endomorphism is zero.

Conversely, suppose that $\mathbf{W}^{\circ}=0$. Then, by (\ref{WWeyl}), for any vector field on $M$ we have $$\KK(\kal{X})=K\kal{X}-\frac{1}{n+1}(\vl{\nabla}K-\textrm{tr}\vl{\nabla}\KK)(\kal{X})\delta.$$ Hence relation $\vl{\nabla}F\circ\KK=0$ yields $$K\vl{\nabla}F(\kal{X})=\frac{1}{n+1}(\vl{\nabla}K-\textrm{tr}\vl{\nabla}\KK)(\kal{X})\vl{\nabla}F(\delta),$$ whence $$\frac{1}{n+1}(\vl{\nabla}K-\textrm{tr}\vl{\nabla}\KK)=\frac{K}{F}\vl{\nabla}F,$$ and so $$\KK=K(\mathbf{1}-\frac{K}{F}\vl{\nabla}F\otimes\delta).$$ This concludes the proof.
\end{proof}

\textbf{Remark.} The property
\begin{center}
'\textit{an at least 3-dimensional isotropic Finsler manifold has zero Weyl endomorphism}'
\end{center}
was discovered and proved by Berwald \cite{Berwald}. The converse is due to del Castillo \cite{delc} and, independently, Z. I. Szabó \cite{szabodel}. The simple proof presented here for this part of the problem was inspired by a paper of P. N. Pandey \cite{Pandey}.\\

In the following, by the \textit{scalar curvature} of an isotropic Finsler manifold we shall mean the function $$R:=\frac{K}{F^2}=\frac{1}{(n-1)F^2}\textrm{tr}\KK$$ (see e.g. \cite{Rund}, p. 147), which is positive-homogeneous of degree 0. An isotropic Finsler manifold is said to be \textit{of constant curvature}, if its scalar curvature is a constant function.

\chapter{The Landsberg tensor and the stretch tensor of a Finsler manifold}\label{ch7}

Using the Finslerian h-Berwald covariant derivative, we define the \textit{Landsberg tensor} $\Pp$ of a Finsler manifold $(M,F)$ by the following formula:
\begin{gather}\label{rela}
\Pp:=-\frac{1}{2}\hl{\nabla}g.
\end{gather}

\begin{lemma}\label{lem711}
For all vector fields $X,Y,Z\in\modx{M}$ we have
\begin{gather}\label{relb}
\hl{\nabla}g(\kal{X},\kal{Y},\kal{Z})=(\ii\BB(\kal{X},\kal{Y})\kal{Z})E,
\end{gather}
therefore the Berwald curvature and the Landsberg tensor of a Finsler manifold are related by
\begin{gather}\label{relc}
\vl{\nabla}E\circ\BB=-2\Pp.
\end{gather}
\end{lemma}

\begin{proof}
\begin{center}
$\displaystyle\hl{\nabla}g(\kal{X},\kal{Y},\kal{Z})=(\nabla_{\hl{X}}g)(\kal{Y},\kal{Z})=\hl{X}g(\kal{Y},\kal{Z})-g(\nabla_{\hl{X}}\kal{Y},\kal{Z})-g(\kal{Y},\nabla_{\hl{X}}\kal{Z})\overset{\textrm{(\ref{fins43}),(\ref{alpha})}}{=}\hl{X}(\vl{Y}(\vl{Z}E)))-(\ii\nabla_{\hl{X}}\kal{Y})\vl{Z}E+(\ii\nabla_{\ii\nabla_{\hl{X}}\kal{Y}}\kal{Z})E-\vl{Y}(\ii\nabla_{\hl{X}}\kal{Z})E+\JJ[\vl{Y},\HH\nabla_{\hl{X}}\kal{Z}]E=\hl{X}(\vl{Y}(\vl{Z}E))-[\hl{X},\vl{Y}]\vl{Z}E-\vl{Y}([\hl{X},\vl{Z}]E)+\JJ[\vl{Y},\HH\VV[\hl{X},\vl{Z}]]E$.
\end{center}
Here, as we have already seen in the proof of \ref{lemma31},
\begin{center}
$\displaystyle\JJ[\vl{Y},\HH\VV[\hl{X},\vl{Z}]]=[\vl{Y},[\hl{X},\vl{Z}]]=\ii\BB(\kal{Y},\kal{X})\kal{Z}=\ii\BB(\kal{X},\kal{Y})\kal{Z}$
\end{center}
therefore
\begin{center}
$\displaystyle\hl{\nabla}g(\kal{X},\kal{Y},\kal{Z})=(\ii\BB(\kal{X},\kal{Y})\kal{Z})E+\vl{Y}(\hl{X}(\vl{Z}E))-$\\
$\vl{Y}(\hl{X}(\vl{Z}E)-\vl{Z}(\hl{X}E))=(\ii\BB(\kal{X},\kal{Y})\kal{Z})E$,
\end{center}
taking into account that $\HH$ is conservative. Thus we have proved relation (\ref{relb}). Relation (\ref{relc}) is merely a reformulation of (\ref{relb}).
\end{proof}

\textbf{Remark.} An immediate consequence of formula (\ref{relc}) is that \textit{any Berwald manifolds have vanishing Landsberg tensor}. The converse, whether there are Finsler manifolds with vanishing Landsberg tensor which are not of Berwald type is a long-standing question in Finsler geometry, and in spite of many efforts, it has not been answered until now.

\begin{coro}\label{coro63}
The Landsberg tensor of a Finsler manifold has the following properties:
\begin{itemize}
\item[(i)] \textit{it is totally symmetric};
\item[(ii)] $\delta\in\left\{\hul{X},\hul{Y},\hul{Z}\right\}\textrm{  }\Rightarrow\textrm{  }\Pp(\hul{X},\hul{Y},\hul{Z})=0$;
\item[(iii)] $\nabla_C\Pp=0$, i.e., $\Pp$ \textit{is homogeneous of degree zero}.
\end{itemize}
\end{coro}

Indeed, (i) and (ii) are immediate consequences of (\ref{relb}) and the corresponding property of the Berwald tensor. Taking into account our calculations in the proof of \ref{lemma35}, we get for any vector fields $X$, $Y$, $Z$ on $M$
\begin{center}
$\displaystyle\left(\nabla_C\Pp\right)(\kal{X},\kal{Y},\kal{Z})=C(\Pp(\kal{X},\kal{Y},\kal{Z}))=-\frac{1}{2}C(\ii\BB(\kal{X},\kal{Y})\kal{Z})E=-\frac{1}{2}([C,\ii\BB(\kal{X},\kal{Y})\kal{Z}]E+(\ii\BB(\kal{X},\kal{Y})\kal{Z})CE)=0$,
\end{center}
which proves (iii).

\begin{coro}\label{coro73}
If $g$ is the metric tensor, $S$ is the canonical spray of a Finsler manifold, then $\nabla_Sg=0$.
\end{coro}

\begin{proof}
For any sections $\hul{X},\hul{Y}$ in $\szel{\pik}$,
\begin{center}
$\displaystyle\left(\nabla_Sg\right)(\hul{X},\hul{Y})=\left(\hl{\nabla}_{\delta}g\right)(\hul{X},\hul{Y})=\hl{\nabla}g(\delta,\hul{X},\hul{Y})=-2\Pp(\delta,\hul{X},\hul{Y})\overset{\textrm{\ref{coro63} (ii)}}{=}0$.
\end{center}
\end{proof}

Now we can easily deduce an important relation between the Cartan tensor and the Landsberg tensor of a Finsler manifold.

\begin{propo}\label{propo65}
$\nabla_S\CC_{\flat}=\Pp$.
\end{propo}

\begin{proof}
Let $X$, $Y$, $Z$ be vector fields on $M$. Applying the hv-Ricci formula (\ref{ric34}), property (\ref{49ny}), and Corollary \ref{coro63} (ii), we obtain:
\begin{center}
$\displaystyle 2\left(\nabla_S\CC_{\flat}\right)(\kal{X},\kal{Y},\kal{Z})=\hl{\nabla}_{\delta}\vl{\nabla}g(\kal{X},\kal{Y},\kal{Z})=\hl{\nabla}\vl{\nabla}g(\delta,\kal{X},\kal{Y},\kal{Z})=\vl{\nabla}\hl{\nabla}g(\kal{X},\delta,\kal{Y},\kal{Z})=\nabla_{\vl{X}}\hl{\nabla}g(\delta,\kal{Y},\kal{Z})=-2\left(\nabla_{\vl{X}}\Pp\right)(\delta,\kal{Y},\kal{Z})=-2\vl{X}\Pp(\delta,\kal{Y},\kal{Z})+2\Pp(\nabla_{\vl{X}}\delta,\kal{Y},\kal{Z})=2\Pp(\kal{X},\kal{Y},\kal{Z})$.
\end{center}
This proves the Proposition.
\end{proof}

Having these results, we can present a simple index-free proof of Proposition 3.1 in our paper \cite{BSz2}.

\begin{propo} \label{landonly}
If the Landsberg tensor of a Finsler manifold depends only on the position, then it vanishes identically, i.e., $\vl{\nabla}\Pp=0$ implies that $\Pp=0$.
\end{propo}

\begin{proof}
Applying the preceding Proposition, the Ricci formula (\ref{ric34}), and taking into account Corollary \ref{lem5}, for any vector fields $X$, $Y$, $Z$ on $M$ we have
\begin{center}
$\displaystyle\Pp(\kal{X},\kal{Y},\kal{Z})=(\nabla_S\CC_{\flat})(\kal{X},\kal{Y},\kal{Z})=(\hl{\nabla}\CC_{\flat})(\delta,\kal{X},\kal{Y},\kal{Z})=\frac{1}{2}(\hl{\nabla}\vl{\nabla}g)(\delta,\kal{X},\kal{Y},\kal{Z})=\frac{1}{2}(\vl{\nabla}\hl{\nabla}g)(\kal{X},\delta,\kal{Y},\kal{Z})=-(\vl{\nabla}\Pp)(\kal{X},\delta,\kal{Y},\kal{Z})=0$.
\end{center}
\end{proof}

By the \textit{stretch tensor} of $(M,F)$ we mean the tensor $\mathbf{\Sigma}\in\tenz{0}{4}{\pik}$ given by
\begin{gather}\label{20}
\frac{1}{2}\mathbf{\Sigma}(\hul{X},\hul{Y},\hul{Z},\hul{U}):=\hl{\nabla}\mathbf{P}(\hul{X},\hul{Y},\hul{Z},\hul{U})-\hl{\nabla}\mathbf{P}(\hul{Y},\hul{X},\hul{Z},\hul{U}).
\end{gather}

\begin{coro}
The stretch tensor of a Finsler manifold is homogeneous of degree zero.
\end{coro}

\begin{proof}
It is enough to show that $\nabla_C\hl{\nabla}\Pp=0$. Let $X$, $Y$, $Z$, $U$ be vector fields on $M$. Using the Ricci formula (\ref{ric2}) taking into account Corollary \ref{coro63} (ii), and the homogeneity property $\hl{\nabla}\delta=0$, we obtain
\begin{center}
$\displaystyle(\nabla_C\hl{\nabla}\Pp)(\kal{X},\kal{Y},\kal{Z},\kal{U})=(\vl{\nabla}\hl{\nabla}\Pp)(\delta,\kal{X},\kal{Y},\kal{Z},\kal{U})=\hl{\nabla}\vl{\nabla}\Pp(\kal{X},\delta,\kal{Y},\kal{Z},\kal{U})=(\nabla_{\hl{X}}\vl{\nabla}\Pp)(\delta,\kal{Y},\kal{Z},\kal{U})=\hl{X}(\vl{\nabla}\Pp(\delta,\kal{Y},\kal{Z},\kal{U}))-\vl{\nabla}\Pp(\delta,\nabla_{\hl{X}}\kal{Y},\kal{Z},\kal{U})-\vl{\nabla}\Pp(\delta,\kal{Y},\nabla_{\hl{X}}\kal{Z},\kal{U})-\vl{\nabla}\Pp(\delta,\kal{Y},\kal{Z},\nabla_{\hl{X}}\kal{U})=\hl{X}((\nabla_C\Pp)(\kal{Y},\kal{Z},\kal{U}))-(\nabla_C\Pp)(\nabla_{\hl{X}}\kal{Y},\kal{Z},\kal{U})-(\nabla_C\Pp)(\kal{Y},\nabla_{\hl{X}}\kal{Z},\kal{U})-(\nabla_C\Pp)(\kal{Y},\kal{Z},\nabla_{\hl{X}}\kal{U})=0$,
\end{center}
since $\Pp$ is homogeneous of degree 0.
\end{proof}

Now we verify by an index-free argument the following result of our above mentioned paper \cite{BSz2}:

\begin{propo}\label{stretchdepp}
If the stretch tensor of a Finsler manifold depends only on the position, then it vanishes identically, i.e., $\vl{\nabla}\Sigma=0$ implies that $\Sigma=0$.
\end{propo}

\begin{proof}
Let $X$, $Y$, $Z$, $U$, $V$ be vector fields on $M$.

\textbf{Step 1} By our assumption
\begin{center}
$\displaystyle 0=(\vl{\nabla}\Sigma)(\kal{X},\kal{Y},\kal{Z},\kal{U},\kal{V})=\vl{X}(\Sigma(\kal{Y},\kal{Z},\kal{U},\kal{V}))=2(\vl{X}(\hl{\nabla}\Pp(\kal{Y},\kal{Z},\kal{U},\kal{V})-\hl{\nabla}\Pp(\kal{Z},\kal{Y},\kal{U},\kal{V})))=2(\vl{\nabla}\hl{\nabla}\Pp(\kal{X},\kal{Y},\kal{Z},\kal{U},\kal{V})-\vl{\nabla}\hl{\nabla}\Pp(\kal{X},\kal{Z},\kal{Y},\kal{U},\kal{V})))$,
\end{center}
hence $$\vl{\nabla}\hl{\nabla}\Pp(\kal{X},\kal{Y},\kal{Z},\kal{U},\kal{V})=\vl{\nabla}\hl{\nabla}\Pp(\kal{X},\kal{Z},\kal{Y},\kal{U},\kal{V}).$$ Since this is a tensorian relation, we also have
\begin{gather}\label{salpha}
\vl{\nabla}\hl{\nabla}\Pp(\kal{X},\kal{Y},\kal{Z},\kal{U},\delta)=\vl{\nabla}\hl{\nabla}\Pp(\kal{X},\kal{Z},\kal{Y},\kal{U},\delta).
\end{gather}
Now, by the Ricci formula (\ref{ric34}) and Corollary \ref{coro63} (ii),
\begin{center}
$\displaystyle\hl{\nabla}\vl{\nabla}\Pp(\kal{X},\kal{Y},\kal{Z},\kal{U},\delta)=\vl{\nabla}\hl{\nabla}\Pp(\kal{Y},\kal{X},\kal{Z},\kal{U},\delta)\overset{\textrm{(\ref{salpha})}}{=}\vl{\nabla}\hl{\nabla}\Pp(\kal{Y},\kal{Z},\kal{X},\kal{U},\delta)\overset{\textrm{(\ref{ric34})}}{=}\hl{\nabla}\vl{\nabla}\Pp(\kal{Z},\kal{Y},\kal{X},\kal{U},\delta)$,
\end{center}
so $\hl{\nabla}\vl{\nabla}\Pp(.,.,.,.,\delta)$ is symmetric in its first and third variables:
\begin{gather}\label{sbeta}
\hl{\nabla}\vl{\nabla}\Pp(\kal{X},\kal{Y},\kal{Z},\kal{U},\delta)=\hl{\nabla}\vl{\nabla}\Pp(\kal{Z},\kal{Y},\kal{X},\kal{U},\delta).
\end{gather}

\textbf{Step 2} We show that
\begin{gather}\label{sgamma}
\hl{\nabla}\Pp(\kal{X},\kal{Y},\kal{Z},\kal{U})+\hl{\nabla}\vl{\nabla}\Pp(\kal{X},\kal{Y},\kal{Z},\kal{U},\delta)=0.
\end{gather}
We start out the identity $\Pp(\kal{Z},\kal{U},\delta)=0$. Operating on both sides first by $\vl{Y}$, and next by $\hl{X}$, we obtain
\begin{center}
$\displaystyle 0=\vl{Y}(\Pp(\kal{Z},\kal{U},\delta))=\vl{\nabla}\Pp(\kal{Y},\kal{Z},\kal{U},\delta)+\Pp(\kal{Z},\kal{U},\kal{Y})=\Pp(\kal{Y},\kal{Z},\kal{U})+\vl{\nabla}\Pp(\kal{Y},\kal{Z},\kal{U},\delta)$;
\end{center}

\begin{center}
$\displaystyle 0=\hl{X}(\Pp(\kal{Y},\kal{Z},\kal{U}))+\hl{X}(\vl{\nabla}\Pp(\kal{Y},\kal{Z},\kal{U},\delta))=(\hl{\nabla}\Pp)(\kal{X},\kal{Y},\kal{Z},\kal{U})+(\hl{\nabla}\vl{\nabla}\Pp)(\kal{X},\kal{Y},\kal{Z},\kal{U},\delta)+\Pp(\nabla_{\hl{X}}\kal{Y},\kal{Z},\kal{U})+\Pp(\kal{Y},\nabla_{\hl{X}}\kal{Z},\kal{U})+\Pp(\kal{Y},\kal{Z},\nabla_{\hl{X}}\kal{U})+\vl{\nabla}\Pp(\nabla_{\hl{X}}\kal{Y},\kal{Z},\kal{U},\delta)+\vl{\nabla}\Pp(\kal{Y},\nabla_{\hl{X}}\kal{Z},\kal{U},\delta)+\vl{\nabla}\Pp(\kal{Y},\kal{Z},\nabla_{\hl{X}}\kal{U},\delta)$.
\end{center}
Since, for example,
\begin{center}
$\displaystyle\vl{\nabla}\Pp(\nabla_{\hl{X}}\kal{Y},\kal{Z},\kal{U},\delta)=[\hl{X},\vl{Y}]\Pp(\kal{Z},\kal{U},\delta)-\Pp(\kal{Z},\kal{U},\nabla_{[\hl{X},\vl{Y}]}\delta)=-\Pp(\kal{Z},\kal{U},\VV[\hl{X},\vl{Y}])=-\Pp(\kal{Z},\kal{U},\nabla_{\hl{X}}\kal{Y})=-\Pp(\nabla_{\hl{X}}\kal{Y},\kal{Z},\kal{U})$,
\end{center}
the last six terms cancel in pairs on the right-hand side of the above relation. So we get (\ref{sgamma}).

\textbf{Step 3} Interchanging $\kal{X}$ and $\kal{Y}$ in (\ref{sgamma}), we find
\begin{center}
$\displaystyle 0=\hl{\nabla}\Pp(\kal{Y},\kal{X},\kal{Z},\kal{U})+\hl{\nabla}\vl{\nabla}\Pp(\kal{Y},\kal{X},\kal{Z},\kal{U},\delta)=\hl{\nabla}\Pp(\kal{Y},\kal{X},\kal{Z},\kal{U})+\hl{\nabla}\vl{\nabla}\Pp(\kal{X},\kal{Y},\kal{Z},\kal{U},\delta)$,
\end{center}
since, by Step 1, the second term does not change under the permutations $$(\kal{Y},\kal{X},\kal{Z})\rightarrow(\kal{Y},\kal{Z},\kal{X})\rightarrow(\kal{X},\kal{Z},\kal{Y})\rightarrow(\kal{X},\kal{Y},\kal{Z}).$$
The last relation and (\ref{sgamma}) imply that $$\hl{\nabla}\Pp(\kal{X},\kal{Y},\kal{Z},\kal{U})=\hl{\nabla}\Pp(\kal{Y},\kal{X},\kal{Z},\kal{U})$$ whence $\Sigma=0$.
\end{proof}

The next important observation gives an index-free reformulation of relation (3.3.2.5) in \cite{Matsumoto2}. For completeness we present here a different proof, using our formalism.

\begin{propo}\label{4.2}
For any sections $\hul{X},\hul{Y},\hul{Z},\hul{U}$ in $\Gamma(\pik)$,
\begin{gather}
\vl{\nabla}E\circ\vl{\nabla}\mathbf{H}(\hul{X},\hul{Y},\hul{Z},\hul{U})=\mathbf{\Sigma}(\hul{Z},\hul{Y},\hul{X},\hul{U}).
\end{gather}
\end{propo}

\begin{proof}
It is enough to check the relation for basic vector fields $\kal{X},\kal{Y},\kal{Z},\kal{U}$. Then
\begin{center}
$\displaystyle\vl{\nabla}E(\vl{\nabla}\mathbf{H}(\kal{X},\kal{Y},\kal{Z},\kal{U}))\overset{(\ref{gradderc})}{=}(\ii\vl{\nabla}\mathbf{H}(\kal{X},\kal{Y},\kal{Z},\kal{U}))E\overset{(\ref{bian3})}{=}\ii(\hl{\nabla}\mathbf{B}(\kal{Y},\kal{Z},\kal{X},\kal{U})-\hl{\nabla}\mathbf{B}(\kal{Z},\kal{Y},\kal{X},\kal{U}))E$.
\end{center}
Here
\begin{center}
$\displaystyle\hl{\nabla}\mathbf{B}(\kal{Y},\kal{Z},\kal{X},\kal{U})=(\nabla_{\hl{Y}}\mathbf{B})(\kal{Z},\kal{X},\kal{U})=\nabla_{\hl{Y}}(\mathbf{B}(\kal{Z},\kal{X})\kal{U})-\mathbf{B}(\nabla_{\hl{Y}}\kal{Z},\kal{X})\kal{U}-\mathbf{B}(\kal{Z},\nabla_{\hl{Y}}\kal{X})\kal{U}-\mathbf{B}(\kal{Z},\kal{X})\nabla_{\hl{Y}}\kal{U}$,
\end{center}
and by (\ref{15})
\begin{center}
$\displaystyle\nabla_{\hl{Y}}\mathbf{B}(\kal{Z},\kal{X})\kal{U}=\VV\left[\hl{Y},\ii\mathbf{B}(\kal{Z},\kal{X})\kal{U}\right]$.
\end{center}
Therefore, applying (\ref{relc}) we get
\begin{center}
$\displaystyle\ii\hl{\nabla}\mathbf{B}(\kal{Y},\kal{Z},\kal{X},\kal{U})E=\left[\hl{Y},\ii\mathbf{B}(\kal{Z},\kal{X})\kal{U}\right]E+$\\
$\displaystyle2\mathbf{P}(\nabla_{\hl{Y}}\kal{Z},\kal{X},\kal{U})+2\mathbf{P}(\kal{Z},\nabla_{\hl{Y}}\kal{X},\kal{U})+2\mathbf{P}(\kal{Z},\kal{X},\nabla_{\hl{Y}}\kal{U})$.
\end{center}
Since $\hl{Y}E=0$ by (\ref{cons}), at the right-hand side the first term is
\begin{center}
$\displaystyle\hl{Y}((\ii\mathbf{B}(\kal{Z},\kal{X})\kal{U})E)\overset{(\ref{relc})}{=}-2\hl{Y}\mathbf{P}(\kal{Z},\kal{X},\kal{U})$,
\end{center}
and hence $$\ii\hl{\nabla}\BB(\kal{Y},\kal{Z},\kal{X},\kal{U})E=-2\hl{\nabla}\mathbf{P}(\kal{Y},\kal{Z},\kal{X},\kal{U}).$$ In the same way we find that
\begin{center}
$\displaystyle\ii\hl{\nabla}\mathbf{B}(\kal{Z},\kal{Y},\kal{X},\kal{U})E=-2\hl{\nabla}\mathbf{P}(\kal{Z},\kal{Y},\kal{X},\kal{U})$.
\end{center}
Hence
\begin{center}
$\displaystyle\vl{\nabla}E(\vl{\nabla}\mathbf{H}(\kal{X},\kal{Y},\kal{Z},\kal{U}))=2(\hl{\nabla}\mathbf{P}(\kal{Z},\kal{Y},\kal{X},\kal{U})-\hl{\nabla}\mathbf{P}(\kal{Y},\kal{Z},\kal{X},\kal{U}))\overset{(\ref{20})}{=}\mathbf{\Sigma}(\kal{Z},\kal{Y},\kal{X},\kal{U})$,
\end{center}
as was to be proved.
\end{proof}

\begin{coro}\label{4.3}
$R$-quadratic Finsler manifolds have vanishing stretch tensor.
\end{coro}

\newpage

\chapter{Orthogonal projection along the canonical section and its applications}\label{sec5}

\begin{center}
\textit{Throughout, let $(M,F)$ be an $n$-dimensional Finsler manifold, $n\geq 2$.}
\end{center}

We define a map $\mathbf{p}: \szel{\pik}\rightarrow\szel{\pik}$ by
\begin{gather}\label{pbA}
\mathbf{p}(\hul{X}):=\hul{X}-\frac{g(\hul{X},\delta)}{g(\delta,\delta)}\delta.
\end{gather}

Then $\mathbf{p}$ is obviously $\modc{\Tk M}$-linear, i.e., $$\mathbf{p}\in\textrm{End}(\szel{\pik})\cong\tenz{1}{1}{\pik},$$ and  $$\textrm{Ker}(\mathbf{p})=\textrm{span}(\delta).$$ For any section $\hul{X}\in\szel{\pik}$ we have
\begin{center}
$\displaystyle\mathbf{p}^2(\hul{X})=\mathbf{p}(\hul{X}-\frac{g(\hul{X},\delta)}{g(\delta,\delta)}\delta)=\mathbf{p}(\hul{X})-\frac{g(\hul{X},\delta)}{g(\delta,\delta)}\mathbf{p}(\delta)=\mathbf{p}(\hul{X})$;
\end{center}
hence $\mathbf{p}^2=\mathbf{p}$. So $\mathbf{p}$ is a projection operator onto $$(\textrm{span}(\delta))^{\perp}:=\left\{\hul{X}\in\szel{\pik}\textrm{  }\vert\textrm{  }g(\hul{X},\delta)=0\right\}$$ along $\textrm{span}(\delta)$, called the \textit{orthogonal projection along} $\delta$.

Since $g(\delta,\delta)=2E$, and by (\ref{can55}) $$g(\hul{X},\delta)=:\theta(\hul{X})=\vl{\nabla}E(\hul{X})=F\nabla F(\hul{X}),$$ $\mathbf{p}$ may be written in the more compact forms
\begin{gather}\label{22}
\mathbf{p}=\mathbf{1}-\frac{1}{2E}\vl{\nabla}E\otimes\delta=\mathbf{1}-\frac{1}{F}\vl{\nabla}F\otimes\delta.
\end{gather}

\begin{lemma}\label{ortprojvantrac}
The h-Berwald differential of the orthogonal projection along $\delta$ vanishes. Its trace is $n-1$.
\end{lemma}

\begin{proof}
For any vector fields $X$, $Y$ on $M$ we have
\begin{center}
$\displaystyle\hl{\nabla}\mathbf{p}(\kal{X},\kal{Y})=(\nabla_{\hl{X}}\mathbf{p})(\kal{Y})=\nabla_{\hl{X}}\mathbf{p}(\kal{Y})-\mathbf{p}(\nabla_{\hl{X}}\kal{Y})=\VV[\hl{X},\ii\mathbf{p}(\kal{Y})]-\mathbf{p}(\VV[\hl{X},\vl{Y}])=\VV[\hl{X},\vl{Y}-\frac{\vl{Y}F}{F}C]-\VV[\hl{X},\vl{Y}]+\frac{1}{F}[\hl{X},\vl{Y}]F\delta=\VV[\hl{X},\vl{Y}]-\VV[\hl{X},\frac{\vl{Y}F}{F}C]-\VV[\hl{X},\vl{Y}]+\frac{1}{F}\hl{X}(\vl{Y}F)\delta=-\hl{X}\left(\frac{\vl{Y}F}{F}\right)\delta+\frac{1}{F}\hl{X}(\vl{Y}F)\delta=-\frac{1}{F}\hl{X}(\vl{Y}F)\delta+\frac{1}{F}\hl{X}(\vl{Y}F)\delta=0$.
\end{center}
The second assertion is immediate:
\begin{center}
$\displaystyle\textrm{tr}\mathbf{p}\overset{(\ref{22})}{=}\textrm{tr}(\mathbf{1}-\frac{1}{F}\vl{\nabla}F\otimes\delta)=\textrm{tr}\mathbf{1}-\frac{1}{F}\textrm{tr}(\vl{\nabla}F\otimes\delta)\overset{(\ref{trace01})}{=}n-\frac{1}{F}\vl{\nabla}F(\delta)=n-\frac{1}{F}CF=n-1$.
\end{center}
\end{proof}

By the \textit{projected tensor} of a tensor of $\mathbf{K}\in\tenz{0}{k}{\pik}$ or of  $\mathbf{L}\in\tenz{1}{k}{\pik}$ we mean the tensors $\mathbf{pK}$ and $\mathbf{pL}$ given by $$\mathbf{pK}(\hul{X}_1,\dots ,\hul{X}_k):=\mathbf{K}(\mathbf{p}\hul{X}_1,\dots ,\mathbf{p}\hul{X}_k)$$ and $$\mathbf{pL}(\hul{X}_1,\dots ,\hul{X}_k):=\mathbf{p}(\mathbf{L}(\mathbf{p}\hul{X}_1,\dots ,\mathbf{p}\hul{X}_k)),$$ respectively.

The following observation is immediate from the definitions.

\begin{lemma}\label{5.2}
Let $\mathbf{K}\in\tenz{0}{k}{\pik}$, $\mathbf{L}\in\tenz{1}{k}{\pik}$ . If
\begin{center}
$\displaystyle\delta\in\left\{\hul{X}_1,\dots ,\hul{X}_k\right\}\Rightarrow \mathbf{K}(\hul{X}_1,\dots ,\hul{X}_k)=0\textrm{ , }\mathbf{L}(\hul{X}_1,\dots ,\hul{X}_k)=0$,
\end{center}
then $\mathbf{pK}=\mathbf{K}$, $\mathbf{pL}=\mathbf{p}\circ\mathbf{L}$.
\end{lemma}

\textbf{Examples.} (1) By the Lemma and Corollaries \ref{lem5}, \ref{coro63}(ii), the Cartan tensors and the Landsberg tensor remain fix under orthogonal projection along $\delta$: $$\mathbf{p}\CC_{\flat}=\CC_{\flat}\textrm{ , }\mathbf{p}\CC=\CC\textrm{ , }\mathbf{p}\Pp=\Pp.$$

(2) \textit{The projected tensor of the metric tensor} $g$ \textit{is the angular metric tensor} $\eta$. Indeed, for any vector fields $X$, $Y$ on $M$,
\begin{center}
$\displaystyle\mathbf{p}g(\kal{X},\kal{Y}):=g(\mathbf{p}(\kal{X}),\mathbf{p}(\kal{Y}))=g(\kal{X}-\frac{1}{2E}(\vl{X}E)\delta,\kal{Y}-\frac{1}{2E}(\vl{Y}E)\delta)=g(\kal{X},\kal{Y})-\frac{1}{2E}(\vl{X}E)g(\delta,\kal{Y})-\frac{1}{2E}(\vl{Y}E)g(\kal{X},\delta)+\frac{1}{4E^2}(\vl{X}E)(\vl{Y}E)g(\delta,\delta)\overset{(\ref{can55}),(\ref{can56})}{=}g(\kal{X},\kal{Y})-\frac{1}{F^2}(\vl{X}E)\vl{\nabla}E(\kal{Y})-\frac{1}{F^2}(\vl{Y}E)\vl{\nabla}E(\kal{X})+\frac{1}{F^2}(\vl{X}E)(\vl{Y}E)=(g-\frac{1}{F^2}\vl{\nabla}E\otimes\vl{\nabla}E)(\kal{X},\kal{Y})=(g-\vl{\nabla}F\otimes\vl{\nabla}F)(\kal{X},\kal{Y})=\eta(\kal{X},\kal{Y})$.
\end{center}

(3) With the help of the projection operator $\mathbf{p}$, the Jacobi endomorphism of an isotropic Finsler manifold given by (\ref{22}) may be written in the extremely simple form $\KK=K\mathbf{p}=F^2R\mathbf{p}$.

\begin{lemma}
The projected tensor of the Berwald curvature of a Finsler manifold is
\begin{gather}
\mathbf{pB}=\mathbf{B}+\frac{1}{E}\mathbf{P}\otimes\delta.
\end{gather}
\end{lemma}

\begin{proof}
By (\ref{lem5kepl}) and Corollary \ref{5.2}, $\mathbf{pB}=\mathbf{p}\circ\mathbf{B}$. Now, for any vector fields $X$, $Y$, $Z$ on $M$,
\begin{center}
$\displaystyle(\mathbf{pB})(\kal{X},\kal{Y},\kal{Z})=\mathbf{p}(\mathbf{B}(\kal{X},\kal{Y})\kal{Z})\overset{(\ref{22})}{=}\mathbf{B}(\kal{X},\kal{Y})\kal{Z}-\frac{1}{2E}(\ii\mathbf{B}(\kal{X},\kal{Y})\kal{Z})E\delta\overset{(\ref{relc})}{=}\mathbf{B}(\kal{X},\kal{Y})\kal{Z}+\frac{1}{E}\mathbf{P}(\kal{X},\kal{Y},\kal{Z})\delta=(\mathbf{B}+\frac{1}{E}\mathbf{P}\otimes\delta)(\kal{X},\kal{Y},\kal{Z})$,
\end{center}
hence our statement.
\end{proof}

\begin{lemma}
The projected tensor of the Douglas curvature is
\begin{gather}\label{27}
\mathbf{pD}=\mathbf{pB}-\frac{1}{n+1}\textrm{tr}\mathbf{B}\odot\mathbf{p}=\mathbf{B}+\frac{1}{E}\mathbf{P}\otimes\delta-\frac{1}{n+1}\textrm{tr}\mathbf{B}\odot\mathbf{p}.
\end{gather}
\end{lemma}

\begin{proof}
First we check that $\mathbf{D}$ satisfies the condition of Corollary \ref{5.2}, i.e., $\mathbf{D}(\hul{X},\hul{Y})\hul{Z}=0$, if $\delta\in\left\{\hul{X},\hul{Y},\hul{Z}\right\}$. Let, for example, $\hul{X}:=\delta$. Then
\begin{center}
$\displaystyle\mathbf{D}(\delta,\hul{Y},\hul{Z}):=\mathbf{B}(\delta,\hul{Y},\hul{Z})-\frac{1}{n+1}(\textrm{tr}\mathbf{B}(\delta,\hul{Y})\hul{Z}+\textrm{tr}\mathbf{B}(\hul{Y},\hul{Z})\delta+\textrm{tr}\mathbf{B}(\hul{Z},\delta)\hul{Y})-\frac{1}{n+1}(\nabla_C\textrm{tr}\mathbf{B})(\hul{Y},\hul{Z})\delta\overset{(\ref{lem5kepl})}{=}-\frac{1}{n+1}(\textrm{tr}\mathbf{B}(\hul{Y},\hul{Z})\delta+\nabla_C\textrm{tr}\mathbf{B})(\hul{Y},\hul{Z})\delta)$.
\end{center}
By Lemma \ref{lemma35}, $\mathbf{B}$ is homogeneous of degree $-1$, i.e., $\nabla_C\mathbf{B}=-\mathbf{B}$. Thus $\nabla_C\textrm{tr}\mathbf{B}=\textrm{tr}\nabla_C\mathbf{B}=-\textrm{tr}\mathbf{B}$, and hence $\mathbf{D}(\delta,\hul{Y},\hul{Z})=0$. The other two cases may be handled similarly. Now it follows that
\begin{center}
$\displaystyle\mathbf{pD}=\mathbf{p}\circ\mathbf{D}\overset{(\ref{d26})}{=}\mathbf{pB}-\frac{1}{n+1}(\mathbf{p}(\textrm{tr}\mathbf{B}\odot\mathbf{1})+\mathbf{p}(\vl{\nabla}\textrm{tr}\mathbf{B}\otimes\delta))$.
\end{center}
We show that the last term at the right-hand side vanishes, and the middle term is just $\frac{1}{n+1}\textrm{tr}\BB\odot\mathbf{p}$. Indeed, for any vector fields $X$, $Y$, $Z$ on $M$ we have
\begin{center}
$\displaystyle\mathbf{p}(\textrm{tr}\mathbf{B}\odot\mathbf{1})(\kal{X},\kal{Y},\kal{Z}):=\mathbf{p}(\textrm{tr}\mathbf{B}\odot\mathbf{1}(\mathbf{p}\kal{X},\mathbf{p}\kal{Y},\mathbf{p}\kal{Z}))\overset{(\ref{lem5kepl}),\textrm{\textup{\ref{5.2}}}}{=}\mathbf{p}(\textrm{tr}\mathbf{B}(\kal{X},\kal{Y})\mathbf{p}(\kal{Z})+\textrm{tr}\mathbf{B}(\kal{Y},\kal{Z})\mathbf{p}(\kal{X})+\textrm{tr}\mathbf{B}(\kal{Z},\kal{X})\mathbf{p}(\kal{Y}))=\textrm{tr}\mathbf{B}(\kal{X},\kal{Y})\mathbf{p}(\kal{Z})+\textrm{tr}\mathbf{B}(\kal{Y},\kal{Z})\mathbf{p}(\kal{X})+\textrm{tr}\mathbf{B}(\kal{Z},\kal{X})\mathbf{p}(\kal{Y})=(\textrm{tr}\mathbf{B}\odot\mathbf{P})(\kal{X},\kal{Y},\kal{Z})$,
\end{center}
while
\begin{center}
$\displaystyle\mathbf{p}(\vl{\nabla}\textrm{tr}\mathbf{B}\otimes\delta)(\kal{X},\kal{Y},\kal{Z})=\mathbf{p}((\vl{\nabla}_{\mathbf{p}\kal{X}}\textrm{tr}\mathbf{B})(\mathbf{p}\kal{Y},\mathbf{p}\kal{Z})\delta)=0$,
\end{center}
since $\mathbf{p}(\delta)=0$.

This concludes the proof of (\ref{27}).
\end{proof}

Temporarily, for convenience, a Finsler manifold will be called a \textit{p-Berwald manifold} if its projected Berwald curvature vanishes, i.e., if it has the property
\begin{gather}\label{24}
\mathbf{B}+\frac{1}{E}\mathbf{P}\otimes\delta=0.
\end{gather}

First we characterize the R-quadratic p-Berwald manifolds.

\begin{propo}\label{propoprquad}
A p-Berwald manifold is R-quadratic, if and only if, its stretch tensor vanishes.
\end{propo}

\begin{proof}
The necessity of the condition is a consequence of Corollary $\ref{4.3}$. To prove the sufficiency, we show that in a p-Berwald manifold we have
\begin{gather}\label{25}
\vl{\nabla}\mathbf{H}(\hul{X},\hul{Y},\hul{Z},\hul{U})=\frac{1}{F^2}\mathbf{\Sigma}(\hul{Z},\hul{Y},\hul{X},\hul{U})\otimes\delta\textrm{ ; }\hul{X},\hul{Y},\hul{Z},\hul{U}\in\Gamma(\pik).
\end{gather}
Observe first that
\begin{center}
$\displaystyle\hl{\nabla}\mathbf{B}\overset{(\ref{24})}{=}-\hl{\nabla}(\frac{1}{E}\mathbf{P}\otimes\delta)\overset{(\ref{cons}),(\ref{14})}{=}-\frac{1}{E}\hl{\nabla}\mathbf{P}\otimes\delta$.
\end{center}
Now, applying Bianchi identity (\ref{bian3}), we get
\begin{center}
$\displaystyle\vl{\nabla}\mathbf{H}(\hul{X},\hul{Y},\hul{Z},\hul{U})=\hl{\nabla}\mathbf{B}(\hul{Y},\hul{Z},\hul{X},\hul{U})-\hl{\nabla}\mathbf{B}(\hul{Z},\hul{Y},\hul{X},\hul{U})=-\frac{1}{E}(\hl{\nabla}\mathbf{P}(\hul{Y},\hul{Z},\hul{X},\hul{U})-\hl{\nabla}\mathbf{P}(\hul{Z},\hul{Y},\hul{X},\hul{U}))\otimes\delta\overset{(\ref{20})}{=}\frac{1}{F^2}\mathbf{\Sigma}(\hul{Z},\hul{Y},\hul{X},\hul{U})$.
\end{center}
This proves (\ref{25}), whence the assertion follows.
\end{proof}

Our aim in the following is to show that \textit{p-Berwald manifolds do not constitute a new class of special Finsler manifolds}.

\begin{lemma}\label{5.4}
Any p-Berwald manifold is a weakly Berwald manifold.
\end{lemma}

\begin{proof}
From our condition (\ref{24}), $$\textrm{tr}\mathbf{B}=-\frac{1}{E}\textrm{tr}(\mathbf{P}\otimes\delta).$$ By (\ref{trace03}),
\begin{center}
$\displaystyle\textrm{tr}(\mathbf{P}\otimes\delta)=i_{\delta}\Pp$, where $\displaystyle i_{\delta}\Pp(\hul{X},\hul{Y}):=\Pp(\delta,\hul{X},\hul{Y})$.
\end{center}
To illustrate how our inductive definition (\ref{trace01}), (\ref{trace02}) works in such cases, we verify this.

Let $\hul{Y}_1$, $\hul{Y}_2$ be sections in $\szel{\pik}$. Then
\begin{center}
$\displaystyle i_{\hul{Y}_1}\textrm{tr}(\Pp\otimes\delta):=\textrm{tr}(j_{\hul{Y}_1}\Pp\otimes\delta)$,\\
$\displaystyle i_{\hul{Y}_2}i_{\hul{Y}_1}\textrm{tr}(\Pp\otimes\delta)=i_{\hul{Y}_2}\textrm{tr}(j_{\hul{Y}_1}\Pp\otimes\delta):=\textrm{tr}(j_{\hul{Y}_2}j_{\hul{Y}_1}\Pp\otimes\delta)=\Pp(\delta,\hul{Y}_1,\hul{Y}_1)$,
\end{center}
that is $$\textrm{tr}(\Pp\otimes\delta)(\hul{Y}_1,\hul{Y}_2)=i_{\delta}\Pp(\hul{Y}_1,\hul{Y}_2),$$ so our formula is indeed true.

Now our assertion follows by using Corollary \ref{coro63} (ii).
\end{proof}

\begin{propo}
If $(M,F)$ is an at least 3-dimensional Finsler manifold, then $(M,F)$ is a p-Berwald manifold, if and only if, it is a weakly Berwald Douglas manifold.
\end{propo}

\begin{proof}
If $(M,F)$ is a p-Berwald manifold, then it is weakly Berwald by Proposition \ref{5.4}, therefore (\ref{27}) reduces to $\mathbf{pD}=0$. However, by a theorem of T. Sakaguchi \cite{Sakaguchi} (see also \cite{Vattamany}), relation $\mathbf{pD}=0$ is equivalent to the vanishing of the Douglas curvature under the condition $\textrm{dim}M>2$.

Conversely, if $(M,F)$ is a weakly Berwald Douglas manifold, then $\mathbf{D}=\mathbf{pD}=0$ and $\textrm{tr}\mathbf{B}=0$ imply by (\ref{27}) that $(M,F)$ is a p-Berwald manifold.
\end{proof}

\begin{coro}
The class of the at least 3-dimensional p-Berwald manifolds coincides with the class of the at least 3-dimensional Berwald manifolds.
\end{coro}

\begin{proof}
By our remark at the end of Chapter \ref{ch5}, a Finsler manifold is a weakly Berwald Douglas manifold, if and only if, it is a Berwald manifold.
\end{proof}

\textbf{Remark.} In the light of this result, Proposition \ref{propoprquad} says nothing new, if the base manifold is at least 3-dimensional. Indeed, Berwald manifolds have vanishing Landsberg tensor, and hence vanishing stretch tensor. On the other side, the canonical connection of a Berwald manifold is basic, which implies that the affine curvature is a vertical lift, whose vertical differential vanishes automatically. The exceptional 2-dimensional case will be treated in the next chapter.

\medskip
We conclude this chapter with a technically new proof of the Finslerian version of classical Schur's lemma on isotropy. We begin with a useful preparatory observation.

\begin{lemma}
$(M,F)$ is an isotropic Finsler manifold, if and only if, the curvature tensor of its canonical connection is of the form
\begin{gather}\label{schurlemma}
\RR=F\mathbf{p}\wedge(R\vl{\nabla}F+\frac{1}{3}F\vl{\nabla}R),
\end{gather}
where $R$ is the scalar curvature of $(M,F)$.
\end{lemma}

\begin{proof}
Assume that $(M,F)$ is isotropic. Then, by (\ref{berwszabp}), the Jacobi endomorphism of $(M,F)$ has the form
$$\KK=K(\mathbf{1}-\frac{1}{F}\vl{\nabla}F\otimes\delta)\textrm{ , }K:=\frac{1}{n-1}\textrm{tr}\KK.$$ In view of Proposition \ref{propo46}, $$3\RR(\kal{X},\kal{Y})=\vl{\nabla}\KK(\kal{Y},\kal{X})-\vl{\nabla}\KK(\kal{X},\kal{Y})\textrm{ ; }X,Y\in\modx{M}.$$ In the proof of \ref{berwszab}, we have already determined the term $\vl{\nabla}\KK(\kal{X},\kal{Y})$. Using this result, we obtain that
\begin{center}
$\displaystyle 3\RR(\kal{X},\kal{Y})=\vl{\nabla}K(\kal{Y})\kal{X}-\vl{\nabla}K(\kal{X})\kal{Y}-\frac{1}{F}((\vl{Y}K)(\vl{X}F)-(\vl{X}K)(\vl{Y}F))\delta-\frac{K}{F}((\vl{X}F)\kal{Y}-(\vl{Y}F)\kal{X})=(\vl{\nabla}K\wedge\mathbf{1})(\kal{Y},\kal{X})-\frac{1}{F}(\vl{\nabla}K\wedge\vl{\nabla}F)\otimes\delta(\kal{Y},\kal{X})-\frac{K}{F}(\vl{\nabla}F\wedge\mathbf{1})(\kal{X},\kal{Y})=(\mathbf{1}\wedge(\vl{\nabla}K+\frac{K}{F}\vl{\nabla}F)+\frac{1}{F}(\vl{\nabla}F\wedge\vl{\nabla}K)\otimes\delta)(\kal{X},\kal{Y})$.
\end{center}
We replace the identity operator $\mathbf{1}$ by $\mathbf{p}+\frac{1}{F}\vl{\nabla}F\otimes\delta$. Then
\begin{center}
$\displaystyle 3\RR=\mathbf{p}\wedge(\vl{\nabla}K+\frac{K}{F}\vl{\nabla}F)+\frac{1}{F}(\vl{\nabla}F\otimes\delta)\wedge\vl{\nabla}K+\frac{K}{F^2}(\vl{\nabla}F\otimes\delta)\wedge\vl{\nabla}F+\frac{1}{F}(\vl{\nabla}F\wedge\vl{\nabla}K)\otimes\delta=F\mathbf{p}\wedge\left(\frac{1}{F}\vl{\nabla}K+R\vl{\nabla}F\right)$.
\end{center}
Since
\begin{center}
$\displaystyle\vl{\nabla}K(\kal{X})=\vl{\nabla}(RF^2)(\kal{X})=\vl{X}(RF^2)=(\vl{X}R)F^2+2RF\vl{X}F=(F^2\vl{\nabla}R+2RF\vl{\nabla}F)(\kal{X})$,
\end{center}
it follows that in the isotropic case
$$\RR=F\mathbf{p}\wedge(R\vl{\nabla}F+\frac{1}{3}F\vl{\nabla}R).$$

Conversely, suppose that the curvature of the canonical connection of $(M,F)$ can be written in this form. Then
\begin{center}
$\displaystyle\KK(\kal{X})\overset{\textrm{\textup{(\ref{jacobi22})}}}{=}\RR(\kal{X},\delta)=F(RF+\frac{1}{3}F(CR))\mathbf{p}(\kal{X})-F(R\vl{\nabla}F+\frac{1}{3}F\vl{\nabla}R)(\kal{X})\mathbf{p}(\delta)=RF^2\mathbf{p}(\kal{X})$,
\end{center}
so, by Example (3) above, $(M,F)$ is isotropic.
\end{proof}

\begin{propo}\textup{(generalized Schur lemma on isotropy).}
If $(M,F)$ is an at least 3-dimensional isotropic Finsler manifold whose scalar curvature depends only on the position, then $(M,F)$ is of constant curvature.
\end{propo}

\begin{proof}
Our condition on the scalar curvature means that $\vl{\nabla}R=0$, hence the form (\ref{schurlemma}) of the curvature of the canonical connection reduces to
\begin{gather}\label{scurrstar}
\RR=FR\mathbf{p}\wedge\vl{\nabla}F=R(\mathbf{p}\otimes\vl{\nabla}E-\vl{\nabla}E\otimes\mathbf{p}).
\end{gather}
First we calculate the h-Berwald differential of $\RR$.

Let $X$, $Y$, $Z$ be vector fields on $M$. Note that
\begin{center}
$\displaystyle(\nabla_{\hl{X}}\vl{\nabla}E)(\kal{Y})=\hl{X}(\vl{Y}E)-\vl{\nabla}E(\nabla_{\hl{X}}\kal{Y})=\hl{X}(\vl{Y}E)-\vl{\nabla}E(\VV[\hl{X},\vl{Y}])=\hl{X}(\vl{Y}E)-[\hl{X},\vl{Y}]E=\vl{Y}(\hl{X}E)=0$.
\end{center}
By this observation and Lemma \ref{ortprojvantrac} we find
\begin{center}
$\displaystyle(\nah(R\projp\otimes\nav E))(\kal{X},\kal{Y},\kal{Z})=(\nabla_{\hl{X}}(R\projp\otimes\nav E))(\kal{Y},\kal{Z})=(\hl{X}R)\projp\otimes\nav E(\kal{Y},\kal{Z})+R(\nabla_{\hl{X}}\projp\otimes\nav E)(\kal{Y},\kal{Z})=\nah R\otimes\projp\otimes\nav E(\kal{X},\kal{Y},\kal{Z})+R((\nabla_X\projp)\otimes\nav E+\projp\nabla_{\hl{X}}\nav E)(\kal{Y},\kal{Z})=\nah R\otimes\projp\otimes\nav E(\kal{X},\kal{Y},\kal{Z})$.
\end{center}
Similarly,
\begin{center}
$\displaystyle\nah(R\nav E\otimes\projp)=\nah R\otimes\nav E\otimes\projp$,
\end{center}
therefore
\begin{center}
$\displaystyle\nah\RR=\nah R\otimes(\projp\otimes\nav E-\nav E\otimes\projp)$.
\end{center}
Let $$(\mathfrak{S}\hl{\nabla}\RR)(\kal{X},\kal{Y},\kal{Z}):=\underset{(X,Y,Z)}{\mathfrak{S}}(\hl{\nabla}\RR)(\kal{X},\kal{Y},\kal{Z}).$$ By the general Bianchi identity (\ref{genbin}), $\mathfrak{S}\hl{\nabla}\RR=0$. In our case this leads to the equality
\begin{center}
$\displaystyle\nah R\otimes\projp\otimes\nav E+\projp\otimes\nav E\otimes\nah R+\nav E\otimes\nah R\otimes\projp-\nah R\times\nav E\otimes\projp-\nav E\otimes\projp\otimes\nah R-\projp\otimes\nah R\otimes\nav E=0$.
\end{center}
Applying the inductive definition (\ref{trace01}), (\ref{trace02}), we calculate, term by term, the trace of the left-hand side.

(i)
\begin{center}
$\displaystyle i_{\kal{X}}\textrm{tr}(\nah R\otimes\projp\otimes\nav E)=\textrm{tr}(\nah R\otimes\nav E\otimes\projp(\kal{X}))$,\\
$\displaystyle i_{\kal{Y}}\textrm{tr}(\nah R\otimes\nav E\otimes\projp(\kal{X}))=\nav E(\kal{Y})\nah R(\projp(\kal{X}))=\nav E(\kal{Y})\HH\projp(\kal{X})(R)=\nav E(\kal{Y})(\hl{X}R-\frac{\nav F(\kal{X})}{F}SR)=(\nah R\otimes\nav E-(SR)\nav F\otimes\nav F)(\kal{X},\kal{Y})$,
\end{center}
thus $$\textrm{tr}(\nah R\otimes\projp\otimes\nav E)=\nah R\otimes\nav E-(SR)\nav F\otimes\nav F.$$

(ii)
\begin{center}
$\displaystyle i_{\kal{X}}\textrm{tr}(\projp\otimes\nav E\otimes\nah R)=\nav E(\kal{X})\textrm{tr}(\projp\otimes\nah R)$,\\
$\displaystyle i_{\kal{Y}}\textrm{tr}(\projp\otimes\nah R)=\nah R(\kal{Y})\textrm{tr}\projp\overset{\textrm{\textup{\ref{ortprojvantrac}}}}{=}(n-1)\nah R(\kal{Y})$,
\end{center}
hence $$\textrm{tr}(\projp\otimes\nav R\otimes\nah R)=(n-1)\nav E\otimes\nah R.$$

(iii)
\begin{center}
$\displaystyle i_{\kal{X}}\textrm{tr}(\nav E\otimes\nah R\otimes\projp)=\nah R(\kal{X})\textrm{tr}(\nav E\otimes\projp)$,\\
$\displaystyle i_{\kal{Y}}\textrm{tr}(\nav E\otimes\projp)=\nav E(\projp(\kal{Y}))=\ii\projp(\kal{Y})E=$\\
$\displaystyle \vl{Y}E-\frac{\nav F(\kal{Y})}{F}\cdot 2E=F(\vl{Y}E)-F(\vl{Y}E)=0$,
\end{center}
thus $$\textrm{tr}(\nav E\otimes\nah R\otimes\projp)=0.$$

In the remainder three cases the calculation is similar. We obtain:

(iv)
\begin{center}
$\displaystyle\textrm{tr}(\nah R\otimes\nav E\otimes\projp)=\nav E\otimes\textrm{tr}(\nah R\otimes\projp)=$\\
$\displaystyle \nav E\otimes\nah R-(SR)\nav F\otimes\nav E$,
\end{center}

(v)
\begin{center}
$\displaystyle\textrm{tr}(\nav E\otimes\projp\otimes\nah R)=\nah R\otimes\textrm{tr}(\nav E\otimes\projp)=0$,
\end{center}

(vi)
\begin{center}
$\displaystyle\textrm{tr}(\projp\otimes\nah R\otimes\nav E)=\nah R\textrm{tr}(\projp\otimes\nav E)=(n-1)\nah R\otimes\nav E$.
\end{center}

Subtracting from the sum of the expressions obtained in (i)-(iii) the sum of the next three ones, we find
\begin{center}
$\displaystyle0=\nah R\otimes\nav E+(n-1)\nav E\otimes\nah R-\nav E\otimes\nah R-(n-1)\nah R\otimes\nav E=(n-2)(\nav E\otimes\nah R)+(2-n)(\nah R\otimes\nav E)=(2-n)(\nah R\otimes\nav E-\nav E\otimes\nah R)=(2-n)\nah R\wedge\nav E$.
\end{center}
By our assumption $n\geq 3$ this implies $$\nah R\wedge\nav E=0.$$ Evaluating the left-hand side on a pair $(\delta,\kal{X})$ , $X\in\modx{M}$, it follows that $$0=SR\nav E(\kal{X})-\nah R(\kal{X})\cdot2E,$$ whence $$\nah R=\frac{SE}{2E}\nav E=:f\nav E.$$

Now, by the condition $\nav R=0$, for any vector fields $X$, $Y$ on $M$ we have
\begin{center}
$\displaystyle 0=\nah\nav R(\kal{Y},\kal{X})\overset{\textrm{\textup{(\ref{ric1})}}}{=}\nav\nah R(\kal{X},\kal{Y})=\nav(f\nav E)(\kal{X},\kal{Y})=(\nabla_{\vl{X}}f\nav E)(\kal{Y})=((\vl{X}f)\nav E+f\nabla_{\vl{X}}\nav E)(\kal{Y})=(\nav f\otimes\nav E)(\kal{X},\kal{Y})+f(\nav\nav E)(\kal{X},\kal{Y})$,
\end{center}
i.e., $$f\nav\nav E=\nav f\otimes\nav E.$$

Since, by (F$_{\textrm{3}}$), $\nav\nav E$ is fibrewise non-degenerate, while the rank of\\ $\nav f\otimes\nav E$ is at most 1 at any point, the last relation implies that the function $f$ is identically zero, and hence $\nah R=0$. Then, for any vector field $\xi$ on $\Tk M$,
\begin{center}
$\displaystyle dR(\xi)=\xi R=(\hh\xi)R+(\vv\xi)R=\nah R(\xi)+\nav R(\xi)=0$;
\end{center}
therefore $dR=0$. Since $M$ is connected by our assumption at the very beginning, relation $dR=0$ implies that the function $R$ is constant.
\end{proof}

\textbf{Remark.} Naturally, the generalized Schur lemma on isotropic Finsler manifolds is a classical result. It is also due to Berwald, see \cite{Berwald}. A concise, but very elegant index and argumentum free proof can be found in del Castillo's note \cite{delc}. Our proof is technically new; its novelty lies in a consequent and index-free use of the Berwald derivative and the inductively defined trace operator.

\newpage

\chapter{Two-dimensional Finsler manifolds}\label{ch9}

Throughout this chapter $(M,F)$ will be a \textit{two-dimensional, positive definite Finsler manifold}. By a \textit{Berwald frame} of $(M,F)$ we mean a pair $(\ell,m)$, where $\ell$ is the normalized support element field defined by (\ref{nsef}), and $m$ is a local section of $\pik$ satisfying
\begin{gather}
g(\ell,m)=0\textrm{ , }g(m,m)=1.
\end{gather}
By the usual orthogonalization process, such a section $m$ may always be constructed. Indeed, we may find a nowhere vanishing local section $\hul{X}$ of $\pik$ such that $\hul{X}$ and $\ell$ are pointwise linearly independent. Then $$\hul{Y}:=\hul{X}-g(\ell,\hul{X})\ell$$ is nowhere zero at the points of the domain of $\hul{X}$, and it is $g$-orthogonal to $\ell$: $$g(\ell,\hul{Y})=g(\ell,\hul{X})-g(\ell,\hul{X})g(\ell,\ell)\overset{\textrm{\textup{(\ref{elelegy})}}}{=}0.$$ Let $$m:=\frac{1}{\sqrt{g(\hul{Y},\hul{Y})}}\hul{Y}.$$ Then $(\ell,m)$ is a Berwald frame for $(M,F)$.

If $(M,F)$ is a non-Riemannian Finsler manifold, we may use another argument. In this case, by Deicke's theorem, there is an open subset $\UU$ of $M$ such that the Cartan vector field $\overset{\ast}{\CC}$ is nowhere zero on $\tauk^{-1}(\UU)$. Let $$m:=\frac{1}{\sqrt{g(\overset{\ast}{\CC},\overset{\ast}{\CC})}}\overset{\ast}{\CC}$$ over $\tauk^{-1}(\UU)$. Then, by Lemma \ref{astdelta}, $(\ell,m)$ is a Berwald frame for $(M,F)$.

\begin{lemma}
If $(\ell,m)$ is a Berwald frame for the Finsler manifold $(M,F)$, then
\begin{gather}
\nabla_C\ell=0,\label{2fins1a}\\
\nabla_Cm=0;\label{2fins1b}
\end{gather}
i.e., $\ell$ and $m$ are positive-homogeneous of degree 0.
\end{lemma}

\begin{proof}
The first relation can be obtained by a straightforward calculation:
\begin{center}
$\displaystyle\nabla_C\ell=\nabla_C\frac{1}{F}\delta=C\left(\frac{1}{F}\right)\delta+\frac{1}{F}\nabla_C\delta=-\frac{1}{F^2}(CF)\delta+\frac{1}{F}\jj[C,S]=-\frac{1}{F^2}F\delta+\frac{1}{F}\delta=0$.
\end{center}

To prove the second, we consider the Fourier expansion of $\nabla_Cm$ with respect to the orthonormal frame $(\ell,m)$: $$\nabla_Cm=g(\nabla_Cm,\ell)\ell+g(\nabla_Cm,m)m.$$ We show that both Fourier coefficients vanish. Since $g(m,\ell)=0$ and $\nabla_Cg=0$,
\begin{center}
$\displaystyle 0=Cg(m,\ell)=(\nabla_Cg)(m,\ell)+g(\nabla_Cm,\ell)+g(m,\nabla_C\ell)\overset{\textrm{\textup{(\ref{2fins1a})}}}{=}g(\nabla_Cm,\ell)$.
\end{center}
Similarly,
\begin{center}
$\displaystyle 0=Cg(m,m)=(\nabla_Cg)(m,m)+2g(\nabla_Cm,m)=2g(\nabla_Cm,m)$;
\end{center}
hence $\nabla_Cm=0$.
\end{proof}

\begin{lemma}
If $(\ell,m)$ is a Berwald frame for $(M,F)$, then the only non-zero component of the curvature of the canonical connection with respect to $(\ell,m)$ is
\begin{gather}\label{2fins2}
\RR(m,\ell)=g(\RR(m,\ell),m)m.
\end{gather}
\end{lemma}

\begin{proof}
We apply Fourier expansion again, with respect to the Berwald frame $(\ell,m)$. Then $\RR(m,\ell)$ can be represented in the form $$\RR(m,\ell)=g(\RR(m,\ell),m)m+g(\RR(m,\ell),\ell)\ell.$$ We are going to show that the second Fourier coefficient vanishes. Obviously, it is enough to check that $g(\RR(m,\ell),\delta)=0$. Using the Hilbert 1-form $\theta$ and taking into account (\ref{can55}), we obtain 
\begin{center}
$\displaystyle g(\RR(m,\ell),\delta)=\theta(\RR(m,\ell))=F\vl{\nabla}F(\RR(m,\ell))=F\ii\RR(m,\ell)F=-F(\vv[\HH m,\HH\ell]F)=-F([\HH m,\HH\ell]F)+F(\hh[\HH m,\HH\ell]F)=0$,
\end{center}
since the horizontal vector fields kill the Finsler function $F$.
\end{proof}

\medskip
This lemma brings the function
\begin{gather}\label{2fins3}
\kappa:=g(\RR(m,\ell),m)
\end{gather}
into the spotlight; it is said to be the \textit{Gauss curvature} of $(M,F)$.

\begin{lemma}\label{2dim94}
Hypothesis as above. The Cartan tensor and the vector-valued Cartan tensor of $(M,F)$ can be represented in the form
\begin{gather}\label{2fins4a}
\CC_{\flat}=I\vl{\nabla}_m\theta\otimes\vl{\nabla}_m\theta\otimes\vl{\nabla}_m\theta
\end{gather}
and
\begin{gather}\label{2fins4b}
\CC=I\vl{\nabla}_m\theta\otimes\vl{\nabla}_m\theta\otimes m,
\end{gather}
respectively, where $$I:=g(\CC(m,m),m)=\CC_{\flat}(m,m,m).$$
\end{lemma}

\begin{proof}
Since $\CC$ and $\CC_{\flat}$ vanish if one of their arguments is $\ell=\frac{1}{F}\delta$, for any vector fields $$\hul{X}=X^1\ell+X^2m\textrm{  ,  }\hul{Y}=Y^1\ell+Y^2m$$ along $\tauk$ we obtain $$\CC(\hul{X},\hul{Y})=X^2Y^2\CC(m,m)=g(\hul{X},m)g(\hul{Y},m)\CC(m,m).$$ Here
\begin{center}
$\displaystyle\CC(m,m)=g(\CC(m,m),\ell)\ell+g(\CC(m,m),m)m=\CC_{\flat}(m,m,\ell)\ell+g(\CC(m,m),m)m=g(\CC(m,m),m)m$;
\end{center}
while, for example,
\begin{center}
$\displaystyle g(\hul{X},m)=g(m,\hul{X})=\frac{1}{2}\vl{\nabla}\vl{\nabla} F^2(m,\hul{X})=\frac{1}{2}(\vl{\nabla}_m\vl{\nabla}F^2)(\hul{X})=(\vl{\nabla}_m(F\vl{\nabla}F))(\hul{X})=\vl{\nabla}_m\theta(\hul{X})$.
\end{center}
Let $I:=g(\CC(m,m),m)=\CC_{\flat}(m,m,m)$. Then $$\CC(\hul{X},\hul{Y})=I\vl{\nabla}_m\theta\otimes\vl{\nabla}_m\theta\otimes m(\hul{X},\hul{Y}),$$ which proves (\ref{2fins4b}). Thus
\begin{center}
$\displaystyle\CC_{\flat}(\hul{X},\hul{Y},\hul{Z})=g(\CC(\hul{X},\hul{Y}),\hul{Z})=I\vl{\nabla}_m\theta\otimes\vl{\nabla}_m\theta(\hul{X},\hul{Y})g(m,\hul{Z})=I\vl{\nabla}_m\theta\otimes\vl{\nabla}_m\theta\otimes\vl{\nabla}_m\theta(\hul{X},\hul{Y}\hul{Z})$,
\end{center}
so (\ref{2fins4a}) is also true.
\end{proof}

\medskip
The function $I$ introduced by the Lemma is said to be the \textit{main scalar} of the 2-dimensional Finsler manifold $(M,F)$. We obtain immediately the next

\begin{coro}
A two-dimensional Finsler manifold reduces to a Riemannian manifold, if and only if, its main scalar vanishes.
\end{coro}
\mbox{}\hfill\tiny$\square$\normalsize \bigskip

\begin{propo}\label{2finspropo1}
Let $S$,$\HH$, $\kappa$ and $I$ be the canonical spray, the canonical connection, the Gauss curvature and the main scalar of $(M,F)$, respectively. If $(\ell,m)$ is a Berwald frame for $(M,F)$, then we have the following commutator formulae:
\begin{gather}
[S,\ii m]=-\HH m,\label{2fins5a}\\
[\HH\ell,\HH m]=\kappa(\ii m),\label{2fins5b}\\
[\HH m,\ii m]=\frac{1}{F}\HH\ell+I(\HH m)+(S I)\ii m.\label{2fins5c}
\end{gather}
\end{propo}

\begin{proof}
\textbf{Step 1} Consider the horizontal projector $\hh:=\HH\circ\jj$ and the vertical projector $\vv=\mathbf{1}-\hh$. Then $$[S,\ii m]=\hh[S,\ii m]+\vv[S,\ii m].$$ First we show that the second term is zero. Observe that $$\vv[S,\ii m]=\ii\VV[S,\ii m]=\ii\nabla_{S}m,$$ since $S$ is horizontal. Now, starting out the relation $g(\ell,m)=0$, and taking into account that $\nabla_{S}g=0$ by Corollary \ref{coro73}, we get:
\begin{center}
$\displaystyle 0=S g(\ell,m)=(\nabla_{S}g)(\ell,m)+g(\nabla_{S}\ell,m)+g(\ell,\nabla_{S}m)=g(\nabla_{S}\ell,m)+g(\ell,\nabla_{S}m)$.
\end{center}
Here
\begin{center}
$\displaystyle\nabla_{S}\ell=\nabla_{S}\frac{1}{F}\delta=S(\frac{1}{F})\delta+\frac{1}{F}\nabla_{S}\delta=\frac{1}{F}\VV[S,C]=-\frac{1}{F}\VV(S)=-\frac{1}{F}\VV\circ\HH(\delta)=0$,
\end{center}
therefore $$g(\ell,\nabla_{S}m)=0.$$ On the other side, from the relation $g(m,m)=1$,
\begin{center}
$\displaystyle 0=S g(m,m)=(\nabla_{S}g)(m,m)+2g(m,\nabla_{S}m)=2g(m,\nabla_{S}m)$.
\end{center}
Thus we have $g(\ell,\nabla_{S}m)=g(m,\nabla_{S}m)=0$, hence $\nabla_{S}m=0$, which implies our first claim. Now we can easily finish the proof of (\ref{2fins5a}):
\begin{center}
$\displaystyle [S,\ii m]=\hh[S,\ii m]=-\FF\circ\mathbf{J}[\ii m,S]=-\FF\circ\ii m=-\HH\circ\VV\circ\ii m+\mathbf{J}\circ\ii m=-\HH m$.
\end{center}

\textbf{Step 2}
\begin{center}
$\displaystyle [\HH\ell,\HH m]=\vv[\HH\ell,\HH m]+\hh[\HH\ell,\HH m]=\ii\circ\VV[\HH\ell,\HH m]+\HH\circ\jj[\HH\ell,\HH m]=-\ii\RR(\ell,m)+\HH\circ\jj[\HH\ell,\HH m]$.
\end{center}
We show that $\jj[\HH\ell,\HH m]=0$. 

By the vanishing of the torsion of $\HH$ we have
$$\jj[\HH\ell,\HH m]=\nabla_{\HH\ell}m-\nabla_{\HH m}\ell.$$ In the right-hand side $\nabla_{\HH\ell}m=F\nabla_{S}m=0$, since, as we have just seen, $\nabla_{S}m=0$. As to the second term, by the homogeneity of the canonical connection $$\mathbf{t}=\hl{\nabla}\delta\overset{\textrm{\textup{(\ref{14})}}}{=}0,$$ hence
\begin{center}
$\displaystyle\nabla_{\HH m}\ell=\nabla_{\HH m}F\delta=\HH m(F)\delta+F\nabla_{\HH m}\delta=F\nabla_{\HH m}\delta=F\tt(m)=0$.
\end{center}
Thus we obtain
\begin{center}
$\displaystyle [\HH\ell,\HH m]=-\ii\RR(\ell,m)=\ii\RR(m,\ell)\overset{\textrm{\textup{(\ref{2fins2})}}}{=}\ii g(\RR(m,\ell),m)m\overset{\textrm{\textup{(\ref{2fins3})}}}{=}\kappa(\ii m)$,
\end{center}
so formula (\ref{2fins5b}) is proved.

\textbf{Step 3}
We start again with the decomposition
\begin{center}
$\displaystyle [\HH m,\ii m]=\ii\circ\VV[\HH m,\ii m]+\HH\circ\jj[\HH m,\ii m]=\ii\nabla_{\HH m}m-\HH\nabla_{\ii m}m$.
\end{center}
Next we calculate the covariant derivatives $\nabla_{\HH m}m$ and $\nabla_{\ii m}m$. By Fourier expansion, $$\nabla_{\HH m}m=g(\nabla_{\HH m}m,\ell)\ell+g(\nabla_{\HH m}m,m)m.$$ We show that the Fourier coefficient in the first term vanishes. Applying the usual trick, we obtain
\begin{center}
$\displaystyle 0=\HH mg(m,\ell)=(\nabla_{\HH m}g)(m,\ell)+g(\nabla_{\HH m}m,\ell)+g(m,\nabla_{\HH m}\ell)=(\nabla_{\HH m}g)(m,\ell)+g(\nabla_{\HH m}m,\ell)$,
\end{center}
since, as we have just seen, $\nabla_{\HH m}\ell=0$. Applying Corollary \ref{coro63}, the first term of the right side of the above equality is
\begin{center}
$\displaystyle(\nabla_{\HH m}g)(m,\ell)=\hl{\nabla}g(m,\ell,m)=-2\Pp(m,\ell,m)=-2F\Pp(m,\delta,m)=0$.
\end{center}
Thus we get
\begin{center}
$\displaystyle \nabla_{\HH m}m=g(\nabla_{\HH m}m,m)m=\frac{1}{2}(2g(\nabla_{\HH m}m,m)m)=\frac{1}{2}(\HH mg(m,m)-(\nabla_{\HH m}g)(m,m))=-\frac{1}{2}\hl{\nabla}g(m,m,m)m=\Pp(m,m,m)m$.
\end{center}
By Proposition \ref{propo65}, the Cartan tensor and the Landsberg tensor are related by $\nabla_S\CC_{\flat}=\Pp$. Using this,
\begin{center}
$\displaystyle\Pp(m,m,m)=(\nabla_S\CC_{\flat})(m,m,m)=S(\CC_{\flat}(m,m,m))+3\CC_{\flat}(\nabla_Sm,m,m)\overset{\textrm{\textup{Step 1}}}{=}S(\CC_{\flat}(m,m,m))\overset{\textrm{\textup{\ref{2dim94}}}}{=}SI$,
\end{center}
therefore $$\nabla_{\HH m}m=\Pp(m,m,m)m=(SI)m.$$

Finally, we calculate the covariant derivative $\nabla_{\ii m}m$. As above, we consider the Fourier expansion
\begin{center}
$\displaystyle\nabla_{\ii m}m=g(\nabla_{\ii m}m,\ell)\ell+g(\nabla_{\ii m}m,m)m$.
\end{center}
Since
\begin{center}
$\displaystyle 0=\ii mg(m,\ell)=\nabla_{\ii m}g(m,\ell)+g(\nabla_{\ii m}m,\ell)+g(m,\nabla_{\ii m}\ell)=\vl{\nabla}g(m,\ell,m)+g(\nabla_{\ii m}m,\ell)+g(m,\nabla_{\ii m}\ell)=2F\CC_{\flat}(m,\delta,m)+g(\nabla_{\ii m}m,\ell)+g(m,\nabla_{\ii m}\ell)=g(\nabla_{\ii m}m,\ell)+g(m,\nabla_{\ii m}\ell)$,
\end{center}
it follows that
\begin{center}
$\displaystyle g(\nabla_{\ii m}m,\ell)=-g(\nabla_{\ii m}\ell,m)=-g(\jj[\ii m,\HH\ell],m)=-g(\jj[\ii m,\frac{1}{F}S],m)=-\frac{1}{F}g(\jj[\ii m,S],m)-g(\jj(\ii m\frac{1}{F})S,m)=-\frac{1}{F}g(m,m)=-\frac{1}{F}$.
\end{center}
As to the second term of the Fourier expansion, $$0=\ii mg(m,m)=(\nabla_{\ii m}g)(m,m)+2g(\nabla_{\ii m}m,m),$$ hence $$g(\nabla_{\ii m}m,m)=-\frac{1}{2}\vl{\nabla}g(m,m,m)=-\CC_{\flat}(m,m,m)=:-I.$$ Thus $$\nabla_{\ii m}m=-\frac{1}{F}\ell-Im,$$ therefore
\begin{center}
$\displaystyle[\HH m,\ii m]=\ii(SI)m+\frac{1}{F}\HH\ell+I(\HH m)=\frac{1}{F}\HH\ell+I(\HH m)+(S I)\ii m$,
\end{center}
as was to be shown.
\end{proof}

\textbf{Remark.} Relation (\ref{2fins5a}) can also be written in the form
\begin{gather}\label{2fins5acs}
[\HH\ell,\ii m]=-\frac{1}{F}\HH m.
\end{gather}
Indeed,
\begin{center}
$\displaystyle[S,\ii m]=[\HH\delta,\ii m]=[F(\HH\ell),\ii m]=F[\HH\ell,\ii m]-\ii m(F)\HH\ell$.
\end{center}
Here
\begin{center}
$\displaystyle\ii m(F)=\vl{\nabla}F(m)=\frac{1}{F}F\vl{\nabla}F(m)=\frac{1}{F}\theta(m)=\frac{1}{F}g(m,\delta)=g(m,\ell)=0$,
\end{center}
therefore $$[\HH\ell,\ii m]=\frac{1}{F}[S,\ii m]\overset{\textrm{\textup{(\ref{2fins5a})}}}{=}-\frac{1}{F}\HH m.$$

\begin{coro}
We have the following formulae for the covariant derivatives of the members of a Berwald frame $(\ell,m)$:
\begin{gather}
\nabla_{\ii\ell}\ell=0,\label{2fins6a}\\
\nabla_{\HH\ell}\ell=0, \label{2fins6b}\\
\nabla_{\ii m}\ell=\frac{1}{F}m,\label{2fins7a}\\
\nabla_{\HH m}\ell=0,\label{2fins7b}\\
\nabla_{\ii\ell}m=0,\label{2fins8a}\\ 
\nabla_{\HH\ell}m=0, \label{2fins8b}\\
\nabla_{\ii m}m=-\frac{1}{F}\ell-Im,\label{2fins9}\\ 
\nabla_{\HH m}m=(S I)m. \label{2fins10}
\end{gather}
\end{coro}

\begin{proof}
Indeed, $\ii\ell=\ii(\frac{1}{F}\delta)=\frac{1}{F}C$, so (\ref{2fins6a}) and (\ref{2fins8a}) are consequences of (\ref{2fins1a}) and (\ref{2fins1b}). Since $\HH\ell=\HH(\frac{1}{F}\delta)=\frac{1}{F}S$, (\ref{2fins6b}) and (\ref{2fins8b}) follow from the relations $\nabla_{S}\ell=0$ and $\nabla_{S}m=0$, obtained in the proof of Proposition \ref{2finspropo1}. (\ref{2fins7b}) and (\ref{2fins10}) were also proved above. Finally,
\begin{center}
$\displaystyle\nabla_{\ii m}\ell=\jj[\ii m,\HH\ell]=-\jj[\HH\ell,\ii m]\overset{\textrm{\textup{(\ref{2fins5acs})}}}{=}\jj(\frac{1}{F}\HH m)=\frac{1}{F}m$
\end{center}
and
\begin{center}
$\displaystyle\nabla_{\ii m}m=\jj[\ii m,\HH m]\overset{\textrm{\textup{(\ref{2fins5c})}}}{=}-\jj(\frac{1}{F^2}S+I(\HH m)+(S I)\ii m)=$\\
$\displaystyle -\frac{1}{F^2}\delta-Im=-\frac{1}{F}\ell-Im$,
\end{center}
so (\ref{2fins7a}) and (\ref{2fins9}) are also true.
\end{proof}

\textbf{Remark.} There is another method to derive the basic relations (\ref{2fins5a})-(\ref{2fins10}). First we calculate the covariant derivatives (\ref{2fins6a})-(\ref{2fins10}) using different techniques, and next we apply the Ricci formulae (\ref{hhricci21}), (\ref{ric1}) to obtain the commutator formulae (\ref{2fins5a})-(\ref{2fins5c}).

\textit{Example.} In view of (\ref{hhricci21}), for any function $f\in\modc{\Tk M}$ we have $$(\hl{\nabla}\hl{\nabla}f)(\ell,m)-(\hl{\nabla}\hl{\nabla}f)(m,\ell)=-\ii\RR(\ell,m).$$ Here
$$(\hl{\nabla}\hl{\nabla}f)(\ell,m)=\HH\ell(\HH m(f))-\hl{\nabla}f(\nabla_{\HH\ell}m),$$
$$(\hl{\nabla}\hl{\nabla}f)(m,\ell)=\HH m(\HH\ell(f))-\hl{\nabla}f(\nabla_{\HH m}\ell).$$
Now, as in Step 1 and Step 2 in the proof of \ref{2finspropo1}, we can show that $$\nabla_{\HH\ell}m=\nabla_{\HH m}\ell=0.$$ So we obtain
\begin{center}
$\displaystyle[\HH\ell,\HH m]f=-\ii\RR(\ell,m)f\overset{\textrm{(\ref{2fins2})}}{=}(\ii g(\RR(m,\ell)m,m)m)f=\kappa(\ii m)f$,
\end{center}
which proves (\ref{2fins5b}).

\begin{propo} \textup{(Berwald identity).}
Let $S$ be the canonical spray of $(M,F)$. Specifying a Berwald frame $(\ell,m)$, the Gauss curvature $\kappa$ and the main scalar $I$ of $(M,F)$ are related by
\begin{gather}\label{2fins11}
I\kappa+(\ii m)\kappa+\frac{1}{F}S(SI)I=0.
\end{gather}
\end{propo}

\begin{proof}
Applying the Jacobi identity to the vector fields $\HH\ell$, $\HH m$ and $\ii m$, taking into account relations (\ref{2fins5a})-(\ref{2fins5acs}) and the property $\hl{\nabla}F=0$, we get
\begin{center}
$\displaystyle 0=[\HH\ell,[\HH m,\ii m]]+[\HH m,[\ii m,\HH\ell]]+[\ii m,[\HH\ell,\HH m]]=[\HH\ell,\frac{1}{F}\HH\ell+I(\HH m)+(S I)\ii m]+[\HH m,\frac{1}{F}\HH m]+[\ii m,\kappa(\ii m)]=((\HH\ell)I)\HH m+I[\HH\ell,\HH m]+\HH\ell(S I)\ii m+(S I)[\HH\ell,\ii m]+(\ii m)\kappa(\ii m)=\frac{S I}{F}\HH m+(I\kappa)\ii m+\frac{1}{F}S(S I)\ii m-\frac{S I}{F}\HH m+(\ii m\kappa)\ii m=(I\kappa+\frac{1}{F}S(S I)+(\ii m)\kappa)\ii m$,
\end{center}
thus proving the Berwald identity.
\end{proof}

\begin{propo}\label{propokappa}
The Jacobi endomorphism of $(M,F)$ has the form
\begin{gather}\label{jacform}
\KK=\kappa F\mathbf{1}-\kappa\vl{\nabla}F\otimes\delta,
\end{gather}
hence the canonical spray of $(M,F)$ is isotropic.
\end{propo}

\begin{proof}
We evaluate both sides of (\ref{jacform}) at the members of the Berwald frame $(\ell,m)$.
\begin{center}
$\displaystyle\KK(\ell):=\VV[S,\HH\ell]=\VV[F(\HH\ell),\HH\ell]=\VV(-((\HH\ell)F)\HH\ell+F[\HH\ell,\HH\ell])=0$;\\
$\displaystyle\KK(m):=\VV[S,\HH m]=\VV[F(\HH\ell),\HH m]=\VV(-((\HH m)F)\HH\ell+F[\HH\ell,\HH m])\overset{\textrm{(\ref{2fins5b})}}{=}F\VV(\kappa\ii m)=(\kappa F)m$;\\
$\displaystyle(\kappa F\mathbf{1}-\kappa\vl{\nabla}F\otimes\delta)(\ell)=(\kappa F)\ell-\kappa\vl{\nabla}F(\ell)\delta\overset{\textrm{(\ref{can56})}}{=}(\kappa F)\ell-\kappa\delta=(\kappa F)\ell-(\kappa F)\ell=0$,\\
$\displaystyle(\kappa F\mathbf{1}-\kappa\vl{\nabla}F\otimes\delta)(m)=(\kappa F)m-\kappa\vl{\nabla}F(m)\delta\overset{\textrm{(\ref{can55})}}{=}(\kappa F)m-\kappa\frac{1}{F}\theta(m)\delta=(\kappa F)m-\kappa g(m,F\ell)\ell=(\kappa F)m$.
\end{center}
These prove our assertion.
\end{proof}

\textbf{Remarks.} (1) Since the function $K:=\frac{1}{n-1}\textrm{tr}\KK=\textrm{tr}\KK$ equals to $$g(\KK(m),m)=g(\kappa Fm,m)=\kappa F,$$ the Jacobi endomorphism of $(M,F)$ can also be written in the form $$\KK=K(\mathbf{1}-\frac{1}{F}\vl{\nabla}F\otimes\delta),$$ which is just formula (\ref{berwszabp}) obtained in the proof of Theorem \ref{berwszab}.

(2) Now we can easily check that \textit{the Weyl endomorphism of }$(M,F)$\textit{ is the zero transformation}.

Formula (\ref{Weylk}) reduces to the following: $$\mathbf{W}^{\circ}=\KK-\kappa F\mathbf{1}+(\textrm{tr}\RR)\otimes\delta.$$ As we have just seen, $\textrm{tr}\KK=\kappa F$. Since
\begin{center}
$\displaystyle(\textrm{tr}\RR)(\ell)=g(\RR(\ell,\ell),\ell)+g(\RR(m,\ell),m)=\kappa$,
$\displaystyle(\textrm{tr}\RR)(m)=g(\RR(\ell,m),\ell)+g(\RR(m,m),m)=0$,
\end{center}
we obtain
\begin{center}
$\displaystyle\mathbf{W}^{\circ}(\ell)=\KK(\ell)-(\kappa F)\ell+\kappa\delta=-(\kappa F)\ell+(\kappa F)\ell=0$,
$\displaystyle\mathbf{W}^{\circ}(m)=\KK(m)-(\kappa F)m=(\kappa F)m-(\kappa F)m=0$.
\end{center}

This remark and Proposision \ref{propokappa} confirm that \textit{condition} $\textrm{dim}M\geq 3$\textit{ in Theorem \ref{berwszab} may indeed be omitted}.

\begin{propo}\label{survbervpropo}
The only surviving component of the Berwald curvature of $(M,F)$ is
\begin{gather}\label{survb}
\BB(m,m)m=-\frac{2SI}{F}\ell+(\ii m(SI)+(\HH m)I)m,
\end{gather}
therefore $(M,F)$ is a Berwald manifold, if and only if, the h-Berwald differential of its main scalar vanishes.
\end{propo}

\begin{proof}
By Corollary \ref{lem5}, $\BB$ is indeed completely determined by the section $\BB(m,m)m$. Using (\ref{2fins9}), (\ref{2fins10}) and (\ref{2fins5c}), we obtain
\begin{center}
$\displaystyle\BB(m,m)m=\nabla_{\ii m}\nabla_{\HH m}m-\nabla_{\HH m}\nabla_{\ii m}m-\nabla_{[\ii m,\HH m]}m=\nabla_{\ii m}(S I)m+\nabla_{\HH m}(\frac{1}{F}\ell+Im)+\nabla_{\frac{1}{F}\HH\ell+I(\HH m)+(S I)\ii m}m=\ii m(S I)m-$\\
$\displaystyle(S I)(\frac{1}{F}\ell+Im)-\frac{1}{F^2}\HH m(F)\ell+((\HH m)I)m+I(S I)m+I(S I)m-\frac{S I}{F}\ell-(S I)Im=\ii m(S I)m-\frac{2S I}{F}\ell+((\HH m)I)m=-\frac{2S I}{F}\ell+(\ii m(S I)+(\HH m)I)m$;
\end{center}
therefore
\begin{center}
$\displaystyle\BB=0\textrm{  }\Leftrightarrow\textrm{  }S I=0\textrm{ and }\ii m(S I)+(\HH m)I=0\textrm{  }\Leftrightarrow$\\
$(\HH\ell)I=0\textrm{ and }(\HH m)I=0\textrm{  }\Leftrightarrow\textrm{  }\hl{\nabla}I=0$.
\end{center}
\end{proof}

\begin{lemma}\label{lem99nonz}
The only nonzero component of the trace of the Berwald tensor is the function
\begin{gather}\label{nonzero}
(\textrm{tr}\BB)(m,m)=\ii m(SI)+\HH m(I).
\end{gather}
\end{lemma}

\begin{proof}
Using the orthonormal frame $(\ell,m)$, $\textrm{tr}\BB$ can be determined as in the Riemannian case; cf., e.g., \cite{Petersen}, p.38. So we have
\begin{center}
$\displaystyle(\textrm{tr}\BB)(m,m)=g(\BB(\ell,m)m,\ell)+g(\BB(m,m)m,m)=$\\
$\displaystyle g(\BB(m,m)m,m)\overset{\textrm{(\ref{survb})}}{=}\ii m(SI)+\HH m(I)$.
\end{center}
\end{proof}

\begin{propo}
The only nonzero component of the Douglas curvature of $(M,F)$ is $\DD(m,m)m$; it is given by
\begin{gather}
\DD(m,m)m=-\frac{1}{3}(\frac{3SI}{E}+\ii m(\ii m(SI))+\ii m(\HH m(I))+2I\ii m(SI)+2I(\HH m)I)\delta.
\end{gather}
\end{propo}

\begin{proof}
In the previous chapter we have already remarked that $\DD$ vanishes if one of its argument is from $\textrm{span}(\delta)$, so $\DD$ is indeed determined by $\DD(m,m)m$. Now we calculate this component. By definition,
\begin{center}
$\displaystyle\DD(m,m)m=\BB(m,m)m-(\textrm{tr}\BB(m,m))m-\frac{1}{3}(\vl{\nabla}\textrm{tr}\BB)(m,m,m)\delta$;
\end{center}
our only task is to evaluate the last term at the right-hand side.
\begin{center}
$\displaystyle(\vl{\nabla}\textrm{tr}\BB)(m,m,m)=(\nabla_{\ii m}\textrm{tr}\BB)(m,m)=\ii m(\textrm{tr}\BB(m,m))-2\textrm{tr}\BB(\nabla_{\ii m}m,m)\overset{\textrm{(\ref{nonzero}),(\ref{2fins9})}}{=}\ii m(\ii m(SI)+\HH m(I))+2\textrm{tr}\BB(Im,m)=\ii m(\ii m(SI))+\ii m(\HH m(I))+2I\ii m(SI)+2I\HH m(I)$.
\end{center}
Now, putting together (\ref{survb}), (\ref{nonzero}) and this result, we get
\begin{center}
$\displaystyle\DD(m,m)m=-\frac{2SI}{F}\ell+(\ii m(SI+\HH m(I))m-(\ii m(SI)+\HH m(I))m-\frac{1}{3}(\ii m(\ii m(SI))+\ii m(\HH m(I))+2I\ii m(SI)+2I\HH m(I))\delta=-\frac{1}{3}\left(\frac{3SI}{E}+\ii m(\ii m(SI))+\ii m(\HH m(I))+2I\ii m(SI)+2I(\HH m)I\right)\delta$.
\end{center}
\end{proof}

\begin{propo}\label{survlandsprop}
The only surviving component of the Landsberg tensor of $(M,F)$ is
\begin{gather}\label{survlands}
\Pp(m,m,m)=SI,
\end{gather}
therefore $(M,F)$ has vanishing Landsberg tensor, if and only if, its main scalar is a first integral of the canonical spray, i.e., $SI=0$.
\end{propo}

\begin{proof}
By Corollary \ref{coro63}, $\Pp$ is indeed completely determined by the function $\Pp(m,m,m)$. It was shown in Step 3 of the proof of \ref{2finspropo1} that $$\nabla_{\HH m}m=\Pp(m,m,m)m.$$ Comparing this relation with (\ref{2fins10}) we obtain (\ref{survlands}).
\end{proof}

\textbf{Remark.} We have already shown in general that the direction independence of the Landsberg tensor implies its vanishing (Proposition \ref{landonly}). In the 2-dimensional case this conclusion may be deduced immediately:
\begin{center}
$\displaystyle\vl{\nabla}\Pp(m,\ell,m,m)=(\nabla_{\ii m}\Pp)(\ell,m,m)=\ii m(\Pp(\ell,m,m))-\Pp(\nabla_{\ii m}\ell,m,m)-2\Pp(\ell,\nabla_{\ii m}m,m)\overset{\textrm{\ref{coro63}}}{=}-\Pp(\nabla_{\ii m}\ell,m,m)\overset{\textrm{(\ref{2fins7a})}}{=}-\frac{1}{F}\Pp(m,m,m)$,
\end{center}
hence $\vl{\nabla}\Pp=0$ implies $\Pp=0$.

\begin{lemma}
The only not necessarily vanishing component of the stretch tensor of $(M,F)$ is
\begin{gather}\label{vanstretch}
\Sigma(\ell,m,m,m)=\frac{2}{F}S(SI).
\end{gather}
\end{lemma}

\begin{proof}
\ref{coro63} (ii) and relations $\nabla_{\HH\ell}\ell=\nabla_{\HH m}\ell=0$ imply immediately that
\begin{center}
\textit{if} $\displaystyle\ell\in\left\{\hul{X},\hul{Y},\hul{Z}\right\}$, \textit{then} $\displaystyle\hl{\nabla}\Pp(\ell,\hul{X},\hul{Y},\hul{Z})=\hl{\nabla}\Pp(m,\hul{X},\hul{Y},\hul{Z})=0$.
\end{center}
Thus $\Sigma$ is indeed uniquely determined by its value on the quadruple $(\ell,m,m,m)$. Now we calculate:
\begin{center}
$\displaystyle\Sigma(\ell,m,m,m)=2\hl{\nabla}\Pp(\ell,m,m,m)=2(\HH\ell(\Pp(m,m,m))-3\hl{\nabla}\Pp(\nabla_{\HH\ell}m,m,m))=\frac{2}{F}S(\Pp(m,m,m))\overset{\textrm{(\ref{survlands})}}{=}\frac{2}{F}S(SI)$.
\end{center}
\end{proof}

\begin{coro}
$(M,F)$ has vanishing stretch tensor, if and only if, its Gauss curvature satisfies $$I\kappa+(\ii m)\kappa=0.$$ In particular, 2-dimensional Finsler manifolds of vanishing Gauss curvature have vanishing stretch tensor.
\end{coro}

\begin{proof}
From relation (\ref{vanstretch}) and the Berwald identity (\ref{2fins11}) we obtain $$I\kappa+\ii m(\kappa)+\frac{1}{2}\Sigma(\ell,m,m,m)=0,$$ whence our assertion.
\end{proof}

\begin{propo}
$(M,F)$ is a p-Berwald manifold, if and only if, its main scalar satisfies the PDE $$(\ii m)SI+(\HH m)I=0.$$
\end{propo}

\begin{proof}
By definition, $(M,F)$ is a p-Berwald manifold if $\BB+\frac{1}{E}\Pp\otimes\delta=0$. In view of the preceding two propositions, this condition takes the following form:
\begin{center}
$\displaystyle 0=\BB(m,m,)m+\frac{1}{E}\Pp(m,m,m)\delta=-\frac{2SI}{F}\ell+(\ii m(SI)+(\HH m)I)m+\frac{2}{F}(SI)\ell=(\ii m(SI)+(\HH m)I)m$,
\end{center}
whence our assertion.
\end{proof}

\begin{coro}
The class of the 2-dimensional p-Berwald manifolds is the same as the class of the 2-dimensional weakly Berwald manifolds.
\end{coro}

\begin{proof}
This is immediate from the Proposition and Lemma \ref{lem99nonz}.
\end{proof}

\begin{coro}\label{finalcoro}
A 2-dimensional Finsler manifold is a Berwald manifold, if and only if, it is a weakly Berwald (or, equivalently, p-Berwald) manifold and has vanishing Landsberg tensor.
\end{coro}

\begin{proof}
The necessity is obvious: \textit{all} Berwald manifolds have traceless Berwald curvature and vanishing Landsberg tensor (see the Remark after Lemma \ref{lem711}).

Conversely, if $(M,F)$ is a 2-dimensional weakly Berwald manifold with vanishing Landsberg tensor, then $\textrm{tr}\BB=0$ implies by Lemma \ref{lem99nonz} $$\ii m(SI)+\HH m(I)=0,$$ and here $SI=0$, by Proposition \ref{survlandsprop}. So
\begin{center}
$\displaystyle(\hl{\nabla}I)(\ell)=(\HH\ell)I=\frac{1}{F}SI=0$ , $(\hl{\nabla}I)(m)=\HH m(I)=0$,
\end{center}
hence $\hl{\nabla}I=0$, therefore $(M,F)$ is a Berwald manifold by Proposition \ref{survbervpropo}.
\end{proof}

\textbf{Remark.} Corollary \ref{finalcoro} has also been obtained by S. Bácsó and R. Yoshikawa (\cite{BY}), using the tools of classical tensor calculus.

\newpage

\chapter{Rapcsák's equations: some consequences and applications}\label{ch10}

Following the terminology of \cite{SzV2} we say that a spray is \textit{Finsler-metrizable in a broad sense} or, after Z. Shen \cite{Shen}, \textit{projectively Finslerian}, or, less precisely, \textit{projectively metrizable}, if there exists a Finsler function whose canonical spray is projectively related to the given spray. If, in particular, the projective relation is trivial in the sense that the projective factor vanishes, then the spray will be called \textit{Finsler metrizable in a natural sense} or \textit{Finsler variational}. In this section, after some preparations, we establish different conditions concerning both types of metrizability of a spray.

\begin{propo}\label{propo66}
Let $(M,F)$ and $(M,\overline{F})$ be Finsler manifolds, and let the geometric data arising from $\overline{F}$ be distinguished by bar. Suppose that the cano\-nical sprays $S$ and $\overline{S}$ of $(M,F)$ and $(M,\overline{F})$ are projectively related, namely $\overline{S}=S-2PC$. Then
\begin{gather}\label{M}
2P=\frac{S\overline{F}}{\overline{F}},
\end{gather}
and
\begin{gather}\label{N}
\nabla_S\overline{\CC}_{\flat}=P\overline{\CC}_{\flat}+\overline{\Pp}.
\end{gather}
\end{propo}

\begin{proof}
Since $\overline{S}$ is horizontal with respect to the canonical connection of $(M,\overline{F})$, we obtain $$0=\overline{S}\textrm{ }\overline{F}=(S-2PC)\overline{F}=S\overline{F}-2P\overline{F},$$ so (\ref{M}) is valid. To prove the second relation, let $X$, $Y$, $Z$ be arbitrary vector fields on $M$. Then, applying (\ref{corr37}) and (\ref{49ny}),
\begin{center}
$\displaystyle\left(\nabla_S\overline{\CC}_{\flat}\right)(\kal{X},\kal{Y},\kal{Z})=S(\overline{\CC}_{\flat}(\kal{X},\kal{Y},\kal{Z}))-\overline{\CC}_{\flat}(\nabla_S\kal{X},\kal{Y},\kal{Z})-\overline{\CC}_{\flat}(\kal{X},\nabla_S\kal{Y},\kal{Z})-\overline{\CC}_{\flat}(\kal{X},\kal{Y},\nabla_S\kal{Z})=\overline{S}(\overline{\CC}_{\flat}(\kal{X},\kal{Y},\kal{Z}))+2PC(\overline{\CC}_{\flat}(\kal{X},\kal{Y},\kal{Z}))-\overline{\CC}_{\flat}(\overline{\nabla}_{\overline{S}}\kal{X},\kal{Y},\kal{Z})-\overline{\CC}_{\flat}(\kal{X},\overline{\nabla}_{\overline{S}}\kal{Y},\kal{Z})-\overline{\CC}_{\flat}(\kal{X},\kal{Y},\overline{\nabla}_{\overline{S}}\kal{Z})+3P\overline{\CC}_{\flat}(\kal{X},\kal{Y},\kal{Z})$.
\end{center}
Since $\overline{\CC}_{\flat}$ is homogeneous of degree $-1$, $$C\overline{\CC}_{\flat}(\kal{X},\kal{Y},\kal{Z})=\left(\nabla_C\overline{\CC}_{\flat}\right)(\kal{X},\kal{Y},\kal{Z})=-\overline{\CC}_{\flat}(\kal{X},\kal{Y},\kal{Z}),$$
so we get $$\left(\nabla_S\overline{\CC}_{\flat}\right)(\kal{X},\kal{Y},\kal{Z})=\overline{\nabla}_{\overline{S}}\overline{\CC}_{\flat}(\kal{X},\kal{Y},\kal{Z})+P\overline{\CC}_{\flat}(\kal{X},\kal{Y},\kal{Z})\overset{\textrm{\ref{propo65}}}{=}(\overline{\mathbf{P}}+P\overline{\CC}_{\flat})(\kal{X},\kal{Y},\kal{Z}),$$
thus proving Proposition \ref{propo66}.
\end{proof}

\begin{lemma}\label{lemma71}
Let a spray $S$ over $M$ be given. Let $\HH$ be the Ehresmann connection associated to $S$, and let $\hh:=\HH\circ\jj$ be the horizontal projector associated to $\HH$. Then for any smooth function $F$ on $\Tk M$ we have
\begin{gather}\label{Rapcsaka}
2d_{\hh}F=d(F-CF)-i_Sdd_{\JJ}F+d_{\JJ}i_SdF.
\end{gather}
\end{lemma}

This useful relation was found by J. Klein, see \cite{Klein}, section 3.2. Its validity may be checked by brute force, evaluating both sides of (\ref{Rapcsaka}) on vertical lifts $\vl{X}$ and complete lifts $\cl{X}$, $X\in\modx{M}$. It is possible, however, to verify (\ref{Rapcsaka}) also by a more elegant, completely `argumentum-free' reasoning, see again \cite{Klein}, and \cite{Szilasi}.

\begin{propo}\label{propo72}
Let $(M,\overline{F})$ be a Finsler-manifold with energy function $\overline{E}:=\frac{1}{2}\overline{F}^2$ and canonical spray $\overline{S}$, given by $i_{\overline{S}}dd_{\JJ}\overline{E}=-d\overline{E}$ on $\Tk M$. Suppose $S$ is a further spray over $M$, and let $\HH$ be the Ehresmann connection associated to $S$.

$S$ is projectively related to $\overline{S}$, if and only if, for each vector field $X\in\modx{M}$ we have
\end{propo}
\begin{itemize}
\item[(R$_{\textrm{1}}$)] $2\hl{X}\overline{F}=\vl{X}(S\overline{F})\textrm{  }\textrm{  }(\hl{X}:=\HH\kal{X})$.
\end{itemize}

\begin{proof}
Let $\overline{\HH}$ be the canonical connection of $(M,\overline{F})$. Suppose first that $S$ and $\overline{S}$ are projectively related, namely $\overline{S}=S-2PC$, $P\in\modc{\Tk M}\cap C^1(TM)$. If $X\in\modx{M}$, $\hl{X}=\HH(\kal{X})$, $\hlfel{X}=\overline{\HH}(\kal{X})$, then $$\hlfel{X}=\hl{X}-P\vl{X}-(\vl{X}P)C$$ by (\ref{corr33}). By Proposition \ref{propo66} we have $2P=\frac{S\overline{F}}{\overline{F}}$. Since $\overline{\HH}$ is conservative, then we obtain:
\begin{center}
$\displaystyle 0=2\hlfel{X}\overline{F}=(2\hl{X}-\frac{S\overline{F}}{\overline{F}}\vl{X}-\vl{X}(\frac{S\overline{F}}{\overline{F}})C)\overline{F}=2\hl{X}\overline{F}-\frac{S\overline{F}}{\overline{F}}(\vl{X}\overline{F})-\frac{1}{\overline{F}}\vl{X}(S\overline{F})\overline{F}+\frac{1}{\overline{F}^2}(S\overline{F})(\vl{X}\overline{F})\overline{F}=2\hl{X}\overline{F}-\vl{X}(S\overline{F})$.
\end{center}
This proves the validity of (R$_{\textrm{1}}$) if $S$ and $\overline{S}$ are projectively related.

Conversely, suppose that (R$_{\textrm{1}}$) is satisfied. Then, for all $X\in\modx{M}$,
\begin{gather}\label{Rapcsakb}
2\hl{X}\overline{E}=2\overline{F}(\hl{X}\overline{F})\overset{\textrm{(R}_{\textrm{1}}\textrm{)}}{=}\overline{F}(\vl{X}(S\overline{F}))=\vl{X}(S\overline{E})-(\vl{X}\overline{F})(S\overline{F}).
\end{gather}
Since $\overline{E}-C\overline{E}=-\overline{E}$, we obtain by Lemma \ref{lemma71}
\begin{gather}\label{Rapcsakc}
d_{\JJ}i_Sd\overline{E}-2d_{\hh}\overline{E}=i_Sdd_{\JJ}\overline{E}+d\overline{E}.
\end{gather}
Next we prove that the left-hand side of (\ref{Rapcsakc}) equals to $\frac{S\overline{F}}{\overline{F}}i_Cdd_{\JJ}\overline{E}$.

An easy calculation shows, on the one hand, that for any vector field $X$ on $M$,
\begin{center}
$\displaystyle (d_{\JJ}i_Sd\overline{E}-2d_{\hh}\overline{E})(\vl{X})=0=\frac{S\overline{F}}{\overline{F}}i_Cdd_{\JJ}\overline{E}(\vl{X})$.
\end{center}
On the other hand,
\begin{center}
$\displaystyle (d_{\JJ}i_Sd\overline{E}-2d_{\hh}\overline{E})(\cl{X})=\JJ\cl{X}(i_Sd\overline{E})-2(\hh\cl{X})\overline{E}=$\\
$\vl{X}(S\overline{E})-2\hl{X}E\overset{\textrm{(\ref{Rapcsakb})}}{=}(\vl{X}\overline{F})(S\overline{F})$,
\end{center}
while
\begin{center}
$\displaystyle\frac{S\overline{F}}{\overline{F}}i_Cdd_{\JJ}\overline{E}(\cl{X})=\frac{S\overline{F}}{\overline{F}}dd_{\JJ}\overline{E}(C,\cl{X})=\frac{S\overline{F}}{\overline{F}}\left(Cd_{\JJ}\overline{E}(\cl{X})-\cl{X}d_{\JJ}\overline{E}(C)-d_{\JJ}\overline{E}[C,\cl{X}]\right)=\frac{S\overline{F}}{\overline{F}}C(\vl{X}\overline{E})=\frac{S\overline{F}}{\overline{F}}([C,\vl{X}]\overline{E}+\vl{X}(C\overline{E}))=\frac{S\overline{F}}{\overline{F}}(\vl{X}\overline{E})=(S\overline{F})(\vl{X}\overline{F})$,
\end{center}
hence
\begin{gather}\label{Rapcsakd}
d_{\JJ}i_Sd\overline{E}-2d_{\hh}\overline{E}=\frac{S\overline{F}}{\overline{F}}i_Cdd_{\JJ}\overline{E},
\end{gather}
as we claimed. (\ref{Rapcsakc}) and (\ref{Rapcsakd}) imply that $$i_Sdd_{\JJ}\overline{E}+d\overline{E}=i_{\frac{S\overline{F}}{\overline{F}}C}dd_{\JJ}\overline{E}$$ whence $$i_{S-\frac{S\overline{F}}{\overline{F}}C}dd_{\JJ}\overline{E}=-d\overline{E}.$$ Since $\overline{S}$ is uniquely determined on $\Tk M$ by the `Euler-Lagrange equation' $i_{\overline{S}}dd_{\JJ}\overline{E}=-d\overline{E}$, we conclude that $\overline{S}=S-\frac{S\overline{F}}{\overline{F}}C$, which proves the Proposition.
\end{proof}

(R$_{\textrm{1}}$) provides a necessary and sufficient condition for the Finsler-metrizability of a spray in a broad sense. In terms of classical tensor calculus, it was first formulated by A. Rapcsák \cite{Rapcsak}, so it will be quoted as \textit{Rapcsák's equation} for $\Ffel$ with respect to $S$.

Now we derive a 'more intrinsic' expression of (R$_{\textrm{1}}$), showing that it can also be written in the form

\begin{itemize}
\item[(R$_{\textrm{2}}$)] $\nabla_S\vl{\nabla}\overline{F}=\hl{\nabla}\overline{F}$.
\end{itemize}

Indeed, for any vector field $X\in\modx{M}$ we have

\begin{center}
$\displaystyle\nabla_S\vl{\nabla}\overline{F}(\kal{X})=S(\vl{X}\overline{F})-\vl{\nabla}\overline{F}(\nabla_S\kal{X})=S(\vl{X}\overline{F})-(\ii\nabla_S\kal{X})\overline{F}=S(\vl{X}\overline{F})-\vv[S,\vl{X}]\overline{F}\overset{\textrm{(\ref{ehre6})}}{=}S(\vl{X}\overline{F})-(\vv\cl{X})\overline{F}=S(\vl{X}\overline{F})-\cl{X}\overline{F}+\hl{X}\overline{F}=[S,\vl{X}]\overline{F}+\vl{X}(S\overline{F})-\cl{X}\overline{F}+\hl{X}\overline{F}\overset{\textrm{(\ref{ehre6})}}{=}(\cl{X}-2\hl{X})\overline{F}+\vl{X}(S\overline{F})-\cl{X}\overline{F}+\hl{X}\overline{F}=\vl{X}(S\overline{F})-\hl{X}\overline{F}$,
\end{center}
hence
\begin{center}
$\displaystyle\nabla_S\vl{\nabla}\overline{F}=\hl{\nabla}\overline{F}\textrm{  }\Leftrightarrow\textrm{  }\vl{X}(S\overline{F})-\hl{X}\overline{F}=\hl{X}\overline{F}\textrm{  }$ for all $X\in\modx{M}$.
\end{center}

This proves the equivalence of (R$_{\textrm{1}}$) and (R$_{\textrm{2}}$). Rapcsák equations (R$_{\textrm{1}}$), (R$_{\textrm{2}}$) have several further equivalents, we collect here some of them:

\begin{itemize}
\item[(R$_{\textrm{3}}$)] $i_Sdd_{\JJ}\overline{F}=0$;
\item[(R$_{\textrm{4}}$)] $i_{\delta}\hl{\nabla}\vl{\nabla}\overline{F}=\hl{\nabla}\overline{F}$;
\item[(R$_{\textrm{5}}$)] $d_{\hh}d_{\JJ}\overline{F}=0$;
\item[(R$_{\textrm{6}}$)] $\hl{\nabla}\vl{\nabla}\overline{F}(\hul{X},\hul{Y})=\hl{\nabla}\vl{\nabla}\overline{F}(\hul{Y},\hul{X})$ ;
\item[(R$_{\textrm{7}}$)] $\vl{\nabla}\hl{\nabla}\overline{F}(\hul{X},\hul{Y})=\vl{\nabla}\hl{\nabla}\overline{F}(\hul{Y},\hul{X})$
\end{itemize}
(in (R$_{\textrm{6}}$) and (R$_{\textrm{7}}$) $\hul{X}$ and $\hul{Y}$ are arbitrary sections along $\tauk$).

Details on a proof of the equivalence of conditions (R$_{\textrm{1}}$)-(R$_{\textrm{7}}$) can be found in \cite{SzV2}, \cite{Szilasi}, \cite{Sz2}, \cite{Sz3}. We note only that the equivalence of (R$_{\textrm{6}}$) and (R$_{\textrm{7}}$) is an immediate consequence of the Ricci identity (\ref{ric1}), while the equivalence of (R$_{\textrm{2}}$) and (R$_{\textrm{4}}$) follows from the identity $i_{\delta}\hl{\nabla}\vl{\nabla}\overline{F}=\nabla_S\vl{\nabla}\overline{F}$.

\begin{propo} \textup{(criterion for Finsler variationality).}\label{propo73}
Let $S$ be a spray over $M$, and let $\HH$ be the Ehresmann connection associated to $S$. $S$ is the canonical spray of a Finsler manifold $(M,\overline{F})$, if and only if, $d\Ffel\circ\HH=0$.
\end{propo}

\begin{proof}
The necessity is obvious since the canonical connection of $(M,\Ffel)$ is conservative. To prove the sufficiency, suppose that $d\Ffel\circ\HH=0$. Then for all $X\in\modx{M}$ we have
\begin{gather}\label{Rapcsake}
d\Ffel\circ\HH(\kal{X})=d\Ffel(\hl{X})=\hl{X}\Ffel=0.
\end{gather}
Since the horizontal lifts $\hl{X}$, $X\in\modx{M}$ generate the $\modc{\Tk M}$-module of $\HH$-horizontal vector fields, this implies that for any $\HH$-horizontal vector field $\xi$ on $\Tk M$ we have $\xi\Ffel=0$. In particular, $S$ is also $\HH$-horizontal, so $S\Ffel=0$ holds too. Then Rapcsák's equation (R$_{\textrm{1}}$) is valid trivially: both sides of the relation vanish identically. By Proposition \ref{propo72}, from this it follows that $\overline{S}$ and $S$ are projectively related: $$\overline{S}=S-2PC\textrm{ , }P\in C^1(TM)\cap\modc{\Tk M}.$$ However, $P\overset{\textrm{(\ref{M})}}{=}\frac{1}{2}\frac{S\Ffel}{\Ffel}=0$, so we obtain the required equality $\overline{S}=S$.
\end{proof}

\textbf{Remark.} In his excellent textbook \cite{Laug}, written in a 'semi-classical style', D. Laugwitz formulates and proves the following theorem:
\textit{The paths of a system} $(x^i)''+2H^i(x,x')=0\textrm{  }(i\in\left\{1,\dots n\right\})$\textit{ are the geodesics of a Finsler function} $F$\textit{, if and only if, }$F$\textit{ is invariant under the parallel displacement} $$\frac{d\xi^i}{dt}+H^i_r(x,\xi)\frac{dx^r}{dt}=0\textrm{ , }H^i_r:=\pp{H^i}{y^r}$$ \textit{associated with the sytem of paths}. (\cite{Laug}, Theorem 15.8.1.) Here we slightly modified Laugwitz's formulation and notation. The 'system', actually a SODE, is given in a chart $(\tau^{-1}(\UU),(x^i,y^i))$ on $TM$, induced by a chart $(\UU,(u^i))$ on $M$, so the coordinate functions are $$x^i:=u^i\circ\tau=\vl{(u^i)}\textrm{ , }y^i:=\cl{(u^i)}\textrm{ ; }i\in\left\{1,\dots,n\right\}.$$ It may be easily seen that our Proposition \ref{propo73} is just an intrinsic reformulation of Laugwitz's metrization theorem. Laugwitz's proof takes more than one page and applies a totally different argument.

\begin{coro} \textup{(the uniqueness of the canonical connection).}
Let $(M,\Ffel)$ be a Finsler manifold. If $\HH$ is a torsion-free, homogeneous Ehresmann connection over $M$ such that $d\Ffel\circ\HH=0$, then $\HH$ is the canonical connection of $(M,\Ffel)$.
\end{coro}

\begin{proof}
Since $\HH$ is torsion-free and homogeneous, Corollary 6 in section 3 of \cite{Szilasi} assures that $\HH$ is associated to a spray. Then the condition $d\Ffel\circ\HH=0$ implies by the preceding Proposition that this spray is the canonical spray, and hence $\HH$ is the canonical connection of $(M,\Ffel)$.
\end{proof}

\textbf{Remark.} The uniqueness proof presented here is based, actually, on the Rapcsák equations. The idea that they may be applied also in this context is due to Z. I. Szabó \cite{Szabo}.

Our next results may be considered as \textit{necessary conditions for the Finsler metrizability of a spray in a broad sense}.

\begin{propo}\label{Rapp75}
Let $S$ be a spray over $M$, and let $\nabla=(\hl{\nabla},\vl{\nabla})$ be the Berwald derivative induced by the Ehresmann connection associated to $S$. If a Finsler function $\Ffel: TM\rightarrow\valR$ satisfies one (and hence all) of the Rapcsák equations with respect to $S$, then
\begin{gather}\label{Rapalpha}
\nabla_S\vl{\nabla}\vl{\nabla}\Ffel=0.
\end{gather}
\end{propo}

\begin{proof}
For any vector fields $X$, $Y$ on $M$ we have
\begin{center}
$\displaystyle\left(\nabla_S\vl{\nabla}\vl{\nabla}\overline{F}\right)(\kal{X},\kal{Y})=\hl{\nabla}\vl{\nabla}(\vl{\nabla}\overline{F})(\delta,\kal{X},\kal{Y})\overset{\textrm{(\ref{ric34})}}{=}\vl{\nabla}\hl{\nabla}(\vl{\nabla}\overline{F})(\kal{X},\delta,\kal{Y})+\vl{\nabla}\overline{F}(\mathbf{B}(\kal{X},\delta)\kal{Y})\overset{(\ref{lem5kepl})}{=}\vl{\nabla}\hl{\nabla}(\vl{\nabla}\overline{F})(\kal{X},\delta,\kal{Y})=\vl{X}\left(\hl{\nabla}\vl{\nabla}\overline{F}(\delta,\kal{Y})\right)-\hl{\nabla}\vl{\nabla}\overline{F}(\kal{X},\kal{Y})\overset{\textrm{(\ref{ric1})}}{=}\vl{X}\left(\vl{\nabla}\hl{\nabla}\overline{F}(\kal{Y},\delta)\right)-\vl{\nabla}\hl{\nabla}\overline{F}(\kal{Y},\kal{X})=\vl{X}(\vl{Y}(S\Ffel)-\hl{\nabla}\Ffel(\kal{Y}))-\vl{Y}(\hl{X}\Ffel)\overset{\textrm{(R}_{\textrm{1}}\textrm{)}}{=}\vl{X}(2\hl{Y}\Ffel-\hl{Y}\Ffel)-\vl{Y}(\hl{X}\Ffel)=\vl{X}(\hl{Y}\Ffel)-\vl{Y}(\hl{X}\Ffel)=\vl{\nabla}\hl{\nabla}\Ffel(\kal{X},\kal{Y})-\vl{\nabla}\hl{\nabla}\Ffel(\kal{Y},\kal{X})\overset{\textrm{(R}_{\textrm{7}}\textrm{)}}{=}0$.
\end{center}
\end{proof}

\begin{coro}
Under the assumptions of the Proposition above, let $$\overline{\mu}:=\vl{\nabla}\vl{\nabla}\Ffel\overset{\textrm{\textup{(\ref{fins48})}}}{=}\frac{1}{F}\overline{\eta},$$ where $\overline{\eta}$ is the angular metric tensor of the Finsler manifold $(M,\Ffel)$. Then 
\begin{gather}\label{Rapcsbeta}
\nabla_S\vl{\nabla}\overline{\mu}+\hl{\nabla}\overline{\mu}=0.
\end{gather}
\end{coro}

\begin{proof}
Let $Y,Z\in\modx{M}$. Then by Proposition \ref{Rapp75},
\begin{center}
$\displaystyle\hl{\nabla}\overline{\mu}(\delta,\kal{Y},\kal{Z})=\nabla_S\vl{\nabla}\vl{\nabla}\overline{F}(\kal{Y},\kal{Z})=0$.
\end{center}
Operating on both sides by $\vl{X}$, where $X\in\modx{M}$, we obtain
\begin{center}
$\displaystyle 0=\vl{X}(\hl{\nabla}\overline{\mu}(\delta,\kal{Y},\kal{Z}))=\left(\nabla_{\vl{X}}\hl{\nabla}\overline{\mu}\right)(\delta,\kal{Y},\kal{Z})+\hl{\nabla}\overline{\mu}(\kal{X},\kal{Y},\kal{Z})=\vl{\nabla}\hl{\nabla}\overline{\mu}(\kal{X},\delta,\kal{Y},\kal{Z})+\hl{\nabla}\overline{\mu}(\kal{X},\kal{Y},\kal{Z})\overset{\textrm{(\ref{ric34}),(\ref{lem5kepl})}}{=}\hl{\nabla}\vl{\nabla}\overline{\mu}(\delta,\kal{X},\kal{Y},\kal{Z})+\hl{\nabla}\overline{\mu}(\kal{X},\kal{Y},\kal{Z})=(\nabla_S\vl{\nabla}\overline{\mu}+\hl{\nabla}\overline{\mu})(\kal{X},\kal{Y},\kal{Z})$,
\end{center}
which proves the Corollary.
\end{proof}

\begin{theorem}
Let $(M,\Ffel)$ be a Finsler manifold with canonical spray $\overline{S}$; let $\overline{\lambda}:=\frac{\vl{\nabla}\Ffel}{\Ffel}$, $\overline{\mu}:=\vl{\nabla}\vl{\nabla}\Ffel$; and let $\overline{\CC}_{\flat}$ and $\overline{\Pp}$ be the Cartan and the Landsberg tensor of $(M,\Ffel)$, respectively. If $\Ffel$ satisfies one (and hence all) of the Rapcsák equations with respect to a spray $S$, and $P$ is the projective factor between $S$ and $\overline{S}$, then
\begin{gather}\label{Rapcsakdelta}
\hl{\nabla}\overline{\mu}=\textrm{\textup{Sym}}(\overline{\lambda}\otimes\overline{\mu})+\frac{2}{\Ffel}(P\overline{\CC}_{\flat}-\overline{\Pp}).
\end{gather}
\end{theorem}

\begin{proof}
\begin{center}
$\displaystyle\hl{\nabla}\overline{\mu}\overset{\textrm{(\ref{Rapcsbeta})}}{=}-\nabla_S\vl{\nabla}\overline{\mu}\overset{\textrm{(\ref{sym60})}}{=}-2\nabla_S\left(\frac{1}{\Ffel}\overline{\CC}_{\flat}\right)+\nabla_S\textrm{Sym}(\overline{\lambda}\otimes\overline{\mu})=\frac{2S\Ffel}{\Ffel^2}\overline{\CC}_{\flat}-\frac{2}{\Ffel}\nabla_S\overline{\CC}_{\flat}+\textrm{Sym}(\nabla_S\overline{\lambda}\otimes\overline{\mu}+\overline{\lambda}\otimes\nabla_S\overline{\mu})\overset{\textrm{(\ref{M}),(\ref{N}),(\ref{Rapalpha})}}{=}\frac{4P}{\Ffel}\overline{\CC}_{\flat}-\frac{2P}{\Ffel}\CC_{\flat}-\frac{2}{\Ffel}\overline{\Pp}+\textrm{Sym}(\nabla_S\overline{\lambda}\otimes\overline{\mu})=\textrm{Sym}(\nabla_S\overline{\lambda}\otimes\overline{\mu})+\frac{2}{\Ffel}(P\overline{\CC}_{\flat}-\overline{\Pp})$.
\end{center}
\end{proof}

\textbf{Remark.} Relation (\ref{Rapcsakdelta}) is an intrinsic, index and argumentrum free version of formula (2.4) in \cite{Bacso}. 

\begin{coro}
If a Finsler function $\Ffel$ satisfies a Rapcsák equation with respect to a spray $S$, and $(\hl{\nabla},\vl{\nabla})$ is the Berwald derivative induced by $S$, then the tensor $\hl{\nabla}\overline{\mu}=\hl{\nabla}\vl{\nabla}\vl{\nabla}\Ffel$ is totally symmetric.
\end{coro}

\begin{proof}
The total symmetry of $\hl{\nabla}\overline{\mu}$ can be read from (\ref{Rapcsakdelta}).
\end{proof}

\textbf{Remark.} We show that the total symmetry of $\hl{\nabla}\overline{\mu}$ (modulo a Rapcsák equation) may also be verified immediately, independently of (\ref{Rapcsakdelta}).

Using the Ricci formula (\ref{ric34}), for any vector fields $X$, $Y$, $Z$ on $M$ we have $$\hl{\nabla}\vl{\nabla}\vl{\nabla}\Ffel(\kal{X},\kal{Y},\kal{Z})=\vl{\nabla}\hl{\nabla}\vl{\nabla}\Ffel(\kal{Y},\kal{X},\kal{Z})+\vl{\nabla}\Ffel(\BB(\kal{Y},\kal{X})\kal{Z}).$$ The second term at the right-hand side is totally symmetric by Lemma \ref{lem4}, so our only task is to show that the first term also has this symmetry property. Since $$\vl{\nabla}\hl{\nabla}\Ffel(\kal{X},\kal{Y})=\vl{X}(\hl{Y}\Ffel)\textrm{  ,  }\vl{\nabla}\hl{\nabla}\Ffel(\kal{Y},\kal{X})=\vl{Y}(\hl{X}\Ffel),$$ Rapcsák's equation (R$_{\textrm{7}}$) implies that
\begin{gather}\label{star}
\vl{X}(\hl{Y}\Ffel)=\vl{Y}(\hl{X}\Ffel)\textrm{  ;  }X,Y\in\modx{M}.
\end{gather}
Thus
\begin{center}
$\displaystyle\vl{\nabla}\hl{\nabla}\vl{\nabla}\Ffel(\kal{Y},\kal{X},\kal{Z})=\vl{Y}(\hl{\nabla}\vl{\nabla}\Ffel(\kal{X},\kal{Z}))=\vl{Y}(\hl{X}(\vl{Z}\Ffel)-\vl{\nabla}\Ffel(\nabla_{\hl{X}}\kal{Z}))=\vl{Y}(\hl{X}(\vl{Z}\Ffel)-[\hl{X},\vl{Z}]\Ffel)=\vl{Y}(\vl{Z}(\hl{X}\Ffel))=\vl{Z}(\vl{Y}(\hl{X}\Ffel))\overset{\textrm{(\ref{star})}}{=}\vl{Z}(\vl{X}(\hl{Y}\Ffel))=\vl{X}(\vl{Z}(\hl{Y}\Ffel))\overset{\textrm{(\ref{star})}}{=}\vl{X}(\vl{Y}(\hl{Z}\Ffel))=\vl{Y}(\vl{X}(\hl{Z}\Ffel))$,
\end{center}
as was to be checked.

\begin{propo}
Let $S$ be a spray over $M$, and suppose that a Finsler function $\Ffel: TM\rightarrow\valR$ satisfies a Rapcsák equation with respect to $S$. If $\RR$ is the curvature of the Ehresmann connection associated to $\HH$, then
\begin{gather}\label{Pe}
\underset{(\kal{X},\kal{Y},\kal{Z})}{\mathfrak{S}}\overline{\mu}(\RR(\kal{X},\kal{Y}),\kal{Z})=0.
\end{gather}
\end{propo}

\begin{proof}
Let $(\hl{\nabla},\vl{\nabla})$ be the Berwald derivative determined by $\HH$. First we show that
\begin{gather}\label{Qu}
\hl{\nabla}\hl{\nabla}\overline{\ell}_{\flat}(\kal{X},\kal{Y},\kal{Z})
\end{gather}
($\overline{\ell}_{\flat}:=\vl{\nabla}\Ffel$; $X,Y,Z\in\modx{M}$) is symmetric in its last two arguments. Indeed,
\begin{center}
$\displaystyle\hl{\nabla}\hl{\nabla}\overline{\ell}_{\flat}(\kal{X},\kal{Y},\kal{Z})=\left(\nabla_{\hl{X}}(\hl{\nabla}\overline{\ell}_{\flat})\right)(\kal{Y},\kal{Z})=\hl{X}(\hl{\nabla}\overline{\ell}_{\flat}(\kal{Y},\kal{Z}))-\hl{\nabla}\overline{\ell}_{\flat}(\nabla_{\hl{X}}\kal{Y},\kal{Z})-\hl{\nabla}\overline{\ell}_{\flat}(\kal{Y},\nabla_{\hl{X}}\kal{Z})\overset{\textrm{(R}_{\textrm{6}}\textrm{)}}{=}\hl{X}(\hl{\nabla}\overline{\ell}_{\flat}(\kal{Z},\kal{Y}))-\hl{\nabla}\overline{\ell}_{\flat}(\nabla_{\hl{X}}\kal{Z},\kal{Y})-\hl{\nabla}\overline{\ell}_{\flat}(\kal{Z},\nabla_{\hl{X}}\kal{Y})=\left(\nabla_{\hl{X}}(\hl{\nabla}\overline{\ell}_{\flat})\right)(\kal{Z},\kal{Y})=\hl{\nabla}\hl{\nabla}\overline{\ell}_{\flat}(\kal{X},\kal{Z},\kal{Y})$,
\end{center}
which proves our claim.

Next we apply the Ricci identity (\ref{hhricci25}) to (\ref{Qu}):
\begin{center}
$\displaystyle\hl{\nabla}\hl{\nabla}\overline{\ell}_{\flat}(\kal{X},\kal{Y},\kal{Z})=\hl{\nabla}\hl{\nabla}\overline{\ell}_{\flat}(\kal{Y},\kal{X},\kal{Z})-\overline{\ell}_{\flat}(\Hh(\kal{X},\kal{Y})\kal{Z})-\overline{\mu}(\RR(\kal{X},\kal{Y}),\kal{Z})$.
\end{center}
Interchanging $\kal{X}$, $\kal{Y}$ and $\kal{Z}$ cyclically:
\begin{center}
$\displaystyle\hl{\nabla}\hl{\nabla}\overline{\ell}_{\flat}(\kal{Y},\kal{Z},\kal{X})=\hl{\nabla}\hl{\nabla}\overline{\ell}_{\flat}(\kal{Z},\kal{Y},\kal{X})-\overline{\ell}_{\flat}(\Hh(\kal{Y},\kal{Z})\kal{X})-\overline{\mu}(\RR(\kal{Y},\kal{Z}),\kal{X})$,\\
$\hl{\nabla}\hl{\nabla}\overline{\ell}_{\flat}(\kal{Z},\kal{X},\kal{Y})=\hl{\nabla}\hl{\nabla}\overline{\ell}_{\flat}(\kal{X},\kal{Z},\kal{Y})-\overline{\ell}_{\flat}(\Hh(\kal{Z},\kal{X})\kal{Y})-\overline{\mu}(\RR(\kal{Z},\kal{X}),\kal{Y})$.
\end{center}
We add these three relations. Then, using the Bianchi identity (\ref{Bi}) and the symmetry of $\hl{\nabla}\hl{\nabla}\overline{\ell}_{\flat}$ in its last two variables, relation (\ref{Pe}) drops.
\end{proof}

\textbf{Remark.} In the language of classical tensor calculus, relation (\ref{Pe}) was first formulated by A. Rapcsák \cite{Rapcsak}. For another index-free treatment, using Grifone's formalism, we refer to \cite{SzV2}.

\begin{lemma}
Let a spray $S: TM\rightarrow TTM$ and a Finsler function\\ $\Ffel: TM\rightarrow\valR$ be given. If $\overline{\mu}:=\vl{\nabla}\vl{\nabla}\Ffel$; $\RR$ is the curvature, and $\KK$ is the Jacobi endomorphism of the Ehresmann connection associated to $S$, then relation (\ref{Pe}) is equivalent to the condition
\begin{gather}\label{Yo}
\overline{\mu}(\KK(\kal{X}),\kal{Y})=\overline{\mu}(\kal{X},\KK(\kal{Y}))\textrm{  ;  }X,Y\in\modx{M}.
\end{gather}
\end{lemma}

\begin{proof}
We recall that by (\ref{jacobi22}), $\KK$ and $\RR$ are related by $$\KK(\hul{X})=\RR(\hul{X},\delta)\textrm{  ,  }\hul{X}\in\szel{\pik}.$$

\textit{(\ref{Pe})}$\Rightarrow$\textit{(\ref{Yo})} By assumption, for any vector fields $X$, $Y$ on $M$ we have $$\overline{\mu}(\RR(\kal{X},\delta),\kal{Y})+\overline{\mu}(\RR(\delta,\kal{Y}),\kal{X})+\overline{\mu}(\RR(\kal{Y},\kal{X}),\delta)=0,$$ or, equivalently, $$\overline{\mu}(\KK(\kal{X}),\kal{Y})-\overline{\mu}(\kal{X},\KK(\kal{Y}))=\overline{\mu}(\RR(\kal{X},\kal{Y}),\delta).$$ We show that the right-hand side vanishes.
\begin{center}
$\displaystyle\overline{\mu}(\RR(\kal{X},\kal{Y}),\delta)=\vl{\nabla}\vl{\nabla}\Ffel(\delta,\RR(\kal{X},\kal{Y}))=(\nabla_C\vl{\nabla}\Ffel)(\RR(\kal{X},\kal{Y}))=C(\ii\RR(\kal{X},\kal{Y})\Ffel)-\vl{\nabla}\Ffel(\nabla_C(\RR(\kal{X},\kal{Y})))=C(\ii\RR(\kal{X},\kal{Y})\Ffel)-\vl{\nabla}\Ffel(\nabla_C\RR(\kal{X},\kal{Y}))\overset{\textrm{(\ref{coro43})}}{=}C(\ii\RR(\kal{X},\kal{Y})\Ffel)-\ii\RR(\kal{X},\kal{Y})\Ffel=[C,\ii\RR(\kal{X},\kal{Y})]\Ffel=[C,\hl{[X,Y]}-[\hl{X},\hl{Y}]]=-[C,[\hl{X},\hl{Y}]]=[\hl{X},[\hl{Y},C]]+[\hl{Y},[C,\hl{X}]]=0$,
\end{center}
taking into account the homogeneity of the associated Ehresmann connection.

\textit{(\ref{Yo})}$\Rightarrow$\textit{(\ref{Pe})} We operate by $\vl{X}$ on both sides of the relation $\overline{\mu}(\KK(\kal{Y}),\kal{Z})=\overline{\mu}(\kal{Y},\KK(\kal{Z}))$, and permute the variables cyclically. Then we obtain:
\begin{center}
$\displaystyle\vl{X}(\overline{\mu}(\KK(\kal{Y}),\kal{Z}))=\vl{X}(\overline{\mu}(\kal{Y},\KK(\kal{Z})))$,\\
$\vl{Y}(\overline{\mu}(\KK(\kal{Z}),\kal{X}))=\vl{Y}(\overline{\mu}(\kal{Z},\KK(\kal{X})))$,\\
$\vl{Z}(\overline{\mu}(\KK(\kal{X}),\kal{Y}))=\vl{Z}(\overline{\mu}(\kal{X},\KK(\kal{Y})))$.
\end{center}
Applying the product rule,
\begin{center}
$\displaystyle\vl{\nabla}\overline{\mu}(\kal{X},\KK(\kal{Y}),\kal{Z})-\vl{\nabla}\overline{\mu}(\kal{X},\kal{Y},\KK(\kal{Z}))=\overline{\mu}(\kal{Y},\vl{\nabla}\KK(\kal{X},\kal{Z}))-\overline{\mu}(\vl{\nabla}\KK(\kal{X},\kal{Y}),\kal{Z})$,\\
$\vl{\nabla}\overline{\mu}(\kal{Y},\KK(\kal{Z}),\kal{X})-\vl{\nabla}\overline{\mu}(\kal{Y},\kal{Z},\KK(\kal{X}))=\overline{\mu}(\kal{Z},\vl{\nabla}\KK(\kal{Y},\kal{X}))-\overline{\mu}(\vl{\nabla}\KK(\kal{Y},\kal{Z}),\kal{X})$,\\
$\vl{\nabla}\overline{\mu}(\kal{Z},\KK(\kal{X}),\kal{Y})-\vl{\nabla}\overline{\mu}(\kal{Z},\kal{X},\KK(\kal{Y}))=\overline{\mu}(\kal{X},\vl{\nabla}\KK(\kal{Z},\kal{Y}))-\overline{\mu}(\vl{\nabla}\KK(\kal{Z},\kal{X}),\kal{Y})$.
\end{center}
Now we add these three relations. Since $\vl{\nabla}\overline{\mu}=\vl{\nabla}\vl{\nabla}\vl{\nabla}\Ffel$ is totally symmetric, we obtain
\begin{center}
$\displaystyle 0=\overline{\mu}(\vl{\nabla}\KK(\kal{Y},\kal{X})-\vl{\nabla}\KK(\kal{X},\kal{Y}),\kal{Z})+$\\
$\displaystyle\overline{\mu}(\vl{\nabla}\KK(\kal{Z},\kal{Y})-\vl{\nabla}\KK(\kal{Y},\kal{Z}),\kal{X})+$\\
$\displaystyle\overline{\mu}(\vl{\nabla}\KK(\kal{X},\kal{Z})-\vl{\nabla}\KK(\kal{Z},\kal{X}),\kal{Y})\overset{\textrm{(\ref{propo46})}}{=}$\\
$\displaystyle3(\overline{\mu}(\RR(\kal{X},\kal{Y}),\kal{Z})+\overline{\mu}(\RR(\kal{Y},\kal{Z}),\kal{X})+\overline{\mu}(\RR(\kal{Z},\kal{X}),\kal{Y}))=$\\
$\displaystyle3\underset{(\kal{X},\kal{Y},\kal{Z})}{\mathfrak{S}}\overline{\mu}(\RR(\kal{X},\kal{Y}),\kal{Z})$,
\end{center}
and this ends the proof.
\end{proof}

\begin{coro}\textup{(the self-adjointness condition).}
If a Finsler function\\ $\Ffel: TM\rightarrow\valR$ satisfies a Rapcsák equation with respect to a spray over $M$, then the Jacobi endomorphism $\KK$ determined by the spray is self-adjoint with respect to the symmetric type $\binom{0}{2}$ tensor $\overline{\mu}=\vl{\nabla}\vl{\nabla}\Ffel$, i.e., $$\overline{\mu}(\KK(\hul{X}),\hul{Y})=\overline{\mu}(\hul{X},\KK(\hul{Y}))\textrm{  ;  }\hul{X},\hul{Y}\in\szel{\pik}.$$
\end{coro}
\mbox{}\hfill\tiny$\square$\normalsize \bigskip

\chapter{Summary}
In the following we present a brief survey of chapter contents.\\

\textbf{Chapter 1} In this chapter we give the necessary preliminaries. We collect the most indispensable concepts and facts from basic differential geometry, and standardize our notation and terminology. We fix the main scene of our considerations: this is the \textit{Finsler bundle} $$\pik: \Tk M\times_MTM\rightarrow\Tk M,$$ the pull-back of the tangent bundle $\tau: TM\rightarrow M$ over the projection of the slit tangent bundle $\tauk: \Tk M\rightarrow M$. We also need the vector bundle $$\pi: TM\times_MTM\rightarrow TM,$$ the pull-back of $\tau$ over $\tau$. The modules of sections of these vector bundles will be denoted by $\szel{\pik}$ and $\szel{\pi}$, respectively. We introduce a canonical tensor derivation, the \textit{vertical derivation} $\vl{\nabla}$, over the tensor algebra of the $\modc{\Tk M}$-module $\szel{\pik}$.

The only new technicality is the \textit{inductively defined trace operator} acting on type $\binom{1}{s+1}$ tensor fields along $\tauk$. This will be proved to be effective and useful in our coordinate-free calculations. For completeness, we reproduce a simple proof of the \textit{differential Bianchi identity} in the context of general vector bundles.\\

\textbf{Chapter 2} Here we fix what we mean by an \textit{Ehresmann connection} and a \textit{spray}. We also introduce some mutations of a spray: semispray, second-order vector field, affine spray. All this is necessary since we find different and non-equivalent definitions for these basic concepts in the literature. We recall the fundamental relation between an Ehresmann connection and a semispray, discovered (independently) by M. Crampin and J. Grifone. We define the most important technical tool of our calculations, the \textit{Berwald derivative} $\nabla$. It is built of a \textit{horizontal part} $\hl{\nabla}$ determined by an Ehresmann connection, and the vertical derivative $\vl{\nabla}$.

We derive a \textit{horizontal Ricci identity} for functions, in which the curvature tensor $\mathbf{R}$ of the Ehresmann connection appears. We prove the following \textit{Bianchi identity} for the horizontal differential of the curvature: $$\underset{(X,Y,Z)}{\mathfrak{S}}(\hl{\nabla}\RR)(\kal{X},\kal{Y},\kal{Z})=0.$$ To our knowledge, this simple and useful relation has not appeared in the literature (at least in this form). It corresponds the Bianchi identity $$[h,R]=0$$ in Proposition I.61 in Grifone's paper \cite{Grif}, where the symbol $[,]$ means Frölicher-Nijenhuis bracket, whose evaluation is quite difficult. A similar Bianchi identity was obtained also by M. Crampin \cite{Cramp3}, but in a quite artifical manner.\\

\textbf{Chapter 3} This chapter is devoted to a brief discussion of the \textit{Berwald curvature} of an Ehresmann connection $\HH$. Consider the usual curvature operator $$R^{\nabla}(\xi,\eta): \hul{Z}\in\szel{\pik}\mapsto R^{\nabla}(\xi,\eta)\hul{Z}:=\nabla_{\xi}\nabla_{\eta}\hul{Z}-\nabla_{\eta}\nabla_{\xi}\hul{Z}-\nabla_{[\xi,\eta]}\hul{Z}\in\szel{\pik}$$ of the Berwald derivative $\nabla$ ($\xi$ and $\eta$ are fixed vector fields on $\Tk M$). Then the Berwald curvature $\BB$ of $\HH$ is defined by $$\mathbf{B}(\hul{X},\hul{Y}):=R^{\nabla}(\ii\hul{X},\HH\hul{Y})\textrm{ ; }\hul{X},\hul{Y}\in\szel{\pik}.$$
(The $\modc{\Tk M}$-linear map $\ii$ identifies the module $\szel{\pik}$ with the module of vertical vector fields on $\Tk M$.)

Beside some technicalities (convenient formulae for calculations of $\mathbf{B}$, symmetry and homogeneity properties, Ricci identities involving $\BB$), we show that \textit{the Berwald curvature vanishes, if and only if, the horizontal derivative arising from the connection is ``h-basic"}, i.e., roughly speaking, it is the natural lift of a covariant derivative operator on the base manifold. More precisely, $\BB$ vanishes, if and only if, there is a covariant derivative operator $D$ on $M$, such that $$\hl{\nabla}_{\kal{X}}\kal{Y}=\kal{D_XY}\textrm{ ; }X,Y\in\modx{M};$$
$$\kal{X}(v):=(v,X(\tau(v)))\textrm{ , }v\in\Tk M.$$

\textbf{Chapter 4} In this chapter we discuss the \textit{affine curvature} $\Hh$ of an Ehresmann connection, with specific emphasis on the case when the Ehresmann connection is generated by a spray. By definition, $$\mathbf{H}(\hul{X},\hul{Y}):=R^{\nabla}(\HH\hul{X},\HH\hul{Y})\textrm{ ; }\hul{X},\hul{Y}\in\szel{\pik}.$$
Our terminology (`affine curvature') follows Berwald's usage \cite{Berwald}. If $\vl{\nabla}\mathbf{H}=0$, we say after Z. Shen that the Ehresmann connection is \textit{R-quadratic}.

We derive between the affine curvature $\mathbf{H}$ and the curvature $\mathbf{R}$ of $\HH$ the following relations: $$\mathbf{H}(\hul{X},\hul{Y})\hul{Z}=\vl{\nabla}\mathbf{R}(\hul{Z},\hul{X},\hul{Y});$$ $$\RR(\hul{X},\hul{Y})=\Hh(\hul{X},\hul{Y})\delta\textrm{ }\textrm{\textit{ , if }}\HH\textrm{\textit{ is homogeneous.}}$$
($\delta: v\in TM\mapsto\delta(v):=(v,v)$ is the canonical section of $\pi$.)

Also in the homogeneous case, we show that $\RR$ is homogeneous of degree 1, and $\Hh$ is homogeneous of degree 0.

\begin{center}
\textit{We assume now that the Ehresmann connection} $\HH$ \textit{is torsion-free.}
\end{center}

We deduce\\

\textit{the algebraic Bianchi identity} $\underset{(X,Y,Z)}{\mathfrak{S}}\Hh(\kal{X},\kal{Y})\kal{Z}=0$,\\

and the \textit{differential Bianchi identity} $$\vl{\nabla}\mathbf{H}(\hul{X},\hul{Y},\hul{Z},\hul{U})-\hl{\nabla}\mathbf{B}(\hul{Y},\hul{X},\hul{Z},\hul{U})+\hl{\nabla}\mathbf{B}(\hul{Z},\hul{X},\hul{Y},\hul{U})=0.$$

As further technicalities, we derive the \textit{Ricci formulae} for the repeated horizontal differential of sections and 1-forms; they involve the affine curvature.

After these, we define and derive in an index-free manner the basic relations which served, in the language of tensor calculus, as the definitions of the basic curvature data in Berwald's classical paper \cite{Berwald}. Let a spray $S$ over $M$ be given. (In Berwald's treatment the role of $S$ is played by a system of second-order differential equations written in terms of local coordinates.) The \textit{affine deviation tensor} (Berwald's terminology) or the \textit{Jacobi endomorphism} of $S$ is the type $\binom{1}{1}$ tensor field $\mathbf{K}$ along $\tauk$ given by $$\KK(\hul{X}):=\VV[S,\HH\hul{X}]\textrm{ , }\hul{X}\in\szel{\pik} ,$$
where $\HH$ is the Ehresmann connection associated to $S$, and $\VV$ is the vertical map belonging to $\HH$ ($\VV\circ\HH=0$, $\VV\circ\ii=\textrm{\textit{identity}}$). We show in our formalism that the curvature of $\HH$ and the affine deviation tensor are related by $$\RR(\hul{X},\hul{Y})=\frac{1}{3}(\vl{\nabla}\KK(\hul{Y},\hul{X})-\vl{\nabla}\KK(\hul{X},\hul{Y}))\textrm{ ; }\hul{X},\hul{Y}\in\szel{\pik}.$$

We conclude this chapter with a brief discussion of the \textit{flatness} and the \textit{isotropy} of a spray. In both cases, by definition, the Jacobi endomorphism has a very specific form. It turns out immediately that flatness implies the vanishing of the Jacobi endomorphism, whence the curvature and the affine curvature also vanish. Isotropic sprays will be studied in some detail in the Finslerian case.\\

\textbf{Chapter 5} Two sprays, $S$ and $\overline{S}$, over a smooth manifold $M$ are said to be \textit{projectively related} if $$\overline{S}=S-2PC,$$ where the \textit{projective factor} $P$ is a positive-homogeneous function of degree 1 (smooth on $\Tk M$), and $C:=\ii\circ\delta$ is the Liouville vector field. The transition from $S$ to $\overline{S}$ is mentioned as a \textit{projective change}.

In this chapter first we review some basic facts concerning a projective change of a spray. Then all of the basic geometric data (Ehresmann connection and its associated objects, horizontal derivative, Berwald curvature, Jacobi endomorphism,...) of the spray change; we give the explicit formulas for these changes. We show that \textit{the Berwald curvature and its trace remain invariant under a projective change at the same time}. The criterion of their invariance leads to a simple PDE for the projective factor, which we solve without using coordinates.

We recall an intrinsic definition of the two basic projectively invariant tensors, the \textit{Douglas curvature} ($\mathbf{D}$), which may be constructed from the Berwald curvature, and the \textit{Weyl endomorphism} ($\mathbf{W}^{\circ}$), which may be built from the Jacobi endomorphism. As for the Weyl endomorphism (or \textit{projective deviation tensor} in Berwald's usage), we adopted del Castillo's definition \cite{delc}, mutatis mutandis, but we expressed it in a more convenient form in terms of $\KK$, $\textrm{tr}\KK$ and their vertical differentials.\\

\textbf{Chapter 6} We begin with the definition of a \textit{Finsler function} and its fundamental geometric data (Hilbert 1-form, normalized supporting element field, angular metric tensor, Cartan tensor, Landsberg tensor). We present some simple, more or less technical, observations about these basic objects. Next we recall an intrinsic definition of the \textit{canonical spray} of a Finsler manifold. The construction is just a fine intrinsic reformulation of the Euler-Lagrange equation of the energy functional. From this point, our general principles may be realized according to the scheme
\begin{center}
\textit{Finsler function} $\longrightarrow$ \textit{canonical spray} $\longrightarrow$ \textit{Ehresmann connection} $\longrightarrow$ \textit{curvatures}.
\end{center}
Note that the Ehresmann connection determined by the canonical spray of a Finsler manifold is said to be the \textit{canonical connection} or \textit{Berwald connection} of the Finsler manifold. From this connection, as in the general theory, a covariant derivative operator can be obtained by linearization in the Finsler bundle $\pik: \Tk M\times_MTM\rightarrow\Tk M$, this is the \textit{(Finslerian) Berwald derivative}. (It is dangerous to confuse the Berwald connection with the Berwald derivative!)

The only truly interesting result in this chapter is essentially classical. In his paper \cite{Berwald} Berwald has shown that an at least 3-dimensional isotropic Finsler manifold has vanishing Weyl endomorphism. (His formulation is distinct to some extent, but equivalent.) It was discovered by L. del Castillo and, independently, by Z. I. Szabó, that the converse of Berwald's theorem is also true. We give here a simple proof of this important observation. (Berwald himself also proved the converse, but he used an additional condition.) For completeness, we also present an independent proof of Berwald's above mentioned statement; in fact, this is the harder part. Note that in Berwald's and Szabó's formulation it is assumed that the Finsler manifold is at least 3-dimensional. In our treatment this condition is superfluous. However, we shall discuss the 2-dimensional case repeatedly in Chapter 9. Then we shall check that the Weyl tensor is automatically zero (which is a well-known fact), while the canonical spray is isotropic (this will be obtained as an easy consequence).\\

\textbf{Chapter \ref{ch7}} Finsler geometric objects are typically position and direction dependent. It may happen, however, that some of them depend only on the position. Mathematically expressed: some Finsler geometric objects may have vanishing vertical differential. We mention here an important, classical example. In an $n$-dimensional, isotropic Finsler manifold $(M,F)$ may be defined by the scalar curvature function $$R:=\frac{1}{(n-1)F^2}\textrm{tr}\KK,$$ where $\KK$ is the Jacobi endomorphism. It is positive-homogeneous of degree 0. Berwald has shown in \cite{Berwald} that if $R$ ``depends only on the position", i.e., $\vl{\nabla}R=0$, and $\textrm{dim}M\geq 3$, then the function $R$ is constant. (It is presupposed that the manifold is connected.) This is the Finslerian version of the well-known \textit{Schur lemma} from Riemannian geometry.

A systematic investigation of Finsler manifolds with direction-independent data was initiated by S. Bácsó and M. Matsumoto \cite{Bacsmat}. In this chapter we show that the direction independence of the Landsberg tensor and the stretch tensor holds only trivially, i.e., if these tensors vanish. We also prove that R-quadratic Finsler manifolds have vanishing stretch tensor. To formulate these results more explicitly, consider\\
the \textit{metric tensor} $\textrm{  }\textrm{  }g:=\frac{1}{2}\vl{\nabla}\vl{\nabla}F^2$,\\
the \textit{Landsberg tensor} $\textrm{  }\mathbf{P}:=-\frac{1}{2}\hl{\nabla}g$,\\
and the \textit{stretch tensor} $\mathbf{\Sigma}$ defined by
$$\Sigma(\hul{X},\hul{Y},\hul{Z},\hul{U}):=2(\hl{\nabla}\Pp(\hul{X},\hul{Y},\hul{Z},\hul{U})-\hl{\nabla}\Pp(\hul{Y},\hul{X},\hul{Z},\hul{U})).$$ Then we have
\begin{itemize}
\item[(1)] $\vl{\nabla}\mathbf{P}=0\textrm{  }\Rightarrow\textrm{  }\mathbf{P}=0$;
\item[(2)] $\vl{\nabla}\mathbf{\Sigma}=0\textrm{  }\Rightarrow\textrm{  }\mathbf{\Sigma}=0$;
\item[(3)] $\vl{\nabla}\mathbf{H}=0\textrm{  }\Rightarrow\textrm{  }\mathbf{\Sigma}=0$.
\end{itemize}

\textbf{Chapter \ref{sec5}} Let $(M,F)$ be a Finsler manifold with metric tensor $g$. First we define the orthogonal projection of the module of sections of the Finsler bundle $\pik: \Tk M\times_MTM\rightarrow\Tk M$ onto the $g$-orthogonal complement of $\textrm{span}(\delta)$ ($\delta$ is the canonical section). On Euclidean analogy, it may simply be given by $$\hul{X}\mapsto\mathbf{p}(\hul{X}):=\hul{X}-\frac{g(\hul{X},\delta)}{g(\hul{X},\hul{X})}\delta.$$ In a more compact form, $$\mathbf{p}=\mathbf{1}-\frac{1}{F}\vl{\nabla}F\otimes\delta.$$ We also define, what we mean by the projected tensor of a type $\binom{0}{k}$ or a type $\binom{1}{k}$ ``Finsler tensor" ($k\geq 1$).

Temporarily, we say that a Finsler manifold is a \textit{p-Berwald manifold}, if the projected tensor of its Berwald curvature vanishes. Our first observation is that \textit{a p-Berwald manifold is R-quadratic, if and only if, its stretch tensor vanishes}. Next we show that \textit{the class of the at least 3-dimensional p-Berwald manifolds is the same as the class of the at least 3-dimensional Berwald manifolds}. Thus we obtain a new characterization of Berwald manifolds in dimension $n\geq 3$. This result is strongly related to Sakaguchi's important theorem in \cite{Sakaguchi}, which states that an at least 3-dimensional Finsler manifold is a Douglas manifold (i.e., has vanishing Douglas curvature), if and only if, its projected Douglas curvature vanishes. Sakaguchi's theorem plays an essential role in our proof.

Having the projection operator $\mathbf{p}$, we may express the curvature of the Berwald connection of an \textit{isotropic} Finsler manifold $(M,F)$ in the very convenient form $$\mathbf{R}=F\mathbf{p}\wedge(R\vl{\nabla}F+\frac{1}{3}F\vl{\nabla}R),$$ where $R$ is the scalar curvature mentioned above. Conversely, if the curvature $\RR$ takes this form, then $(M,F)$ is isotropic. If, in addition, $R$ `depends only on the position', then we obtain $$\RR=FR(\mathbf{p}\otimes\vl{\nabla}F-\vl{\nabla}F\otimes\mathbf{p)}.$$ Starting from these observations, to demonstrate the efficiency of our tools, we conclude the Chapter with a new proof of the Finslerian Schur lemma.\\

\textbf{Chapter \ref{ch9}} The greater part of this chapter consists essentially of transcriptions in order to give an intrinsic formulation in our setup of Berwald's theory of \textit{2-dimensional Finsler manifolds}, explained by him so beautifully in terms of the classical tensor calculus in \cite{Berwald3}. In this process all ingredients of the preceding chapters appear once again, but in a more transparent form. This transparency is mostly due to the fact that we have an intrinsically constructed orthonormal 2-frame, called \textit{Berwald frame}, and we may apply Fourier expansion with respect to this frame. So, on the one hand, this chapter may be considered as an application of our tools and techniques to a concrete situation. On the other hand, we find an opportunity to tie up some loose ends.

We give an explicit representation of the Jacobi endomorphism, and conclude that all 2-dimensional Finsler manifolds are isotropic. On the other hand, we can easily show that the Weyl endomorphism annulates both members of the Berwald frame, and hence it is the zero transformation. We show that \textit{a 2-dimensional Finsler manifold is p-Berwald, if and only if, it is weakly Berwald, i.e., its Berwald curvature is traceless}. We conclude, finally, that \textit{a 2-dimensional Finsler manifold is a Berwald manifold, if and only if, it is weakly Berwald and has vanishing Landsberg tensor}.\\

\textbf{Chapter 10} Given a spray over a manifold $M$, we may ask:

\textit{When does a Finsler function exist such that its canonical spray} is \textit{the given spray? When does a Finsler function exist such that its canonical spray is} projectively related\textit{ to the given spray?}

The first question is the problem of \textit{Finsler metrizability} or \textit{Finsler-variationality}, the second one is the problem of \textit{Finsler metrizablity in a broad sense} or, briefly, the problem of \textit{projective metrizability}. In terms of the classical tensor calculus, A. Rapcsák has formulated two equivalent criteria for the projective relatedness of the canonical sprays of two Finsler functions $F$ and $\overline{F}$ over the same manifold $M$. These criteria are mentioned as \textit{Rapcsák equations} nowadays. In Rapcsák equations we find the partial derivatives of $\overline{F}$ and the spray coefficients of the canonical spray of $(M,F)$, or the Christoffel symbols of the Berwald connection of $(M,F)$. So it makes sense to speak of a \textit{Rapcsák equation for a Finsler function with respect to a spray}. In what follows, we use the term in this sense. Then, obviously, \textit{Rapcsák equations give a key to attack the problem of projective metrizability}.

In the first essential step of this chapter we formulate one of the Rapcsák equations in an intrinsic (first index-free, next index and argumentum-free) manner. Using these new forms, \textit{we derive a simple necessary and sufficient condition for Finsler variationality}. Applying this criterion, we obtain an \textit{extremely simple proof for the unicity of the canonical connection of a Finsler manifold}.

The rest of the chapter is devoted to necessary conditions for projective metrizability of a spray. The most interesting among them (with the most difficult proof) is the following:\\

\textit{If a Finsler function} $\overline{F}: TM\rightarrow\valR$ \textit{satisfies a Rapcsák equation with respect to a spray over} $M$\textit{, then the Jacobi endomorphism} $\KK$ \textit{determined by the spray is ``self-adjoint" with respect to the symmetric type }$\binom{0}{2}$ \textit{tensor} $\overline{\mu}:=\vl{\nabla}\vl{\nabla}\overline{F}$\textit{, i.e., for any sections }$\hul{X}$, $\hul{Y}$ \textit{along} $\tauk$ \textit{we have} $$\overline{\mu}(\KK(\hul{X}),\hul{Y})=\overline{\mu}(\hul{X},\KK(\hul{Y})).$$

\subsection*{Előadások}

\begin{itemize}
\item[(1)] Térgeometriai problémák megoldása a ciklografikus leképezés magasabb dimenziós általánosításának
alkalmazásával, 2006. május 5., Budapest, Országos ábrázoló geometria konferencia
\item[(2)] P-Berwald sokaságok, 2008. december 5., Debrecen, Geometria tanszéki szeminárium
\end{itemize}

\footnotesize
\textsc{Zoltán Szilasi\\Institute of Mathematics\\University of Debrecen\\H-4010 Debrecen\\Hungary}\\
\normalsize
\textit{E-mail}: szilasi.zoltan@inf.unideb.hu

\end{document}